%% file: Shearfree_congruences_Robinson_arXiv_v4.tex
\documentclass [11pt] {article}
\usepackage{amsfonts}
\usepackage{amsmath}
\usepackage{amssymb}
\usepackage{nccmath}
\usepackage{stmaryrd}
\usepackage{hyperref}
\usepackage{bm}
\usepackage{upgreek}
\usepackage{mathtools}
\usepackage{fullpage,theorem}
\usepackage{scalerel}
\usepackage{multirow}
\usepackage[shortlabels]{enumitem}
\usepackage{tensor}
\usepackage{caption}
\usepackage{accents}
\usepackage[polutonikogreek,english]{babel}
\usepackage{xcolor}
\usepackage{color, colortbl}
\usepackage[first=0,last=9]{lcg}

\usepackage{etex}
\usepackage[all]{xy}

\definecolor{Gray}{gray}{0.9}
\newcolumntype{g}{>{\columncolor{Gray}}c}

\newcommand{\ut} [1]{\undertilde{#1}}
\newcommand{\NRho} {\Nh\mkern-5mu\Rho}
\renewcommand{\bigwedge}{\scaleobj{1.2}{\wedge}}

\include{macro}

\providecommand{\keywords}[1]{{\small \textbf{Keywords:} #1}}

\providecommand{\subjclass}[1]{{\small \textbf{Mathematics Subject Classification (2010):} #1}}

\title{Twisting non-shearing congruences of null geodesics,\\
almost CR structures, and Einstein metrics in even dimensions
}
 \author{Arman Taghavi-Chabert\thanks{Department of Mathematics, Faculty of Arts and Sciences, American University of Beirut, P.O. Box 11-0236, Riad El Solh, Beirut 1107 2020, Lebanon,
 		\textit{E-mail address:} \texttt{at68@aub.edu.lb}}}
\date{}

\begin{document}
\maketitle

\begin{abstract}
We investigate the geometry of a twisting non-shearing congruence of null geodesics on a conformal manifold of even dimension greater than four and Lorentzian signature. We give a necessary and sufficient condition on the Weyl tensor for the twist to induce an almost Robinson structure, that is, the screen bundle of the congruence is equipped with a bundle complex structure. In this case, the (local) leaf space of the congruence acquires a partially integrable contact almost CR structure of positive definite signature. We give further curvature conditions for the integrability of the almost Robinson structure and the almost CR structure, and for the flatness of the latter.

We show that under a mild natural assumption on the Weyl tensor, any metric in the conformal class that is a solution to the Einstein field equations determines an almost CR--Einstein structure on the leaf space of the congruence. These metrics depend on three parameters, and include the Fefferman--Einstein metric and Taub--NUT--(A)dS metric in the integrable case. In the non-integrable case, we obtain new solutions to the Einstein field equations, which, we show, can be constructed from strictly  almost K\"{a}hler--Einstein manifolds.
\end{abstract}
\\


\section{Introduction}\label{sec:intro}
A \emph{non-shearing congruence of null geodesics} on a Lorentzian manifold $(\mc{M}, g)$ is a (local) foliation $\mc{K}$ by geodesics generated by a null vector field $k$, i.e.\ $g(k,k)=0$, such that the metric
$g$ is preserved along the flow of $k$ when restricted to vectors orthogonal to $k$, i.e.\
\begin{align*}
\mathsterling_k g (v,w) & \propto g (v ,w) \, , & \mbox{for any vector fields $v, w$ such that $g(k,v)= g(k,w)=0$.}
\end{align*}
In dimension four, non-shearing congruences of null geodesics are central objects of mathematical relativity, and their existence is intimately connected with solutions to the vacuum Maxwell equation and solutions to the Einstein field equations according to the Robinson theorem \cite{Robinson1961} and the Goldberg-Sachs theorem \cite{Goldberg1962,Goldberg2009} respectively. The latter asserts that any Einstein spacetime admits a non-shearing congruence of null geodesics if and only if its Weyl tensor is algebraically special. This includes many well-known solutions such as the Kerr black hole \cite{Kerr1963}, Robinson--Trautman spacetimes  \cite{Robinson1961/62} and Kundt spacetimes \cite{Kundt1961}.

Another feature of these congruences in dimension four is that they are equivalent to the existence of an involutive totally null complex $2$-plane distribution $N$. What is more, $N$ induces a \emph{Cauchy--Riemann (CR)} structure on the three-dimensional (local) leaf space $\ul{\mc{M}}$ of $\mc{K}$, that is, $\ul{\mc{M}}$ is endowed with a rank-$2$ distribution together with a bundle complex structure \cite{Trautman1984,Trautman1985,Robinson1985,Penrose1986,Robinson1986,Robinson1989,Trautman1999}. This is particularly relevant to the study of solutions to the Einstein field equations, which, as beautifully demonstrated in \cite{Lewandowski1990}, can then be reduced to CR data on $\ul{\mc{M}}$ when the congruence is twisting i.e.\ $\kappa \wedge \dd \kappa \neq 0$ where $\kappa = g(k,\cdot)$.

In dimensions greater than four, such congruences have not been as prominent in Einstein spacetimes as they have in dimension four. The notable exceptions are the higher-dimensional generalisations of Robinson--Trautman and Kundt spacetimes, where the congruence is non-twisting, i.e.\ $\kappa \wedge \dd \kappa = 0$ with $\kappa = g(k,\cdot)$. In odd dimensions, it has been shown \cite{Ortaggio2007} that if the Weyl tensor satisfies
\begin{align}\label{eq:aligned}
W  (k, v , k, v) & = 0 \, , & \mbox{for any vector field $v$ such that $g(k,v)=0$,}
\end{align}
then the non-shearing congruence of null geodesics generated by $k$ is necessarily non-twisting. The even-dimensional case, however, has not been thoroughly investigated. The only known examples are Taub--NUT-type metrics of \cite{Bais1985,Awad2002}, as pointed out in \cite{Ortaggio2013}, and Einstein metrics on Fefferman spaces of CR manifolds, which were described in \cite{Leitner2007,Cap2008}, and where the congruence is generated by a null conformal Killing field.

The aim of the present article is to fill the gap in that respect and provide a detailed understanding of twisting non-shearing congruences of null geodesics in dimension $2m+2$ where $m>1$. For this purpose, we shall adopt the strategy and philosophy of \cite{Robinson1986,Penrose1986} and subsequent work \cite{Lewandowski1988,Lewandowski1991,nurowski93-phd,nurowski96,Nurowski1997}. For more analytical issues, see also \cite{Tafel1985,Lewandowski1990a,Hill2008,Schmalz2019}. In particular, we shall emphasise the conformally invariant aspect of these congruences: a congruence $\mc{K}$ of null curves can be expressed in terms of an \emph{optical geometry}   $(\mc{M},\mbf{c}, K)$, where $(\mc{M}, \mbf{c})$ is a conformal manifold and $K$ the null line distribution tangent to $\mc{K}$. The screen bundle $H_K := K^\perp / K$ inherits a bundle conformal structure $\mbf{c}_{H_K}$ from $\mbf{c}$. The congruence $\mc{K}$ is then geodesic if $K^\perp$ is preserved along $\mc{K}$, and in addition, non-shearing if $\mbf{c}_{H_K}$ is also preserved along $\mc{K}$ \cite{Robinson1983,Fino2020}.

On the other hand, as advocated by \cite{Nurowski2002,Trautman2002,Taghavi-Chabert2014,Fino2021},  one can start with an \emph{almost Robinson manifold} $(\mc{M}, \mbf{c}, N, K)$, that is, a Lorentzian conformal manifold of dimension $2m+2$ equipped with a totally null complex $(m+1)$-plane distribution $N$: it defines an optical structure $K$, and thus a congruence $\mc{K}$ of null curves, together with a bundle complex structure on the screen bundle $H_K$ compatible with $\mbf{c}_{H_K}$. Under suitable conditions, $(N, K)$ induces an almost CR structure on the (local) leaf space of $\mc{K}$, which is integrable if and only if $(N,K)$ is integrable. But, unlike in dimension four, $\mc{K}$ is shearing in general \cite{Trautman2002a,Mason2010} except in rare constructions such as the Fefferman conformal structure \cite{Fefferman1976,Fefferman1976a,Lee1986,Graham1987} and Taub--NUT-type metrics \cite{Bais1985,Awad2002,Alekseevsky2021}.

The converse problem seems a priori quite hopeless: when does an optical geometry single out an almost Robinson structure in even dimensions greater than four? As this article will reveal, there is a definite, and surprisingly natural, answer to this question provided the congruence is geodesic, twisting and non-shearing. The ramification into almost CR geometry, if somewhat simpler, will then prove as powerful as in dimension four, but will offer a new feature: the underlying almost CR structure may be \emph{non-integrable}.

The structure of the paper and its main results are as follows. The background material on conformal geometry and optical geometries is given in Sections \ref{sec:conf_geom} and \ref{sec:opt_geom}, where Theorem \ref{thm-int-cond>4} highlights the integrability condition for the existence of a non-shearing congruence of null geodesics.  Section \ref{sec:CR_geom} provides a fairly detailed account of partially integrable contact almost CR geometry. We extend in particular the definition of a \emph{CR--Einstein structure} to non-integrable almost CR geometry, and relate this concept to almost K\"{a}hler--Einstein metrics, notably in Proposition \ref{prop:aK2aCR}. We then review the notion of almost Robinson manifolds in Section \ref{sec-Robinson}. This leads naturally to the following theorem, which gives an invariant relation between twisting non-shearing congruences and almost CR structures:
\begin{thm}\label{thm-int-cond2m}
	Let $(\mc{M},\mbf{c},K)$ be a $(2m+2)$-dimensional conformal optical geometry, where $m>1$, with twisting non-shearing congruence of null geodesics $\mc{K}$. The following statements are equivalent:
	\begin{enumerate}
		\item The Weyl tensor satisfies
		\begin{align}\label{eq-weakestC}
		W  (k, v , k, v) & = 0 \, , & \mbox{for any sections $k$ of $K$, $v$ of $K^\perp$.}
		\end{align}
		\item The twist of $\mc{K}$ induces an almost Robinson structure $(N,K)$ on $\mc{M}$.
		\item The twist of $\mc{K}$ induces a partially integrable contact almost CR structure $(\ul{H},\ul{J})$ of positive definite signature on the (local) leaf space $\ul{\mc{M}}$ of $\mc{K}$.
	\end{enumerate}
\end{thm}
Under the assumptions of Theorem \ref{thm-int-cond2m}, Section \ref{sec-curv-cond} delves into the relation between the Weyl tensor and the invariants of the almost CR structure -- see  Theorems \ref{thm:main_int}, \ref{thm:main-pre-flatness} and \ref{thm:main-flatness}, where we give conditions on the Weyl tensor leading to the integrability and flatness of the almost CR structure.

In Section \ref{sec:Weyl_tensor}, we include a brief discussion on additional prescriptions on the Weyl curvature as potential generalisations of the notion of algebraically special Weyl tensors from four to higher dimensions. One such candidate is used to prove, in Section \ref{sec:Einstein}, the following theorem:
\begin{thm}\label{thm:main-Einstein}
	Let $(\mc{M},\mbf{c},K)$ be a $(2m+2)$-dimensional conformal optical geometry, where $m>1$,  with twisting non-shearing congruence of null geodesics $\mc{K}$. Suppose that the Weyl tensor satisfies
	\begin{align*}
	W (k, v, k, \cdot) & = 0 \, , &  \mbox{for any sections $k$ of $K$, $v$ of $K^\perp$,}
	\end{align*}
	and that $\mbf{c}$ contains a metric $\wh{g}$ that is Einstein with Ricci scalar $(2m + 2)\Lambda$ on (some open subset of) $\mc{M}$. Then
	\begin{itemize}
		\item the twist of $\mc{K}$ induces an almost Robinson structure $(N,K)$ on $\mc{M}$;
		\item $\mc{M}$ is (locally) diffeomorphic to $\left( -\frac{\pi}{2} , \frac{\pi}{2} \right) \times \ul{\mc{M}}$, where $\ul{\mc{M}}$ is the (local) leaf space of $\mc{K}$, and $(N,K)$ induces a partially integrable contact almost CR structure $(\ul{H}, \ul{J})$ of positive definite signature on $\ul{\mc{M}}$;
		\item there is a distinguished contact form $\ul{\theta}^0$ such that $(\ul{H}, \ul{J},\ul{\theta}^0)$ is almost CR--Einstein -- in particular, the Webster--Ricci scalar is given by $\ul{\Sc} = m \ul{\Lambda} +  \| \ul{\Nh} \|^2_{\ul{h}}$, where $\ul{\Lambda}$ is some constant and $\| \ul{\Nh} \|^2_{\ul{h}}$ is the square of the norm of the Nijenhuis tensor of $(\ul{H}, \ul{J}, \ul{\theta}^0 )$ with respect to the Levi form $\ul{h}$ of $\ul{\theta}^0$;
		\item locally, the metric takes the form
		\begin{align*}
		\wh{g} & = \sec^2 \phi \, g \, , & \mbox{for $- \frac{\pi}{2} < \phi < \frac{\pi}{2}$,} 
		\end{align*}
		where, denoting the natural projection from $\mc{M}$ to $\ul{\mc{M}}$ by $\varpi$,
		\begin{align*}
		g & =2 \, \kappa \odot \lambda + h  \, , \\
		\kappa & = 2 \, \varpi^* \ul{\theta}^0 \, , &
		h & = \varpi^* \ul{h} \, , &
		\lambda & = \dd \phi + \lambda_0 \, \varpi^* \ul{\theta}^0 \, , 
		\end{align*}
		with
		\begin{multline*}
		\lambda_0 = \frac{\ul{\Lambda}}{2m+2}   +   \left( \frac{\Lambda}{2m+1}  - \frac{\ul{\Lambda} }{2m+2} \right)\left(  \sum_{j=0}^{m} a_{j} \cos^{2 j} \phi   -2  a_{m}  \cos^{2m+2} \phi \right) \\
		+ \ul{c} \cos^{2m+1} \phi \sin \phi \, ,
		\end{multline*}
		for some constant $\ul{c}$ and 
		\begin{align*}
		a_{0} = 1  \, , & & a_{j} =  \frac{2m-2j+4}{2m-2j+1 }  a_{j-1} \, , & & j = 1, \ldots, m \, .
		\end{align*}
	\end{itemize}
	
	Further, the following statements are equivalent:
	\begin{enumerate}
		\item The Weyl tensor satisfies
		\begin{align*}
		W  (k, u , v, w) & = 0 \, , & \mbox{for any sections $k$ of $K$, $u,v,w$ of $N$;}
		\end{align*}
		\item $(N,K)$ is integrable;
		\item $(\ul{H},\ul{J})$ is integrable.
	\end{enumerate}
\end{thm}
The theorem above thus tells us that any Einstein metric that admits a twisting non-shearing congruence of null geodesics with curvature prescription \eqref{eq-weakestC} depends on the three parameters $\Lambda$, $\ul{\Lambda}$ and $\ul{c}$. This should be contrasted with the situation in dimension four, where the range of solutions is much larger. Further results under weaker assumptions on the Ricci tensor are given in Theorem \ref{thm:Ricci-deg}. Using the results of Section \ref{sec:CR_geom}, Theorem \ref{thm:main-Einstein} also provides a way of constructing examples of Einstein almost Robinson manifolds as lifts of almost CR--Einstein structures over almost K\"{a}hler--Einstein manifolds.

In Section \ref{sec:spec}, we relate the Einstein metrics of Theorem \ref{thm:main-Einstein} for certain values of the parameters $\Lambda$, $\ul{\Lambda}$ and $\ul{c}$ to Fefferman--Einstein metrics in Section \ref{sec:Feff-Einstein} and Taub--NUT-type metrics in Section \ref{sec:Taub-NUT}. These metrics are well-known when the almost Robinson structure is integrable, but to the author's knowledge, the solutions in the non-integrable case are new, and have no analogues in dimension four.

Section \ref{sec:further_prop} contains results on additional geometric structures on the Einstein manifold of Theorem \ref{thm:main-Einstein}, notably on the existence of a distinguished conformal Killing field in Proposition \ref{prop:Killing} and on the properties of a dual almost Robinson structure in Proposition \ref{prop:aRob2}.

Finally, we briefly discuss the possible generalisations to different metric signatures in Section \ref{sec:diff_sign}. We have relegated the computation of the curvature tensors to Appendix \ref{app:computations}.

\paragraph{Acknowledgements}
	The author would like to thank Rod Gover for useful conversations. Parts of the results in this article were presented at the workshop ``Twistors and Loops Meeting in Marseille" that took place in September 2019, at CIRM, Marseille, France. Both the present article and the recent paper \cite{Alekseevsky2021}, which overlap regarding some aspects and content of this topic, were written independently and simultaneously.

The author declares that this work was partially supported by the grant 346300 for IMPAN from the Simons Foundation and the matching 2015-2019 Polish MNiSW fund. He was also supported by a long-term faculty development grant from the American University of Beirut for his visit to IMPAN, Warsaw, in the summer 2018, where this research was partly conducted.

This is a post-peer-review, pre-copyedit version of an article published in \emph{Annali di Matematica Pura ed Applicata (1923 -)}. The final authenticated version is available online at:\\
\url{http://dx.doi.org/10.1007/s10231-021-01133-2}.

\section{Conformal geometry}\label{sec:conf_geom}
Let $(\mc{M}, \mbf{c})$ be an oriented and time-oriented conformal smooth manifold of Lorentzian signature, i.e.\ $(+,\ldots,+,-)$, and of dimension $n+2$. Most of our notation and conventions will be relatively standard in the field of differential geometry. Pullback and pushforward maps will be adorned with an upper and lower $*$ respectively, the Lie derivative along a vector field $v$ will be denoted $\mathsterling_v$ and so on. The $k$-th exterior power of the cotangent bundle $T^* \mc{M}$ will be denoted by $\bigwedge^k T^* \mc{M}$, and its $k$-th symmetric power by $\bigodot^k T^* \mc{M}$.  For any two $1$-forms $\alpha$ and $\beta$, we take the convention that $\alpha \wedge \beta = \frac{1}{2}(\alpha \otimes \beta - \beta \otimes \alpha )$  and $\alpha \odot \beta = \frac{1}{2}(\alpha \otimes \beta + \beta \otimes \alpha )$. We shall also denote the space of sections of a vector bundle $E$ by $\Gamma(E)$. We write ${}^\C E$ for the complexification of $E$. For any subbundle $F$ of $E$, $\Ann(F)$ will denote the subbundle of $E^*$ consisting of elements annihilating sections of $F$. In abstract index notation, sections of $T \mc{M}$, respectively, $T^* \mc{M}$ will be adorned with upper, respectively, lower minuscule Roman indices starting from the beginning of the alphabet, e.g. \ $v^a \in \Gamma(T \mc{M})$, and $\xi_{a b} \in  \Gamma(\bigotimes^2 T^* \mc{M})$. Symmetrisation will be denoted by round brackets, and skew-symmetrisation by square brackets, e.g.\ $\lambda_{(a b)} = \frac{1}{2} \left( \lambda_{a b} + \lambda_{b a} \right)$ and $\lambda_{[a b]} = \frac{1}{2} \left( \lambda_{a b} - \lambda_{b a} \right)$. These conventions will also be applied to the other types of geometries appearing in this article.

Following \cite{Bailey1994}, one can naturally introduce density bundles, denoted $\mc{E}[w]$ for any $w \in \R$. In particular, we interpret $\mc{E}[1]$ as the bundle of conformal scales: sections of $\mc{E}[1]$ corresponds to metrics in the conformal class. The correspondence is achieved via the \emph{conformal metric}, $\bm{g}_{ab}$, which is a non-degenerate section of $\bigodot^2 T^* \mc{M} \otimes \mc{E} [2]$, preserved by the Levi-Civita connection of any metric in $\mbf{c}$: if $s \in \mc{E}[1]$ is a conformal scale, then the corresponding metric in $\mbf{c}$ is given by $g_{a b} = s^{-2} \bm{g}_{a b}$. This conformal metric allows us to identify sections of $T \mc{M}$ with $T^* \mc{M}\otimes \mc{E} [2]$. In this conformal setting, indices will be lowered and raised with $\bm{g}_{a b}$ and its inverse $\bm{g}^{a b}$ respectively, but with a choice of metric $g$ in $\mbf{c}$, we shall often use the metric isomorphism $T \mc{M} \accentset{g}{\cong} T^* \mc{M}$. The subbundle of $\bigodot^k T^* \mc{M}$ consisting of tracefree elements will be denoted by $\bigodot^k_\circ T^* \mc{M}$.  The tracefree part of tensors with respect to $\bm{g}_{a b}$ will be adorned with a ring, e.g. either as $\lambda_{(a b)_\circ}$ or as $\left( \lambda_{a b} \right)_\circ$.

If two metrics $g$ and $\wh{g}$ in $\mbf{c}$ are related by
\begin{align}\label{eq-conf-res}
\widehat{{g}} & = \ee^{2 \varphi} {g} \, , & \mbox{for some smooth function $\varphi$ on $\mc{M}$,}
\end{align}
their respective Levi-Civita connections $\nabla$ and $\wh{\nabla}$ are related by
\begin{align}
\widehat{\nabla}_a \alpha_b & = \nabla_a \alpha_b + (w - 1) \Upsilon_a \alpha_b - \Upsilon_b \alpha_a + \Upsilon_c \alpha^c \bm{g}_{a b} \, , & \alpha_a \in \Gamma(T^* \mc{M}[w])  \, ,\label{eq-conf-tr-form}
\end{align}
where $\Upsilon_a := \nabla_a \varphi$. The covariant exterior derivative, denoted $\dd^{\nabla}$, is given by
\begin{align*}
(\dd^{\nabla} \alpha)_{a b_1 \ldots b_k} & = \nabla_{[a} \alpha_{b_1 \ldots b_k]} \, , & \alpha_{b_1 \ldots b_k}  \in \Gamma(\bigwedge^k T^* \mc{M}[w]) \, .
\end{align*}
By convention, we take the Riemann tensor of a given metric $g_{a b}$ in $\mbf{c}$ to be defined by
\begin{align}
2 \nabla_{[a} \nabla_{b]} V^c & =: R_{a b}{}^{c}{}_{d} V^d \, , &  V^a \in \Gamma(T \mc{M}) \, . \label{eq:Riem}
\end{align}
It decomposes as
\begin{align}\label{eq-Riem_decomp}
R_{a b c d} & = W_{a b c d} + 4 \, \bm{g}_{[a|[c} \Rho_{d]|b]} \, ,
\end{align}
where its tracefree part $W_{a b c d}$ is the Weyl tensor, and $\Rho_{a b}$ is the Schouten tensor, given in terms of the Ricci tensor $\Ric_{a b} := R_{c a}{}^{c}{}_{b}$ and the Ricci scalar $\Sc := \Ric_{ab} \bm{g}^{a b}$ by
\begin{align*}
\Rho_{a b} & := \frac{1}{n} \left( \Ric_{a b} - \frac{\Sc}{2(n+1)} \bm{g}_{a b} \right) \, .
\end{align*}
We also define the Schouten scalar $\Rho := \Rho_{a b} \bm{g}^{a b}$. While the Weyl tensor is conformally invariant, the Schouten tensor transforms as
\begin{align}
\widehat{\Rho}_{a b} &  = \Rho_{a b} - \nabla_{a}\Upsilon_{b} + \Upsilon_{a}\Upsilon_{b} - \frac{1}{2} \Upsilon^{c} \Upsilon_{c} \bm{g}_{a b} \, , \label{eq:Rho_transf}
\end{align}
under the change \eqref{eq-conf-res}.

\section{Optical geometries}\label{sec:opt_geom}
\subsection{Basic definitions and facts}
We summarise the exposition given in \cite{Fino2020}, which also draws on \cite{Robinson1983}. Let $(\mc{M}, \mbf{c},K)$ be a \emph{(conformal) optical geometry} of dimension $n+2$, that is, $(\mc{M},\mbf{c})$ is a  time-oriented and oriented Lorentzian conformal manifold, and $K$ a null line distribution. This line distribution is oriented by virtue of the orientation and time-orientation of $\mc{M}$. It is a subbundle of its orthogonal complement  $K^\perp$ with respect to $\mbf{c}$, i.e.\
\begin{align}\label{eq-K-filt}
K \subset K^\perp \subset T \mc{M} \, .
\end{align}
We call the oriented rank-$n$ quotient
\begin{align*}
H_K & := K^\perp / K \, ,
\end{align*}
the \emph{screen bundle} of $K$.  The conformal structure $\mbf{c}$ on $\mc{M}$ induces a conformal structure $\mbf{c}_{H_K}$ of Riemannian signature on $H_K$, and in particular, a conformal metric $\bm{h}$ on $H_K$, that is, the non-degenerate section of $\bigodot^2 H_K^*$ defined by
\begin{align*}
\bm{h} (v + K , w + K ) & := \bm{g}(v,w)  \, , & \mbox{for any $v, w \in \Gamma(K^\perp)$.}
\end{align*}
Any non-vanishing section of $K$ will be referred to as an \emph{optical vector field}, and any non-vanishing section of $\Ann(K^\perp)$ as an \emph{optical $1$-form}.

In abstract index notation, we shall use upper, respectively, lower, minuscule Roman indices starting from the middlle of the alphabet, i.e.\ $i, j, k, \ldots$ for sections of $H_K$, respectively, $H_K^*$. Thus, the conformal metric above may be denoted $\bm{h}_{i j}$. If $s$ is a conformal scale, then $h_{i j} = s^{-2} \bm{h}_{i j}$ is a metric in $\mbf{c}_{H_K}$.

Concretely, it will be convenient to fix a metric $g$ in $\mbf{c}$ and introduce a null line distribution $L$ dual to $K$ to split the filtration \eqref{eq-K-filt} so that
\begin{align*}
T \mc{M} & = L \oplus H_{K,L} \oplus K \, , & \mbox{where $H_{K,L}  :=  K^\perp \cap L^\perp $.}
\end{align*}
Sections of $H_K$ can then be identified with sections of $H_{K,L}$, and we shall use the same index notation for sections of $H_{K,L}$ as for those of $H_{K}$. We shall also introduce a frame $\{\ell, e_i, k\} = \{e_0, e_i, e^0\}$ \emph{adapted} to the optical geometry $(\mc{M}, \mbf{c}, K)$, where $k$ and $\ell$ are sections of $K$ and $L$ respectively, such that $g(k, \ell)=1$,  and $\{e_i \}$, $i=1, \ldots ,n$, form an orthonormal frame for $H_{K,L}$. The coframe dual to it will be denoted $\{\kappa, \theta^i, \lambda \} = \{\theta^0, \theta^i,\theta_0\}$. With this notation, the metric $g$ takes the form
\begin{align*}
g & = 2  \, \theta^0 \odot \theta_0 + h_{i j} \theta^i \odot \theta^j \, .
\end{align*}
For any section $\alpha_a$ of $T^* \mc{M}$, we shall write
\begin{align*}
& \alpha^0 = \alpha( k) \, , & & \alpha_i = \alpha (e_i) \, , & & \alpha_0 = \alpha(\ell) \, ,
\end{align*}
and similarly for tensor fields of other valences. These indices may be viewed abstractly. The tracefree part of a tensor $T_{i j}$, say, with respect to $\bm{h}^{i j}$ will also be denoted by $\left( T_{i j} \right)_\circ$ or $T_{(i j)_\circ}$.

\subsection{Congruences of null geodesics}
The foliation by null curves tangent to $K$, i.e.\ the aggregate of the integral curves of any optical vector field $k$ of $K$, will be referred to as the \emph{congruence} $\mc{K}$ of  null curves associated to $K$. These curves are oriented since $K$ is oriented.

Henceforth, we assume that the curves of $\mc{K}$ are \emph{geodesics}. This property can be defined as follows: the weighted $1$-form $\bm{\kappa} = \bm{g}(k,\cdot)$ corresponding to any optical vector field $k$ satisfies
\begin{align}\label{eq:geod}
\mathsterling_k \bm{\kappa}  (v)  & = 0 \, , & v & \in \Gamma(K^\perp) \, .
\end{align}
Condition \eqref{eq:geod} tells us that any optical $1$-form is preserved along the flow of $k$, and in particular, descends to a $1$-form on the $(n+1)$-dimensional leaf space $\ul{\mc{M}}$ of $\mc{K}$. This means that $\ul{\mc{M}}$ inherits a rank-$n$ distribution $\ul{H}$ from the screen bundle $H_K$.

This leaf space $(\ul{\mc{M}},\ul{H})$ inherits additional structures on $\ul{\mc{M}}$ from the invariants of $\mc{K}$. Notably, for any optical vector field $k$, we introduce \cite{Fino2020}
\begin{enumerate}
	\item the \emph{twist} of $k$, that is, the section $\bm{\tau}$ of $\bigwedge^2 H_K^* \otimes \mc{E} [2]$ defined by
	\begin{align}
	\bm{\tau} (v + K , w + K) & := \dd^\nabla \bm{\kappa} ( v , w) \, , & v,w & \in \Gamma(K^\perp) \, ; \label{eq-conf-twist}
	\end{align}
	\item the \emph{shear} of $k$, that is, the section $\bm{\sigma}$ of $\bigodot^2_\circ H^*_K \otimes \mc{E} [2]$ defined by
	\begin{align}
	\bm{\sigma} (v +K , w+K) \bm{\kappa} & := \frac{1}{2} \left( \mathsterling_{k} \bm{g} (v , w ) \bm{\kappa} - \bm{g} (v , w )  \mathsterling_{k} \bm{\kappa} \right) \, , & v, w & \in \Gamma(K^\perp) \, .\label{eq-conf-shear} 
	\end{align}
\end{enumerate}
It is clear that any rescaling of $k$ induces a rescaling of its twist and shear. These definitions thus extend to the notions of \emph{twist} and \emph{shear of the congruence $\mc{K}$}, both of which are conformal invariants.

With a choice of metric $g$ in $\mbf{c}$, we can also define the \emph{expansion} of $k$ to be the smooth function $\epsilon$ given by
\begin{align*}
\epsilon \, \kappa & :=   \kappa \, \mathrm{div}  k -  \nabla_{k} \kappa   \, ,
\end{align*}
where $\kappa = g(k, \cdot)$ and $\mathrm{div}  k = \nabla_a k^a$. Again, there is a well-defined notion of expansion of $\mc{K}$. While the expansion is not conformally invariant, one can always choose a metric $g$ in the conformal class for which the congruence $\mc{K}$ generated by $k$ is non-expanding, $\epsilon = 0$. In fact, locally, this defines a subclass $\accentset{n.e.}{\mbf{c}}$ of metrics in $\mbf{c}$ with the property that whenever $g$ is in $\accentset{n.e.}{\mbf{c}}$, the congruence $\mc{K}$ is non-expanding. Any two metrics in $\accentset{n.e.}{\mbf{c}}$ differ by a factor constant along $K$ -- see \cite{Fino2020}.

Let us now review the geometric interpretation of the twist and shear of $\mc{K}$. The twist of $\mc{K}$, if non-zero, induces a skew-symmetric bundle map on $\ul{H}$ whose rank is given by the rank of $\bm{\tau}$ -- this is clear from the geodesic property \eqref{eq:geod}, the defining equation \eqref{eq-conf-twist}, and the naturality of the exterior derivative. Let us now assume that $\mc{M}$ has dimension $2m+2$ and $\mc{K}$ is maximally twisting, that is, $\bm{\tau}$ has maximal rank. Choose a metric $g$ in $\accentset{n.e.}{\mbf{c}}$, and an optical vector field $k$ for which the geodesics of $\mc{K}$ are affinely parametrised. Then the $1$-form $\kappa = g(k, \cdot)$ satisfies $\mathsterling_k \kappa = 0$ and $\kappa \wedge ( \dd \kappa )^m$ is non-zero. This means that $\kappa$ is the pullback of a $1$-form $\ul{\theta}^0$ on $\ul{\mc{M}}$ that annihilates $\ul{H}$ and satisfies $\ul{\theta}^0 \wedge (\dd \ul{\theta}^0)^{m} \neq 0$, i.e.\ $\ul{\theta}^0$ is a \emph{contact form}. Thus, the distribution $\ul{H}$ on $\ul{\mc{M}}$ must be \emph{contact}, i.e.\ $\ul{H}$ bracket-generates the tangent space of $\mc{M}$ at every point, i.e.\ $T \ul{\mc{M}} = \ul{H} + [ \ul{H}, \ul{H} ]$.

On the other hand, if the congruence $\mc{K}$ is non-shearing, i.e.\ $\bm{\sigma}=0$, it is immediate from \eqref{eq-conf-shear} that the conformal structure induced on $H_K$ is preserved along the flow of any generator of $\mc{K}$. In this case, the distribution $\ul{H}$ on $\ul{\mc{M}}$ is endowed with a conformal structure $\ul{\mbf{c}}_{\ul{H}}$. More precisely, there is a one-to-one correspondence between metrics in $\accentset{n.e.}{\mbf{c}}$ and metrics in $\ul{\mbf{c}}_{\ul{H}}$.

Thus, combining these two properties, we conclude that a maximally twisting non-shearing congruence of null geodesics $\mc{K}$ induces a so-called \emph{sub-conformal contact structure} $(\ul{H}, \ul{\mbf{c}}_{\ul{H}})$ on its leaf space $\ul{\mc{M}}$, i.e.\ $\ul{H}$ is contact and is endowed with a conformal structure $\ul{\mbf{c}}_{\ul{H}}$ \cite{Alekseevsky2018,Alekseevsky2021,Fino2020}. In Section \ref{sec:CR_geom}, we shall equip $\ul{H}$ with a bundle complex structure. 

\subsection{Integrability condition}
It is well-known that if $(\mc{M},\mbf{c}, K)$ is a four-dimensional  conformal optical geometry with non-shearing congruence of null geodesics $\mc{K}$, then, for any optical vector field $k$, the Weyl tensor satisfies
\begin{align*}
\bm{\kappa}_{[a} W_{b]ef[c} \bm{\kappa}_{d]} k^e k^f & = 0 \, ,& \mbox{where $\bm{\kappa}_a = \bm{g}_{a b} k^b$.}
\end{align*}
In higher dimensions, the analogous result is given by the following proposition:
\begin{thm}\label{thm-int-cond>4}
	Let $(\mc{M},\mbf{c},K)$ be an $(n+2)$-dimensional conformal optical geometry, where $n>2$,  with non-shearing congruence of null geodesics $\mc{K}$. Then, for any optical vector field $k$, the Weyl tensor satisfies
	\begin{align}\label{eq-int-cond-non-shearing}
	4 \, \bm{\kappa}_{[a} W_{b]ef[c}  \bm{\kappa}_{d]} k^e k^f & = \bm{\tau}_{a b e} \bm{\tau}^{e}{}_{c d} + \frac{4}{n} \bm{\tau}_{[a}{}^{e f} \bm{g}_{b][c} \bm{\tau}_{d] e f} \, ,
	\end{align}
	where $\bm{\kappa}_a = \bm{g}_{a b} k^b$ and  $\bm{\tau}_{abc} := 3 \, \bm{\kappa}_{[a} \nabla_{b} \bm{\kappa}_{c]}$.
\end{thm}

\begin{proof}
	With no loss of generality, we work with a metric $g$ in $\accentset{n.e.}{\mbf{c}}$ so that $\mc{K}$ is non-expanding, and we choose an optical vector field $k$ for which the geodesic curves of $\mc{K}$ are affinely parametrised so that, with $\bm{\kappa}_a = \bm{g}_{a b} k^b$, we have
	\begin{align}
	k^b \nabla_b \bm{\kappa}_a & = 0 \, , \label{eq-aff-par} \\
	\bm{\kappa}_{[a} \left( \nabla_{b]} \bm{\kappa}_{[c} \right) \bm{\kappa}_{d]} & = \bm{\kappa}_{[a} \bm{\tau}_{b] [c} \bm{\kappa}_{d]} \, , \label{eq-to-prol}
	\end{align}
	where $\bm{\tau}_{b c}$ is a (weighted) $2$-form such that $\bm{\tau}_{a b c} = 3 \, \bm{\kappa}_{[a} \bm{\tau}_{b c]}$. It also satisfies $\bm{\tau}_{b c} k^c =0$. In other words, $\bm{\tau}_{a b}$ represents the twist $\bm{\tau}_{i j}$ of $k$ as a $2$-form on $\mc{M}$. Now, taking a covariant derivative of \eqref{eq-to-prol} along $k^a$, commuting the covariant derivatives, applications of the Leibniz rule, using \eqref{eq:Riem} and repeated applications of \eqref{eq-aff-par} leads to
	\begin{align*}
	\bm{\kappa}_{[a} R_{b] e f [c} \bm{\kappa}_{d]} k^{e} k^{f} 
	& =  \bm{\kappa}_{[a} \bm{\tau}_{b]}{}^{e} \bm{\tau}_{e [c} \bm{\kappa}_{d]} \, .
	\end{align*}
	Taking the trace of this expression gives
	\begin{align*}
	\Ric_{a b} k^{a} k^{b} & = \bm{\tau}_{a b} \bm{\tau}^{a b} \, ,
	\end{align*}
	so that using the expression for the Weyl tensor \eqref{eq-Riem_decomp} yields
	\begin{align*}
	\bm{\kappa}_{[a} W_{b] e f [c} \bm{\kappa}_{d]} k^{e} k^{f} & = \bm{\kappa}_{[a} \bm{\tau}_{b]}{}^{e} \bm{\tau}_{e [c} \bm{\kappa}_{d]} + \frac{1}{n} \bm{\tau}_{e f} \bm{\tau}^{e f} \bm{\kappa}_{[a} \bm{g}_{b] [c} \bm{\kappa}_{d]} \, .
	\end{align*}
	Finally, to obtain the expression \eqref{eq-int-cond-non-shearing}, we simply note that $\bm{\tau}_{a b c} = 3 \, \bm{\kappa}_{[a} \bm{\tau}_{b c]}$.
\end{proof}

The formula \eqref{eq-int-cond-non-shearing} can also be derived from the computations given in \cite{Ortaggio2007}. We shall soon give a geometric interpretation to it. But before we proceed, we need to introduce some additional geometric concepts.

\begin{rem}
	In dimension four, the RHS of condition \eqref{eq-int-cond-non-shearing} is always zero.
\end{rem}

\section{Partially integrable contact almost CR structures}\label{sec:CR_geom}
\subsection{Almost CR structures}
We recall some basic notions regarding almost CR geometry. By and large, we follow the conventions and approaches of \cite{Tanaka1975,Webster1978,Gover2005,Cap2008,Cap2010,Matsumoto2016,Case2020,Fino2021}, to which the reader is referred for more detailed accounts. Let $\ul{\mc{M}}$ be a $(2m+1)$-dimensional smooth manifold. An \emph{almost Cauchy--Riemann (CR) structure} on $\ul{\mc{M}}$ consists of a pair $(\ul{H},\ul{J})$ where $\ul{H}$ is a rank-$2m$ distribution and $\ul{J}$ a bundle complex structure on $\ul{H}$, i.e.\ $\ul{J} \circ \ul{J} = - \ul{\Id}$, where $\ul{\Id}$ is the identity map on $\ul{H}$. This means that the complexification ${}^\C \ul{H}$ of $\ul{H}$ splits as ${}^\C \ul{H} = \ul{H}^{(1,0)} \oplus \ul{H}^{(0,1)}$ where $\ul{H}^{(1,0)}$ and $\ul{H}^{(0,1)}$ are the rank-$m$ $\ii$-eigenbundle and $-\ii$-eigenbundle of $\ul{J}$ respectively. If $\ul{H}^{(1,0)}$ (or equivalently $\ul{H}^{(0,1)}$) is involutive or integrable,\footnote{We shall not distinguish between the two terms, involutive and integrable, here, brushing aside any analytic issues that may arise.} i.e.\ $[ \ul{H}^{(1,0)} , \ul{H}^{(1,0)} ] \subset \ul{H}^{(1,0)}$, we refer to $(\ul{H},\ul{J})$ simply as a \emph{CR structure}. When $m=1$, an almost CR structure is always integrable. An \emph{(almost) pseudo-Hermitian structure} on $\ul{\mc{M}}$ is an (almost) CR structure $(\ul{H},\ul{J})$ together with a choice of non-vanishing section of $\Ann(\ul{H})$.

We shall assume further that the almost CR structure is \emph{contact} or \emph{non-degenerate}, i.e.\ $\ul{H}$ is a contact distribution, and that it is \emph{partially integrable}, i.e.\ the bracket of two sections in $\ul{H}^{(1,0)}$ is a section of ${}^\C \ul{H}$. One can also describe such an almost CR structure as a sub-conformal contact structure $(\ul{H}, \ul{\mbf{c}}_{\ul{H}})$ equipped with a compatible bundle complex structure. Further equivalent descriptions can be found in the aforementioned references.

In order to make the description of $(\ul{\mc{M}}, \ul{H}, \ul{J})$ more concrete, let us fix a contact form $\ul{\theta}^0$. Then there exists a unique vector field $\ul{e}_0$, known as the \emph{Reeb vector field}, satisfying $\ul{\theta}^0 (\ul{e}_0) = 1$ and $\dd \ul{\theta}^{0} ( \ul{e}_0 , \cdot ) = 0$. It induces a splitting
\begin{align*}
{}^\C T \ul{\mc{M}} =  {}^\C \ul{L} \oplus \ul{H}^{(1,0)} \oplus \ul{H}^{(0,1)} \, ,
\end{align*}
where $\ul{L}$ is the real line distribution spanned by $\ul{e}_0$. Complete $\ul{e}_0$ to a (complex) frame $\{ \ul{e}_0 , \ul{e}_\alpha , \overline{\ul{e}}_{\bar{\beta}} \}$, $\alpha, \bar{\beta} = 1, \ldots , m$, adapted to $(\ul{H}, \ul{J})$, i.e.\ $\{ \ul{e}_\alpha \}$  and $\{ \overline{\ul{e}}_{\bar{\beta}} \}$, $\alpha, \bar{\beta} = 1, \ldots , m$, span $\ul{H}^{(1,0)}$ and $\ul{H}^{(0,1)}$ respectively. Denote by $\{ \ul{\theta}{}^0 , \ul{\theta}{}^\alpha , \overline{\ul{\theta}}{}^{\bar{\alpha}} \}$, the coframe dual to $\{ \ul{e}_0 , \ul{e}_\alpha , \overline{\ul{e}}_{\bar{\beta}} \}$, ${\alpha, \bar{\beta} = 1, \ldots , m}$. Then the contact form $\ul{\theta}^0$ satisfies $\dd \ul{\theta}^{0}  = \ii \ul{h}_{\alpha \bar{\beta}} \ul{\theta}^{\alpha} \wedge \overline{\ul{\theta}}{}^{\bar{\beta}}$, where $\ul{h}_{\alpha \bar{\beta}}$ is a Hermitian matrix referred to as the \emph{Levi form} of $\ul{\theta}^0$. The signature of $\ul{h}_{\alpha \bar{\beta}}$ is an invariant of $(\ul{H},\ul{J})$. \emph{We shall henceforth assume that $\ul{h}_{\alpha \bar{\beta}}$ has positive definite signature.}

We shall also use the indices just introduced in an abstract way. Thus, sections of $\ul{H}^{(1,0)}$ and $\ul{H}^{(0,1)}$ will be adorned with minuscule Greek indices, plain and barred respectively, and similarly for their duals, e.g.\ $\ul{V}^{\alpha} \in \Gamma(\ul{H}^{(1,0)})$ and $\ul{\mu}{}_{\bar{\alpha}} \in \Gamma((\ul{H}^{(0,1)})^*)$. In addition, sections of $\Ann(\ul{H})$ and their duals will be adorned with a lower, respectively, upper $0$, e.g.\ $\ul{\alpha}{}_0 \in \Gamma(\Ann(\ul{H}))$ and $\ul{v}{}^0 \in \Gamma(\ul{L})$.  As before, symmetrisation and skew-symmetrisation will be denoted by round brackets and square brackets respectively. Index types can be converted using the Levi form $h_{\alpha \bar{\beta}}$ of a given contact form. Clearly, complex conjugation on ${}^\C \ul{H}$ changes the index type, so we shall write $\ul{v}{}^{\bar{\alpha}}$ for $\overline{\ul{v}{}^{\alpha}}$, and so on. We also note that $\overline{\ul{h}_{\alpha \bar{\beta}}} = \ul{h}_{\alpha \bar{\beta}}$. The tracefree part of a tensor of mixed valence with respect to $\ul{h}_{\alpha \bar{\beta}}$, for instance $\ul{T}_{\alpha \bar{\beta}}$, will be adorned with a ring, e.g.\ $\left( \ul{T}_{\alpha \bar{\beta}} \right)_\circ$ so that $\left( \ul{T}_{\alpha \bar{\beta}} \right)_\circ \ul{h}^{\alpha \bar{\beta}} = 0$.

An \emph{(infinitesimal) symmetry} of $(\ul{\mc{M}}, \ul{H}, \ul{J})$ is a vector field $\ul{v}$ on $\ul{\mc{M}}$ that preserves both $\ul{H}$ and $\ul{J}$, i.e.\ $\left( \mathsterling_{\ul{v}} \ul{\theta}^0 \right) \wedge \ul{\theta}^0 = 0$ for any contact form $\ul{\theta}^0$ of  $\ul{H}$, and $\mathsterling_{\ul{v}} \ul{\alpha} \in \Gamma(\Ann(\ul{H}^{(0,1)}))$ for any $1$-form $\ul{\alpha}$ in $\Ann(\ul{H}^{(0,1)})$. Such an infinitesimal symmetry $\ul{v}$ is said to be \emph{transverse} if it inserts non-trivially into any contact form, i.e.\ $\ul{v}$ is the Reeb vector field of some contact form of $\ul{H}$.

Analogous to conformal density bundles, one can also introduce \emph{CR density bundles} $\ul{\mc{E}}(w,w')$ for any $w, w' \in \C$ such that $w - w' \in \Z$. The details of the definition of these $\C^*$-principal bundles over $\ul{\mc{M}}$ are given in \cite{Gover2005,Cap2008,Cap2010}. We shall simply note here that $\overline{\ul{\mc{E}}(w,w')} = \ul{\mc{E}}(w',w)$. Such density bundles allow us to define analogues of the conformal metric in the CR setting, namely, a canonical section $\ul{\bm{\theta}}{}^0$ of $T^* \ul{\mc{M}} \otimes \ul{\mc{E}}(1,1)$, and a canonical section $\ul{\bm{h}}_{\alpha \bar{\beta}}$ of $(\ul{H}^{(1,0)})^* \otimes (\ul{H}^{(0,1)})^* \otimes \ul{\mc{E}}(1,1)$ with the property that for each $\ul{s} \in \Gamma(\ul{\mc{E}}(-1,-1))$, $\ul{\theta}{}^0 = \ul{s} \ul{\bm{\theta}}{}^0$ is a contact form with Levi form $\ul{h}_{\alpha \bar{\beta}} = \ul{s}  \ul{\bm{h}}_{\alpha \bar{\beta}}$. These sections are weighted analogues of the contact form and its Levi form. The latter and its inverse $\ul{\bm{h}}{}^{\alpha \bar{\beta}}$ identify $\ul{H}^{(1,0)}$ with $(\ul{H}^{(0,1)})^*(1,1)$ and $\ul{H}^{(0,1)}$ with $(\ul{H}^{(1,0)})^*(1,1)$. In effect, indices can also be raised and lowered using $\ul{\bm{h}}_{\alpha \bar{\beta}}$, e.g.\ $\ul{v}{}_{\alpha} = \ul{\bm{h}}{}_{\alpha \bar{\beta}} \ul{v}{}^{\bar{\beta}}$, thereby changing the weights of the tensors.

\subsection{Webster--Tanaka connections}\label{sec:WT_conn}
Let $(\ul{\mc{M}}, \ul{H}, \ul{J})$ be a partially integrable contact almost CR structure of dimension $2m+1$ as before. For definiteness, we shall assume $m>1$. Then for each contact form $\ul{\theta}^0$, there is a unique linear connection $\ul{\nabla}$ on $T \ul{\mc{M}}$ called the \emph{Webster--Tanaka connection}, which preserves $\ul{\theta}^0$, $\dd \ul{\theta}^0$ and the complex structure $\ul{J}$ and has prescribed torsion as follows. If $\{ \ul{\theta}{}^0 , \ul{\theta}{}^\alpha , \overline{\ul{\theta}}{}^{\bar{\alpha}} \}$ is an adapted coframe, the Cartan structure equations read as
\begin{subequations}\label{eq:structure_CR}
	\begin{align}
	\dd \ul{\theta}^{0} & = \ii \ul{h}_{\alpha \bar{\beta}} \ul{\theta}^{\alpha} \wedge \overline{\ul{\theta}}{}^{\bar{\beta}}  \, , \\
	\dd \ul{\theta}^{\alpha} & = \ul{\theta}^{\beta} \wedge \ul{\Gamma}_{\beta}{}^{\alpha} +  \ul{A}^{\alpha}{}_{\bar{\beta}}   \ul{\theta}^{0} \wedge \overline{\ul{\theta}}{}^{\bar{\beta}} - \frac{1}{2} \ul{\Nh}_{\bar{\beta} \bar{\gamma}}{}^{\alpha}\overline{\ul{\theta}}{}^{\bar{\beta}} \wedge \overline{\ul{\theta}}{}^{\bar{\gamma}} \, , \label{eq:structure_CR_A}\\
	\dd \overline{\ul{\theta}}{}^{\bar{\alpha}} & = \overline{\ul{\theta}}{}^{\bar{\beta}} \wedge \ul{\Gamma}_{\bar{\beta}}{}^{\bar{\alpha}} + \ul{A}^{\bar{\alpha}}{}_{\beta}  \ul{\theta}^{0} \wedge \ul{\theta}^{\beta} - \frac{1}{2} \ul{\Nh}_{\beta \gamma}{}^{\bar{\alpha}} \theta^{\beta} \wedge \theta^{\gamma} \, .
	\end{align}
\end{subequations}
where $\ul{\Gamma}_{\beta}{}^{\alpha}$ is the connection $1$-form of $\ul{\nabla}$ for that coframe, $\ul{h}_{\alpha \bar{\beta}}$ the Levi form of $\ul{\theta}^0$, $\ul{A}_{\alpha \beta}$ (and its conjugate $\ul{A}_{\bar{\alpha} \bar{\beta}}$) the \emph{pseudo-Hermitian torsion tensor}, and $\ul{\Nh}_{\alpha \beta \gamma}$ (and its conjugate $\ul{\Nh}_{\bar{\alpha} \bar{\beta} \bar{\gamma}}$) the \emph{Nijenhuis tensor},\footnote{This is a slight abuse of terminology, since strictly, the Nijenhuis tensor is the real tensor defined by both $\ul{\Nh}_{\alpha \beta \gamma}$ and $\ul{\Nh}_{\bar{\alpha} \bar{\beta} \bar{\gamma}}$.}  and these satisfy the symmetries $\ul{A}_{\alpha \beta} = \ul{A}_{(\alpha \beta)}$ and $\ul{\Nh}_{\alpha \beta \gamma} = \ul{\Nh}_{[\alpha \beta] \gamma}$ with $\ul{\Nh}_{[\alpha \beta \gamma]} = 0$.
\begin{rem}
	The Nijenhuis tensor $\ul{\Nh}_{\alpha \beta \gamma}$ is the obstruction to the integrability of $(\ul{H},\ul{J})$ and is a CR invariant. On the other hand, 
	$\ul{A}_{\alpha \beta}$ depends on the choice of contact form $\ul{\theta}^0$. Its vanishing is equivalent to its Reeb vector field $\ul{e}_0$ being an infinitesimal symmetry of $(\ul{H},\ul{J})$.
\end{rem}

The structure equations \eqref{eq:structure_CR} give the commutation relations
\begin{subequations}\label{eq:CR_com}
	\begin{align}
	(\ul{\nabla}_{\alpha} \ul{\nabla}_{\bar{\beta}} - \ul{\nabla}_{\bar{\beta}} \ul{\nabla}_{\alpha} ) \ul{f} & = - \ii \bm{\ul{h}}_{\alpha \bar{\beta}} \ul{\nabla}_{0} \ul{f} \, , \\
	(\ul{\nabla}_{\alpha} \ul{\nabla}_{0} - \ul{\nabla}_{0} \ul{\nabla}_{\alpha} ) \ul{f} & = \ul{A}_{\alpha}{}^{\bar{\beta}} 
	\ul{\nabla}_{\bar{\beta}} \ul{f} \, , \\
	(\ul{\nabla}_{\alpha} \ul{\nabla}_{\beta} - \ul{\nabla}_{\beta} \ul{\nabla}_{\alpha} ) \ul{f} & = \ul{\Nh}_{\alpha \beta}{}^{\bar{\gamma}} \ul{\nabla}_{\bar{\gamma}} \ul{f} \, , & 
	\end{align}
\end{subequations}
for any smooth function $\ul{f}$, and similarly for their complex conjugates.

The curvature tensors of the Webster--Tanaka connection are given by
\begin{align*}
(\ul{\nabla}_{\alpha} \ul{\nabla}_{\bar{\beta}} - \ul{\nabla}_{\bar{\beta}} \ul{\nabla}_{\alpha} ) \ul{V}^{\gamma} + \ii \bm{\ul{h}}_{\alpha \bar{\beta}} \ul{V}^\gamma & =: \ul{R}_{\alpha \bar{\beta} \delta}{}^{\gamma} \ul{V}^{\delta}  \, , \\
(\ul{\nabla}_{\alpha} \ul{\nabla}_{0} - \ul{\nabla}_{0} \ul{\nabla}_{\alpha} ) \ul{V}^{\gamma} - \ul{A}_{\alpha}{}^{\bar{\beta}} \nabla_{\bar{\beta}} \ul{V}^\gamma & =: \ul{R}_{\alpha 0 \delta}{}^{\gamma} \ul{V}^{\delta}  \, , \\
(\ul{\nabla}_{\alpha} \ul{\nabla}_{\beta} - \ul{\nabla}_{\beta} \ul{\nabla}_{\alpha} ) \ul{V}^{\gamma} - \ul{\Nh}_{\alpha \beta}{}^{\bar{\delta}} \nabla_{\bar{\delta}} \ul{V}^\gamma & =: \ul{R}_{\alpha \beta \delta}{}^{\gamma} \ul{V}^{\delta}  \, ,\end{align*}
for any section $\ul{V}^\alpha$ of $\ul{H}^{(1,0)}$, and similarly for their complex conjugates. Analogous formulae can be derived for sections of $\ul{H}^{(0,1)}$ and their duals.

The curvature and torsion tensors are related by the first Bianchi identities:
\begin{subequations}\label{eq:Bianchi1_CR}
	\begin{align}
	2 \ul{R}_{\beta}\,^{[\gamma}\,_{\alpha}\,^{\delta]} & =   - \ul{\Nh}_{\epsilon \beta \alpha} \ul{\Nh}^{\gamma \delta \epsilon} \, , \label{eq:B1_R=N2}\\
	\ul{R}_{\alpha 0 \beta}\,^{\gamma}   & =  \ul{\nabla}^{\gamma}{\ul{A}_{\alpha \beta}}  + \ul{A}^{\delta \gamma} \ul{\Nh}_{\delta \alpha \beta} \, , \label{eq:B1_R=DA+AN} \\
	\ul{R}_{\beta \gamma \alpha}\,^{\delta}  & =  \ul{\nabla}^{\delta}{\ul{\Nh}_{\beta \gamma \alpha}} - 2 \ii \ul{A}_{\alpha [\beta} \delta^{\delta}_{\gamma]} \, , \label{eq:B1_R=DN-A}\\
	\ul{\nabla}_{0}{\ul{\Nh}_{\beta \gamma \alpha}}  & =  - 2 \ul{\nabla}_{[\beta}{\ul{A}_{\gamma] \alpha }}  \, ,  \label{eq:B1_DN=DA} \\
	\ul{\nabla}_{[\delta}{\ul{\Nh}_{\beta \gamma] \alpha}} & = 0 \, ,  \label{eq:B1_DN=0}
	\end{align}
\end{subequations}
together with their complex conjugates, from which we obtain $\ul{R}_{[\beta \delta \gamma]}{}^{\alpha} = 0$ and $2 \ul{R}_{0 [\beta \gamma]}\,^{\alpha}  = \ul{A}^{\alpha \delta} \ul{\Nh}_{\beta \gamma \delta}$. We record the second Bianchi identities in the following general form:
\begin{subequations}\label{eq:Bianchi2_CR}
	\begin{align}
	& 2 \ul{\nabla}_{[\delta} {\ul{R}_{\epsilon] \bar{\gamma}}\,_{\alpha}\,^{\beta}} 
	+ \ul{\nabla}_{\bar{\gamma}}{\ul{R}_{\delta \epsilon \alpha}\,^{\beta}} 
	- \ul{\Nh}_{\delta \epsilon}{}^{\bar{\phi}}  \ul{R}_{\bar{\phi} \bar{\gamma} \alpha}\,^{\beta}  
	+ 2 \ii  \ul{\bm{h}}_{[\delta| \bar{\gamma}} \ul{R}_{|\epsilon] 0 \alpha}\,^{\beta}    = 0  \, ,  \label{eq:B2_DmR}\\
	&
	\ul{\nabla}_{\gamma}{\ul{R}_{\bar{\delta} 0 \alpha}\,^{\beta}} 
	- \ul{\nabla}_{\bar{\delta}}{\ul{R}_{\gamma 0 \alpha}\,^{\beta}} 
	+ \ul{\nabla}_{0}{\ul{R}_{\gamma \bar{\delta} \alpha}\,^{\beta}} - \ul{A}_{\gamma}{}^{\bar{\epsilon}}  \ul{R}_{\bar{\delta} \bar{\epsilon} \alpha}\,^{\beta}  + \ul{A}_{\bar{\delta}}{}^{\epsilon}  \ul{R}_{\gamma \epsilon \alpha}\,^{\beta} = 0 \, , \label{eq:B2_D0R}\\
	&  
	2 \ul{\nabla}_{[\gamma}{\ul{R}_{\delta] 0 \alpha}\,^{\beta}} 
	+  \ul{\nabla}_{0} {\ul{R}_{\gamma \delta}\,_{\alpha}\,^{\beta}}  
	- 2 \ul{A}^{\bar{\epsilon}}{}_{[\gamma}  \ul{R}_{\delta] \bar{\epsilon} \alpha}\,^{\beta}  
	- \ul{\Nh}_{\gamma \delta}{}^{\bar{\epsilon}}  \ul{R}_{\bar{\epsilon} 0 \alpha}\,^{\beta} = 0 \, ,  \\
	&  \ul{\nabla}_{[\gamma} \ul{R}_{\delta \epsilon] \alpha}\,^{\beta}  + \ul{\Nh}_{[\delta \epsilon| \phi}  \ul{R}_{|\gamma]}\,^{\phi}\,_{\alpha}\,^{\beta} = 0\, , 
	\end{align}
\end{subequations}
together with their complex conjugates.

At this stage, we define the \emph{Chern--Moser tensor} $\ul{S}_{\alpha \bar{\gamma} \beta \bar{\delta}}$ of $(\ul{\mc{M}},\ul{H}, \ul{J})$ to be the totally tracefree totally  symmetric part of $\ul{R}_{\alpha \bar{\gamma} \beta \bar{\delta}} = \ul{R}_{\alpha \bar{\gamma}}{}_{\beta}{}^{\epsilon} \ul{\bm{h}}_{\epsilon \bar{\delta}}$, i.e.\
\begin{align*}
\ul{S}_{\alpha}{}^{\gamma}{}_{\beta}{}^{\delta} & := \left( \ul{R}_{(\alpha}{}^{(\gamma}{}_{\beta)}{}^{\delta)} \right)_\circ \, .
\end{align*}
The Chern--Moser tensor is a CR invariant. The vanishing of both $\ul{\Nh}_{\alpha \beta \gamma}$ and $\ul{S}_{\alpha \bar{\gamma} \beta \bar{\delta}}$ is equivalent to the almost CR structure being locally \emph{CR flat}, i.e.\ $\ul{\mc{M}}$ is locally diffeomorphic to the CR sphere.

We shall also need the \emph{Webster--Ricci tensor} $\ul{\Ric}{}_{\gamma}{}^{\delta} := \ul{\bm{h}}{}^{\alpha \bar{\beta}} \ul{R}{}_{\alpha \bar{\beta} \gamma}{}^{\delta}$, the \emph{Webster--Ricci scalar} $\ul{\Sc} := \ul{\Ric}{}_{\gamma}{}^{\gamma}$, the \emph{Webster--Schouten tensor} and the \emph{Webster--Schouten scalar}
\begin{align*}
\ul{\Rho}_{\alpha \bar{\beta}} & := \frac{1}{m+2}\left( \ul{\Ric}_{\alpha \bar{\beta}} - \frac{1}{2m+2} \ul{\Sc} \, \ul{\bm{h}}_{\alpha \bar{\beta}} \right) \, , & \ul{\Rho} :=  \ul{\Rho}_{\alpha \bar{\beta}} \ul{\bm{h}}^{\alpha \bar{\beta}} \, ,
\end{align*}
respectively. Equation \eqref{eq:B1_R=N2} then allows us to decompose $\ul{R}_{\alpha \bar{\gamma} \beta}{}^{\delta}$ as
\begin{align*}
\ul{R}_{\alpha}{}^{\gamma}{}_{\beta}{}^{\delta} 
=  \frac{1}{4} \ul{\Nh}^{\gamma \delta \epsilon} \ul{\Nh}_{\alpha \beta \epsilon} - \frac{1}{2} \ul{\Nh}_{\epsilon (\alpha \beta)} \ul{\Nh}^{\gamma \delta \epsilon} - \frac{1}{2} \ul{\Nh}^{\epsilon (\gamma \delta)} \ul{\Nh}_{\alpha \beta \epsilon} 
+  \ul{S}_{\alpha}{}^{\gamma}{}_{\beta}{}^{\delta} + 4  {\ul{ \NRho}}_{(\alpha}{}^{(\gamma} \delta{}_{\beta)}^{\delta)} \, ,
\end{align*}
where
\begin{align*}
\ul{\NRho}_{\alpha}{}^{\gamma}   & :=  \ul{\Rho}{}_{\alpha}{}^{\gamma}  + \frac{1}{m+2} \left(   - \frac{1}{2} \ul{\Nh}_{\beta \delta \alpha} \ul{\Nh}^{\beta \delta \gamma} + \frac{1}{4} \ul{\Nh}_{\alpha \beta \delta} \ul{\Nh}^{\gamma \beta \delta}  + \frac{1}{8 (m+1)  }  \ul{\Nh}_{\epsilon \delta \beta } \ul{\Nh}^{\epsilon \delta \beta}   \delta{}_{\alpha}^{\gamma} \right) \, .
\end{align*}

\begin{defn}\label{def:aCR-E}
	Let $(\ul{\mc{M}},\ul{H},\ul{J})$ be a partially integrable contact almost CR manifold. We say that $(\ul{\mc{M}},\ul{H},\ul{J})$ is \emph{almost CR--Einstein} if it admits a contact form $\ul{\theta}^{0}$ such that its pseudo-Hermitian torsion tensor $\ul{A}_{\alpha \beta}$, its Webster--Schouten tensor $\ul{\Rho}_{\alpha \bar{\beta}}$ and the Nijenhuis tensor $\ul{\Nh}_{\alpha \beta \gamma}$ satisfy
	\begin{align}\label{eq:CR-Einstein}
	\ul{A}_{\alpha \beta} & = 0 \, , &
	\ul{\nabla}^{\gamma} \ul{\Nh}_{\gamma (\alpha \beta)} & = 0 \, , &
	\left( \ul{\Rho}_{\alpha \bar{\beta}} - \frac{1}{m+2}\ul{\Nh}_{\alpha \delta \gamma} \ul{\Nh}_{\bar{\beta}}{}^{\delta \gamma} \right)_\circ & = 0 \, .
	\end{align}
	We refer to $(\ul{H},\ul{J},\ul{\theta}^0)$ as an \emph{almost CR--Einstein structure}. When $(\ul{H},\ul{J})$ is integrable, i.e.\ $\ul{\Nh}_{\alpha \beta \gamma} =0$, we say that $(\ul{\mc{M}},\ul{H},\ul{J})$ is \emph{CR--Einstein}.
\end{defn}
We shall now give an equivalent formulation of almost CR--Einstein structures.
\begin{prop}\label{prop:CR-Einstein_Ric}
	A partially integrable contact almost CR manifold $(\ul{\mc{M}},\ul{H},\ul{J})$ is almost CR--Einstein if and only if it admits a contact form $\ul{\theta}^{0}$ such that its pseudo-Hermitian torsion tensor $\ul{A}_{\alpha \beta}$, its Webster--Ricci tensor $\ul{\Ric}_{\alpha \bar{\beta}}$ and the Nijenhuis tensor $\ul{\Nh}_{\alpha \beta \gamma}$ satisfy
	\begin{subequations}\label{eq:CR-E2}
		\begin{align}
		\ul{A}_{\alpha \beta} & = 0 \, ,  \label{eq:CR-E-A}\\
		\ul{\nabla}^{\gamma}{\ul{\Nh}_{\gamma (\alpha \beta)}} & = 0 \, ,  \label{eq:CR-E-NabN}\\
		\ul{\Ric}_{\alpha}\,^{\beta} -\ul{\Nh}_{\alpha \delta \gamma} \ul{\Nh}^{\beta \delta \gamma}   & = \ul{\Lambda} \delta_{\alpha}^{\beta} \, , & \mbox{for some constant $\ul{\Lambda}$.}\label{eq:CR-E-R}
		\end{align}
	\end{subequations}
\end{prop}

\begin{proof}
	That \eqref{eq:CR-E2} implies \eqref{eq:CR-Einstein} is clear. For the converse, let us assume \eqref{eq:CR-Einstein}. Then \eqref{eq:CR-E2} holds except that we do not know whether $\ul{\Lambda}$ is constant. We proceed to demonstrate that this is the case. To this end, we take covariant derivatives of \eqref{eq:CR-E-R} to get
	\begin{align}
	\ul{\nabla}_{0} \ul{\Lambda} \delta_{\alpha}^{\beta} & = \ul{\nabla}_{0}  \ul{\Ric}_{\alpha}\,^{\beta} - \ul{\nabla}_{0} \left(\ul{\Nh}_{\alpha \delta \gamma} \ul{\Nh}^{\beta \delta \gamma} \right)  \, , \label{eq:nabLam0} \\
	\ul{\nabla}_{\alpha} \ul{\Lambda} & = \ul{\nabla}_{\beta}  \ul{\Ric}_{\alpha}\,^{\beta} - \ul{\nabla}_{\beta} \left(\ul{\Nh}_{\alpha \delta \gamma} \ul{\Nh}^{\beta \delta \gamma} \right)  \, .\label{eq:nabLam1}
	\end{align}
	Under our assumptions \eqref{eq:CR-Einstein}, the first Bianchi identities \eqref{eq:B1_R=DA+AN}, \eqref{eq:B1_R=DN-A} and \eqref{eq:B1_DN=DA} reduce respectively to
	\begin{align}
	\ul{R}_{\beta 0 \alpha}\,^{\gamma}   & =  0\, , \nonumber \\
	\ul{R}_{\beta \gamma \alpha}\,^{\delta}  & =  \ul{\nabla}^{\delta}{\ul{\Nh}_{\beta \gamma \alpha}} \nonumber  \, , \\
	\ul{\nabla}_{0}{\ul{\Nh}_{\beta \gamma \alpha}}  & = 0  \, , \label{eq:Nab0N}
	\end{align}
	and the second Bianchi identities \eqref{eq:B2_DmR} and \eqref{eq:B2_D0R} become
	\begin{align}
	& 2 \ul{\nabla}{}_{[\delta} {\ul{R}{}_{\epsilon] \bar{\gamma} \alpha}\,^{\beta}} 
	+ \ul{\nabla}_{\bar{\gamma}} \ul{\nabla}^{\beta}{\ul{\Nh}_{\delta \epsilon \alpha}}
	+ \ul{\Nh}_{\delta \epsilon}{}^{\bar{\phi}}  \ul{\nabla}_{\alpha} \ul{\Nh}_{\bar{\phi} \bar{\gamma}}{}^{\beta}  
	= 0  \, , \label{eq:NabR}  \\
	& \ul{\nabla}_{0}{\ul{R}{}_{\epsilon}\,^{\gamma}\,_{\alpha}\,^{\beta}}  = 0 \, ,  \label{eq:Nab0R} 
	\end{align}
	respectively. Combining \eqref{eq:nabLam0} with \eqref{eq:Nab0N}  and \eqref{eq:Nab0R}  clearly yields $\ul{\nabla}_{0} \ul{\Lambda} = 0$.
	
	Now, we trace \eqref{eq:NabR} over $\alpha$ and $\beta$, and over $\bar{\gamma}$ and $\epsilon$ to find
	\begin{align*}
	\ul{\nabla}_{\beta}\ul{\Ric}_{\alpha}{}^{\beta}  =\ul{\nabla}_{\alpha} {\Sc}  + \ul{\nabla}^{\beta}{\ul{\nabla}^{\gamma}\ul{\Nh}_{\alpha \beta \gamma}} 
	- \ul{\Nh}_{\alpha \beta \gamma} \ul{\nabla}_{\delta}\ul{\Nh}^{\beta \gamma \delta} 
	+ \ul{\nabla}_{\beta}\left( \ul{\Nh}_{\gamma \delta \alpha} \ul{\Nh}^{ \gamma \delta \beta}  \right) 
	-  \ul{\nabla}_{\beta}\left( \ul{\Nh}_{\alpha \gamma \delta} \ul{\Nh}^{\beta \gamma \delta} \right)
	\, .
	\end{align*}
	After plugging this expression into \eqref{eq:nabLam1} and a number of tensorial manipulations, using \eqref{eq:CR-E-NabN} in particular, we arrive at
	\begin{multline*}
	\ul{\nabla}_{\alpha} \ul{\Lambda}  = \frac{1}{1 - m} \left( 
	\ul{\nabla}^{\beta}{\ul{\nabla}^{\gamma}\ul{\Nh}_{\alpha \beta \gamma}} 
	+ \left( \ul{\nabla}_{\alpha} \ul{\Nh}_{\beta \gamma \delta} \right) \ul{\Nh}^{\beta \gamma \delta}  
	+ \left( \ul{\nabla}_{\alpha} \ul{\Nh}^{\beta \gamma \delta} \right)  \ul{\Nh}_{\beta \gamma \delta} \right. \\
	\left. 
	+ \left( \ul{\nabla}_{\delta} \ul{\Nh}_{\beta \gamma \alpha} \right) \ul{\Nh}^{\beta \gamma  \delta}   
	+ \left( \ul{\nabla}_{\delta}  \ul{\Nh}^{\beta \gamma  \delta}  \right)  \ul{\Nh}_{\beta \gamma \alpha}
	-  2 \left( \ul{\nabla}_{\delta} \ul{\Nh}_{\alpha \beta \gamma} \right) \ul{\Nh}^{\delta \beta \gamma}   
	\right) \, .
	\end{multline*}
	To show that the RHS of this equality vanishes, we shall need the identity
	\begin{align*}
	\ul{\nabla}^{\beta}{\ul{\nabla}^{\gamma}\ul{\Nh}_{\alpha \beta \gamma}} 
	& = 
	- \ul{\Nh}^{\beta \gamma \delta} \ul{\nabla}_{\delta} \ul{\Nh}_{\beta \gamma \alpha} 
	- \ul{\Nh}_{ \beta \gamma \delta}  \ul{\nabla}{}_{\alpha} \ul{\Nh}^{\beta \gamma \delta}
	- \ul{\Nh}_{ \beta \gamma \alpha} \ul{\nabla}{}_{\delta} \ul{\Nh}^{\beta \gamma \delta} \, ,
	\end{align*}
	which can be obtained by commuting the covariant derivatives and using our assumptions \eqref{eq:CR-Einstein} again.
	This allows us to simplify our expression to
	\begin{align*}
	\ul{\nabla}_{\alpha} \ul{\Lambda} &  = \frac{1}{1 - m} \left( 
	\left( \ul{\nabla}_{\alpha} \ul{\Nh}_{\beta \gamma \delta} \right) \ul{\Nh}^{\beta \gamma \delta}  
	-  2 \left( \ul{\nabla}_{\delta} \ul{\Nh}_{\alpha \beta \gamma} \right) \ul{\Nh}^{\delta \beta \gamma}   
	\right) \, .
	\end{align*}
	Renaming the indices and using \eqref{eq:B1_DN=0} eventually leads to $\ul{\nabla}_{\alpha} \ul{\Lambda} = 0$, as required.
\end{proof}

\begin{rem}
	We see at once from \eqref{eq:CR-E-A} that an almost CR--Einstein manifold admits a transverse infinitesimal CR symmetry. If the almost CR structure is integrable, Definition \ref{def:aCR-E} corresponds to the one given in \cite{Cap2008} and \cite{Leitner2007}. In the latter reference, they are referred to as \emph{transversally symmetric pseudo-Einstein spaces}.  Condition \eqref{eq:CR-E-R} alone defines the notion of \emph{pseudo-Einstein structures} \cite{Lee1988}. The geometric interpretation of \eqref{eq:CR-E-NabN} will be given in Section \ref{sec:almost_Kahler}.
\end{rem}

\begin{rem}
	In dimension three, one can still define a Webster--Tanaka connection. Since the Nijenhuis tensor $\ul{\Nh}_{\alpha \beta \gamma}$ and Chern--Moser tensor  $\ul{S}_{\alpha \bar{\gamma} \beta \bar{\delta}}$ do not exist here, CR flatness is equivalent to the vanishing of the fourth-order CR invariant
	\begin{align}\label{eq:Q_curv}
	\ul{Q}_{\alpha \beta} & := \ii \ul{\nabla}_0 \ul{A}_{\alpha \beta} - 2 \ii \ul{\nabla}_{\alpha} \ul{T}_{\beta} + 2 \ul{\Rho}_{\alpha}{}^{\gamma} \ul{A}_{\gamma \beta}  \, , 
	& \mbox{where} & &
	\ul{T}_{\alpha} & := \frac{1}{3} \left( \ul{\nabla}_{\alpha} \ul{\Rho} - \ii \ul{\nabla}^{\gamma} \ul{A}_{\gamma \alpha} \right) \, .
	\end{align}
	One can also use equations \eqref{eq:CR-Einstein} to define the notion of a CR--Einstein manifold, but the second Bianchi identities no longer implies \eqref{eq:CR-E-R} in the sense that $\ul{\Lambda}$ is not necessarily constant. One may then make the additional assumption that the Webster--Ricci tensor is proportional to the Levi form by a constant factor. Such a condition is however too strong since it is equivalent to $(\ul{\mc{M}}, \ul{H}, \ul{J})$ being locally flat as can be gleaned from equation \eqref{eq:Q_curv}.
\end{rem}

\subsection{Relation to almost K\"{a}hler geometry}\label{sec:almost_Kahler}
Recall (see e.g.  \cite{Gray1980}) that an \emph{almost K\"{a}hler manifold} is an almost Hermitian manifold $(\ut{\mc{M}}, \ut{h}, \ut{J})$,  where  $\ut{h}$ is a Riemannian metric and $\ut{J}$ an almost complex structure compatible with $\ut{h}$, such that the Hermitian $2$-form $\ut{\omega} := \ut{h} \circ \ut{J}$ is closed, i.e.\ $\dd \ut{\omega} = 0$. When $\ut{J}$ is integrable, $(\ut{\mc{M}}, \ut{h}, \ut{J})$ is said to be \emph{K\"{a}hler}. An almost K\"{a}hler manifold that is not K\"{a}hler will be referred to as \emph{strictly almost K\"{a}hler}. We take the dimension of $\ut{\mc{M}}$ to be $2m$ with $m>1$, since clearly, for $m=1$, almost K\"{a}hler necessarily implies K\"{a}hler. In abstract index notation, we shall use minuscule Roman letters starting from the middle of the alphabet, i.e.\ $i,j,k, \ldots$, for sections of $T \ut{\mc{M}}$, of its dual, and tensor products thereof. Indices will be lowered and raised by means of $\ut{h}{}_{i j}$ and its inverse $\ut{h}{}^{i j}$.

Denoting by $\ut{{\nabla}}$ the Levi-Civita connection associated to $\ut{{h}}$, the Hermitian $2$-form satisfies \cite{Gray1980}
\begin{align*}
\ut{{\nabla}}{}_i \ut{{\omega}}{}_{j k} & = 2  \ut{{J}}{}_{j}{}^{\ell} \ut{{\Nh}}{}_{\ell k i} \, ,
\end{align*}
where $\ut{{\Nh}}{}_{j k}{}^{i}= \ut{{\Nh}}{}_{[j k]}{}^{i}$ is the Nijenhuis tensor of $\ut{J}$ which, for an almost K\"{a}hler manifold, satisfies $\ut{{\Nh}}{}_{[j k i]}=0$ and $\ut{{J}}{}_{(i}{}^{\ell} \ut{{\Nh}}{}_{j) \ell k} = 0$. This is the obstruction to the integrability of $\ut{J}$. Thus, when $\ut{J}$ is integrable, we have $\ut{\nabla}_{i} \ut{{\omega}}{}_{j k}=0$.

The complexified tangent bundle $T \ut{\mc{M}}$ splits as ${}^\C T \ut{\mc{M}} = T^{(1,0)} \ut{\mc{M}} \oplus  T^{(0,1)} \ut{\mc{M}}$, and in line with our previous notation, we shall use minuscule Greek letters for sections of  $T^{(1,0)} \ut{\mc{M}}$, and their barred analogues for sections of  $T^{(0,1)} \ut{\mc{M}}$, e.g.\ $\ut{v}^{\alpha} \in \Gamma( T^{(1,0)} \ut{\mc{M}} )$, and $\ut{w}^{\bar{\alpha}} \in \Gamma( T^{(0,1)} \ut{\mc{M}} )$, and similarly for their duals.  In this notation, the only non-vanishing complex components of $\ut{{\Nh}}{}_{i j k}$ are $\ut{{\Nh}}{}_{\alpha \beta \gamma}$ and $\ut{{\Nh}}{}_{\bar{\alpha} \bar{\beta} \bar{\gamma}}$.

\begin{lem}\label{lem:can_conn_aK}
	An almost K\"{a}hler manifold $(\ut{\mc{M}}, \ut{h}, \ut{J})$ admits a unique linear connection $\mathring{\ut{\nabla}}$ that preserves both $\ut{h}$ and $\ut{J}$ and has torsion given by
	\begin{align*}
	2 \mathring{ \ut{\nabla}}{}_{[i} \mathring{ \ut{\nabla}}{}_{j]} \ut{f} & = \ut{\Nh}{}_{i j}{}^{k} \mathring{ \ut{\nabla}}{}_{k} \ut{f} \, , & \mbox{for any smooth function $\ut{f}$ on $\ut{\mc{M}}$.}
	\end{align*}
	Its relation to the Levi-Civita connection $\ut{\nabla}$ is given by
	\begin{align}\label{eq:LC-can_conn}
	\mathring{ \ut{\nabla}}{}_{i} \ut{\alpha}{}_{j} & = \ut{\nabla}{}_{i} \ut{\alpha}{}_{j} - \ut{\alpha}{}^k \ut{\Nh}{}_{k i j} \, , & \mbox{for any $1$-form $\ut{\alpha}_{j}$.}
	\end{align}
\end{lem}

\begin{proof}
	The required properties of $\mathring{ \ut{\nabla}}$ follow from the ansatz \eqref{eq:LC-can_conn} and the fact that $\ut{\Nh}{}_{(i j) k} =0$ and $\ut{{J}}{}_{(j}{}^{\ell} \ut{{\Nh}}{}_{k) \ell i} = 0$. Uniqueness can easily be proved as in the case of the Levi-Civita connection.
\end{proof}

The connections $\ut{\nabla}$ and $\mathring{ \ut{\nabla}}$ clearly coincide if and only if $\ut{J}$ is integrable. Their relation can be expressed by writing the structure equations explicitly in terms of a unitary frame $\{ \ut{\theta}^{\alpha} \}$:
\begin{align}\label{eq:str_aK}
\dd \ut{\theta}^{\alpha} & =    \ut{\theta}^{\beta}  \wedge \mathring{ \ut{\Gamma}}{}_{\beta}{}^{\alpha} - \frac{1}{2} \ut{\Nh}_{\bar{\beta} \bar{\gamma}}{}^{\alpha}\overline{\ut{\theta}}{}^{\bar{\beta}} \wedge \overline{\ut{\theta}}{}^{\bar{\gamma}} = -  \ut{\Gamma}{}^{\alpha}{}_{\beta} \wedge \ut{\theta}{}^{\beta}  -  \ut{\Gamma}{}^{\alpha}{}_{\bar{\beta}} \wedge \overline{\ut{\theta}}{}^{\bar{\beta}}  \, ,
\end{align}
and similarly for its complex conjugate. Here, $\ut{\Gamma}{}^{\alpha}{}_{\beta}$ and $\mathring{ \ut{\Gamma}}{}_{\beta}{}^{\alpha}$ are the respective connection $1$-forms of $\ut{\nabla}$ and $\mathring{\ut{\nabla}}$ with respect to $\{ \ut{\theta}^{\alpha} \}$. A cursory comparison of \eqref{eq:str_aK} with \eqref{eq:structure_CR_A}, with $\ul{A}_{\bar{\alpha} \bar{\beta}}=0$, reveals that the connection $\mathring{ \ut{\nabla}}$ is closely related to the Webster--Tanaka connection, and its properties, such as the Bianchi identities, mirror those of the Webster--Tanaka connection.  This analogy justifies the discrepancy in the choice of staggering of indices between $\ut{\Gamma}{}^{\alpha}{}_{\beta}$ and $\mathring{ \ut{\Gamma}}{}_{\beta}{}^{\alpha}$. For the same concern of convention, the respective curvature tensors of $\ut{{\nabla}}$ and $\mathring{ \ut{\nabla}}$ will be defined by
\begin{align*}
2 \ut{\nabla}{}_{[i} \ut{\nabla}{}_{j]} \ut{\alpha}_k & =:  \ut{R}_{i j k}{}^{\ell} \ut{\alpha}_{\ell} \, , &
2\mathring{ \ut{\nabla}}{}_{[i}\mathring{ \ut{\nabla}}{}_{j]} \ut{\alpha}_k & =:  - \mathring{ \ut{R}}{}_{i j k}{}^{\ell} \ut{\alpha}_{\ell} \, ,
\end{align*}
for any $1$-form $\ut{\alpha}_i$. The complex components of $\ut{R}_{i j k \ell}$ can then be expressed as
\begin{align*}
\ut{{R}}{}_{\alpha \beta \gamma \delta} & = 2\mathring{ \ut{\nabla}}{}_{[\gamma|} \ut{\Nh}{}_{\alpha \beta |\delta]}\, , &
\ut{{R}}{}_{\gamma \delta}{}^{\beta}{}_{\alpha} & =\mathring{ \ut{\nabla}}{}^{\beta} \ut{\Nh}_{\gamma \delta \alpha}\, , &
\ut{{R}}{}_{\gamma}{}^{\delta \beta}{}_{\alpha} & =  \mathring{ \ut{R}}{}_{\gamma}{}^{\delta}{}_{\alpha}{}^{\beta}  
- \ut{\Nh}^{\beta \epsilon \delta}  \ut{\Nh}_{\epsilon \alpha \gamma}  \, ,
\end{align*}
and from the last equation, we deduce $\ut{{R}}^{\gamma \delta}{}_{\alpha \beta} = - \ut{\Nh}{}_{\alpha \beta \epsilon} \ut{\Nh}^{\gamma \delta \epsilon}$. The Ricci tensor of $\ut{\nabla}$ is defined to be $\ut{{\Ric}}_{i j} := \ut{{R}}_{i k j}{}^{k}$ as is conventional.  The complex components of the Ricci tensor are given by
\begin{align}\label{eq:Ric_LC2can}
\ut{{\Ric}}_{\alpha \beta} & = 2   \ut{\nabla}^{\gamma}{\ut{\Nh}_{\gamma (\alpha \beta)}}  \, , &
\ut{{\Ric}}_{\alpha}{}^{\beta} &  =  \mathring{\ut{\Ric}}{}_{\alpha}\,^{\beta} - \ut{\Nh}_{\alpha \delta \gamma} \ut{\Nh}^{\beta \delta \gamma} 
\, ,
\end{align}
where we have defined
\begin{align*}
\mathring{ \ut{\Ric}}{}_{\alpha \bar{\beta}} & := \mathring{ \ut{R}}{}_{\gamma}{}^{\gamma}{}_{\alpha \bar{\beta}}  \, .
\end{align*}
This choice of definition is analogous to the definition of the Webster--Ricci tensor, and will prove judicious in the light of Corollary \ref{cor:aCR2aK} below. The next result follows directly from identities \eqref{eq:Ric_LC2can}.

\begin{prop}\label{prop:aK-E}
	Let $(\ut{\mc{M}}, \ut{h}, \ut{J})$ be an almost K\"{a}hler manifold. Denote by $\ut{\nabla}$ the Levi-Civita connection of $\ut{h}$ and by $\mathring{\ut{\nabla}}$ the compatible linear connection of Lemma \ref{lem:can_conn_aK}.
	\begin{enumerate}
		\item The Ricci tensor of $\ut{\nabla}$ commutes with $\ut{J}$, i.e.\ $\ut{J}_{(i}{}^{k} \ut{\Ric}_{j) k} = 0$, if and only if $\mathring{\ut{\nabla}}^{\gamma}{\ut{\Nh}_{\gamma (\alpha \beta)}}  = 0$.
		\item The metric $\ut{h}$ is Einstein, i.e.\
		\begin{align}\label{eq:aK-Einstein}
		\ut{\Ric}_{i j} & = \ut{\Lambda} \ut{h}_{i j} \, , & \mbox{with constant $\ut{\Lambda}$,}
		\end{align}
		if and only if
		\begin{align}\label{eq:aK-Einstein-comp}
		\mathring{\ut{\nabla}}^{\gamma}{\ut{\Nh}_{\gamma (\alpha \beta)}}  & = 0 \, , &
		\mathring{\ut{\Ric}}_{\alpha \bar{\beta}} - \ut{\Nh}_{\alpha \delta \gamma} \ut{\Nh}{}_{\bar{\beta}}{}^{\delta \gamma} & = \ut{\Lambda} \ut{h}{}_{\alpha \bar{\beta}} \, , & \mbox{with constant $\ut{\Lambda}$.} 
		\end{align}
	\end{enumerate}
\end{prop}

Let us return to a partially integrable contact almost CR manifold $(\ul{\mc{M}},\ul{H},\ul{J})$. We recall here that we assume that the Levi form has positive definite signature. Suppose that $(\ul{\mc{M}},\ul{H},\ul{J})$ admits an infinitesimal transverse symmetry $\ul{e}_0$, i.e.\ $\ul{e}_0$ is the Reeb vector field of some contact form $\ul{\theta}^0$ for which the pseudo-Hermitian torsion tensor vanishes, i.e.\ $\ul{A}_{\alpha \beta} = 0$. Let us denote by $\ut{\mc{M}}$ the (local) leaf space of the corresponding foliation. Then, since $\ul{J}$ and $\dd \ul{\theta}^0$ are preserved along the flow of $\ul{e}_0$, they descend to $\ut{\mc{M}}$ endowing it with an almost K\"{a}hler structure $(\ut{h}, \ut{J})$ -- under the assumption of integrability, this is already proved in \cite{Leitner2007}. In a nutshell:
\begin{prop}\label{prop:aCR2aK}
	Let $(\ul{\mc{M}},\ul{H},\ul{J})$ be a partially integrable contact almost CR manifold that admits an infinitesimal transverse symmetry $\ul{e}_0$. Then $(\ul{H},\ul{J})$ induces an almost K\"{a}hler structure $(\ut{h},\ut{J})$ on the (local) leaf space $\ut{\mc{M}}$ of the foliation defined by $\ul{e}_0$. Further $\ul{J}$ is integrable if and only if $\ut{J}$ is.
\end{prop}

In fact, the Webster--Tanaka connection $\ul{\nabla}$ descends to the compatible connection $\mathring{ \ut{\nabla}}$ on  $(\ut{\mc{M}}, \ut{h}, \ut{J})$ defined in Lemma \ref{lem:can_conn_aK}. In particular, the restriction of the curvature of $\ul{\nabla}$ to projectable vector fields can be identified with the curvature of $\mathring{ \ut{\nabla}}$. Hence, with reference to Proposition \ref{prop:CR-Einstein_Ric} and Proposition \ref{prop:aK-E} -- see in particular, equations \eqref{eq:aK-Einstein-comp} -- we obtain as a corollary of Proposition \ref{prop:aCR2aK}:

\begin{cor}\label{cor:aCR2aK}
	Let $(\ul{\mc{M}},\ul{H},\ul{J})$ be a partially integrable contact almost CR manifold that admits an infinitesimal transverse symmetry $\ul{e}_0$. Let $\ul{\theta}^0$ be the contact form dual to $\ul{e}_0$ and $\ul{\nabla}$ its corresponding Webster--Tanaka connection. Denote by $(\ut{\mc{M}}, \ut{h}, \ut{J})$ the almost K\"{a}hler (local) leaf space of the foliation defined by $\ul{e}_0$.
	\begin{enumerate}
		\item The Ricci tensor of the Levi-Civita connection $\ut{\nabla}$ of $\ut{h}$ commutes with $\ut{J}$, i.e.\ $\ut{J}_{(i}{}^{k} \ut{\Ric}_{j) k} = 0$, if and only if $\ul{\nabla}^{\gamma}\ul{\Nh}_{\gamma (\alpha \beta)} = 0$.
		\item The metric $\ut{h}$ is Einstein, i.e.\ equation \eqref{eq:aK-Einstein} holds, if and only if  $\ul{\theta}^0$ defines an almost CR--Einstein structure on $(\ul{\mc{M}},\ul{H},\ul{J})$, i.e.\ equations \eqref{eq:CR-E2} hold with $\ul{\Lambda} = \ut{\Lambda}$.
	\end{enumerate}
\end{cor}
This corollary is also given in \cite{Leitner2007} for CR--Einstein structures and K\"{a}hler--Einstein manifolds.

We shall now provide a converse of Corollary \ref{cor:aCR2aK} using a modification of the construction given in the integrable case in \cite{Leitner2007}. Let $(\ut{\mc{M}},\ut{h},\ut{J})$ be a $2m$-dimensional almost K\"{a}hler manifold, $\ut{\mc{F}}$ its $\mathbf{U}(m)$-frame bundle and $\ul{\mc{M}}$ the total space of the circle bundle associated to the anti-canonical bundle of $(\ut{\mc{M}},\ut{h},\ut{J})$, i.e.\
\begin{align}\label{eq:S1bdl}
\ul{\mc{M}} & := \ut{\mc{F}} \times_{\det} S^1 \accentset{\ul{\varpi}}{\longrightarrow} \ut{\mc{M}} \, ,
\end{align}
where $\det : \mathbf{U}(m) \rightarrow S^1$. The compatible connection $\mathring{ \ut{\nabla}}$ on $\ut{\mc{M}}$ given in Lemma \ref{lem:can_conn_aK} induces  a principal bundle connection  $1$-form $\mathring{\ul{\gamma}}$ on $\ul{\mc{M}}$ with values in $\ii \R$.  We can write $\mathring{ \ul{\gamma}} = \ii \left( \dd t - \ii \ul{\varpi}^* \mathring{ \ut{\Gamma}}_{\gamma}{}^{\gamma} \right)$ where $\mathring{ \ut{\Gamma}}_{\alpha}{}^{\beta}$ is the connection $1$-form on $T \ut{\mc{M}}$ with respect to some unitary frame $\{ \ut{e}_{\alpha} \}$, and $t$ is a coordinate on $S^1$ such that $\ee^{\ii t} \in S^1$. Then, in terms of the dual coframe $\{ \ut{\theta}^{\alpha} \}$,
\begin{align*}
\dd \mathring{ \ul{\gamma}} & =  \ul{\varpi}^* \left(  \left(\mathring{ \ut{\Ric}}{}_{\alpha \bar{\beta}} -  \ut{\Nh}_{\alpha \gamma \delta} \ut{\Nh}_{\bar{\beta}}{}^{\gamma \delta} \right)  \ut{\theta}^{\alpha} \wedge \overline{\ut{\theta}}{}^{\bar{\beta}} 
- \ii \ut{\eta}  \right) \, ,
\end{align*}
where
\begin{align*}
\ut{\eta} := \ii  \left( 2 \, \ut{\Nh}_{\alpha \gamma \delta} \ut{\Nh}_{\bar{\beta}}{}^{\gamma \delta} - \ut{\Nh}_{\gamma \delta \alpha} \ut{\Nh}^{ \gamma \delta}{}_{\bar{\beta}} \right)  \ut{\theta}{}^{\alpha} \wedge \overline{\ut{\theta}}{}^{\bar{\beta}} 
+ \frac{\ii}{2}\mathring{ \ut{\nabla}}{}^{\gamma}\ut{\Nh}_{\alpha \beta \gamma} \ut{\theta}{}^{\alpha} \wedge \ut{\theta}{}^{\beta} -\frac{\ii}{2}\mathring{ \ut{\nabla}}{}^{\bar{\gamma}}\ut{\Nh}_{\bar{\alpha} \bar{\beta} \bar{\gamma}} \overline{\ut{\theta}}{}^{\bar{\alpha}} \wedge \overline{\ut{\theta}}{}^{\bar{\beta}}   \, .
\end{align*}
Suppose now that $\ut{h}$ is Einstein, i.e. \eqref{eq:aK-Einstein} holds. Then
\begin{align*}
\dd \mathring{ \ul{\gamma}} & =  - \ii \ul{\varpi}^*\left( \frac{1}{2} \ut{\Lambda} \ut{\omega} + \ut{\eta} \right) \, ,
\end{align*}
where we recall that $\ut{\omega}$ denotes the Hermitian $2$-form on $(\ut{\mc{M}},\ut{h},\ut{J})$. Since both $\ut{\omega}$ and $\dd \mathring{ \ul{\gamma}}$ are closed, so must $\ut{\eta}$. Hence, locally, we can write $\ut{\omega} = \dd \ut{\alpha}$ and $\ut{\eta} = \dd \ut{\beta}$ for some $1$-forms  $\ut{\alpha}$ and $\ut{\beta}$ on $\ut{{\mc{M}}}$. Define
\begin{subequations}\label{eq:def_contact-E}
	\begin{align}
	\ul{\theta}^0 & := \frac{1}{\ut{\Lambda}} \left( \ii \mathring{ \ul{\gamma}}  -  \ul{\varpi}^*\ut{\beta} \right) \, , & \mbox{if $\ut{\Lambda} \neq 0$,} \label{eq:def_contact-E1}\\
	\ul{\theta}^0 & :=   \ii \mathring{ \ul{\gamma}} -  \ul{\varpi}^*\ut{\beta} + \frac{1}{2}  \ul{\varpi}^*\ut{\alpha} \, , & \mbox{if $\ut{\Lambda}=0$.} \label{eq:def_contact-E2}
	\end{align}
\end{subequations}
In both cases, $\ul{\theta}^0$ is a horizontal $1$-form on $\ul{\mc{M}}$ satisfying $\dd \ul{\theta}^{0} = \frac{1}{2}  \ul{\varpi}^*\ut{\omega}$, i.e.\ $\ul{\theta}^0$ is a contact form. These definitions of $\ul{\theta}^0$ are unique up to the addition of an exact $1$-form $\dd f$, where $f$ is a diffeomorphism on $\ul{\mc{M}}$, such that $\mathsterling_{\ul{v}} f \neq -\ul{\theta}^0 (\ul{v})$ for any smooth vertical vector field $\ul{v}$ on $\ul{\mc{M}}$, i.e.\ $\ul{\varpi}_*\ul{v} = 0$. Further, since the contact distribution $\ul{H}$ is fibrewise isomorphic to $T \ut{\mc{M}}$, it also inherits a bundle complex structure from $\ut{J}$, thereby endowing $\ul{\mc{M}}$ with a partially integrable contact almost CR structure $(\ul{H}, \ul{J})$. By construction, the triple $(\ul{H},\ul{J},\ul{\theta}^0)$ defines an almost CR--Einstein structure on $\ul{\mc{M}}$. In summary:

\begin{prop}\label{prop:aK2aCR}
	Let $(\ut{\mc{M}},\ut{h},\ut{J})$ be a $2m$-dimensional almost K\"{a}hler--Einstein manifold so that equation \eqref{eq:aK-Einstein} holds. Denote by $\ut{\mc{F}}$ its  $\mathbf{U}(m)$-frame bundle. Then the associated circle bundle $\ul{\mc{M}} := \ut{\mc{F}} \times_{\det} S^1 \accentset{\ul{\varpi}}{\longrightarrow}  \ut{\mc{M}}$ inherits an almost CR--Einstein structure $(\ul{H},\ul{J},\ul{\theta}^0)$ from 
	$(\ut{\mc{M}},\ut{h},\ut{J})$ with $\ul{\theta}^0$ given by \eqref{eq:def_contact-E}, i.e.\ the Webster--Tanaka connection satisfies equations \eqref{eq:CR-E2} with $\ul{\Lambda} = \ut{\Lambda}$.
\end{prop}

\begin{rem}
	In the integrable case, a similar construction is given in \cite{Alekseevsky2021} where the resulting CR--Einstein manifold is described as a \emph{Sasaki} manifold over a \emph{quantizable} K\"{a}hler--Einstein manifold. Here, `quantizable' means that the K\"{a}hler manifold admits a principal circle or line bundle together with a connection $1$-form whose exterior derivative is the pullback of the K\"{a}hler form. Such a definition can also be extended to the non-integrable case, and in fact, leaving global considerations aside, any almost K\"{a}hler--Einstein manifold is (locally) quantizable by Proposition \ref{prop:aK2aCR}. An almost CR--Einstein manifold can also be constructed more simply as a trivial line bundle over an almost K\"{a}hler--Einstein manifold.
\end{rem}

The offshoot is that, at least from a local perspective, one can construct any almost CR--Einstein manifold as a circle bundle over an almost K\"{a}hler--Einstein manifold. In the integrable case, the existence of K\"{a}hler--Einstein manifolds is well-established -- see for instance \cite{Besse1987} and references therein. The non-integrable case is somewhat more problematic. The Goldberg conjecture \cite{Goldberg1969} states that any \emph{compact} almost K\"{a}hler--Einstein manifold is necessarily K\"{a}hler, a conjecture proved to be correct when the scalar curvature is non-negative \cite{Sekigawa1987}. However, \emph{non-compact} strictly almost K\"{a}her--Einstein manifolds do exist. In dimension four, Nurowski and Przanowski constructed a Ricci-flat example in \cite{Nurowski1999}. Their method, generalised by Tod, was used to characterise certain families of Ricci-flat strictly almost K\"{a}hler--Einstein manifolds in \cite{Armstrong2002} -- see also \cite{Apostolov2003}. References \cite{Alekseevskiui1970,Apostolov2001} provide examples in any even dimensions, which are not necessarily Ricci-flat. Any of these manifolds can be used to produce \emph{non-integrable} almost CR--Einstein manifolds by applying Proposition \ref{prop:aK2aCR}.

\section{Almost Robinson structures}\label{sec-Robinson}
Further details on the content of this section can be found in \cite{Nurowski2002,Trautman2002,Trautman2002a,Fino2021}. Let $(\mc{M},\mbf{c})$ be a time-oriented and oriented Lorentzian conformal manifold of dimension $2m+2$. An \emph{almost Robinson structure} on $(\mc{M},\mbf{c})$ consists of a pair $(N,K)$ where $N$ is a complex distribution of rank $m+1$, totally null with respect to $\mbf{c}$, and $K$ is a real null line distribution such that ${}^\C K = N \cap \overline{N}$.  When $N$ is involutive or integrable,\footnote{Again, no difference will be made between the two terms in this article.} i.e.\ $[ N , N] \subset N$, we say that $(N,K)$ is a \emph{Robinson structure}. The quadruple $(\mc{M},\mbf{c},N,K)$ is referred to as an \emph{almost Robinson manifold} or \emph{almost Robinson geometry}, and as a \emph{Robinson manifold} or \emph{Robinson geometry} when $N$ is integrable.

The complex distribution $N$ in the definition above is also referred to as an \emph{almost null structure} \cite{Taghavi-Chabert2016}. It is said to have (regular) \emph{real index} one: at every point $p$ of $\mc{M}$, the dimension of the real span of $N_p \cap \overline{N}_p$ is one \cite{Kopczy'nski1992}.

One can equivalently describe an almost Robinson structure as an optical structure whose screen bundle $H_K = K^\perp/K$ is equipped with a bundle complex structure $J$ compatible with the screen bundle conformal structure. Here, we identify the eigenbundles $H_K^{(1,0)}$ and $H_K^{(0,1)}$ of $J$ with the subbundles $\overline{N}/{}^\C K$ and $N/{}^\C K$ of  ${}^\C H_K$.  Under certain conditions, $(N,K)$ induces an almost CR structure $(\ul{H}, \ul{J})$ on the (local) leaf space of $\mc{K}$, and one can show that the involutivity of $(N,K)$ is equivalent to the involutivity of $(\ul{H}, \ul{J})$.

An almost Robinson structure clearly defines an optical structure $K$ and a congruence of null curves $\mc{K}$ associated to it. But in general, not every optical geometry is endowed with a distinguished almost Robinson structure. There is however one exception that is particularly relevant to the present article: let $(\mc{M}, \mbf{c}, K)$ be an optical geometry of dimension $2m+2$ with twisting congruence of oriented null geodesics $\mc{K}$. Let $k$ be an optical vector field, and suppose that its twist satisfies
\begin{align}\label{eq-F^2}
\bm{\tau}_{i k} \bm{\tau}^{k}{}_{j} +\frac{1}{2m} \bm{\tau}_{k \ell} \bm{\tau}^{k \ell} \bm{h}_{i j} & = 0 \, .
\end{align}
Then we can always find an optical vector field whose twist  defines a bundle complex structure $J$ on the screen bundle $H_K$, compatible with the conformal structure there \cite{Fino2020,Fino2021}. In other words, the twist induces an almost Robinson structure on $(\mc{M},\mbf{c})$. We shall refer to an almost Robinson structure arising in this way as a \emph{twist-induced almost Robinson structure}.

As shown in \cite{Fino2020,Fino2021}, a twist-induced almost Robinson structure with non-shearing congruence of null geodesics $\mc{K}$ induces a partially integrable contact almost CR structure $(\ul{H}, \ul{J})$ on the leaf space $\ul{\mc{M}}$ of $\mc{K}$. Further, locally there is a one-to-one correspondence between metrics in $\accentset{n.e.}{\mbf{c}}$ and contact forms in $\Ann(\ul{H})$. More specifically, there is a unique optical vector field $k$ whose twist is normalised to $2m$ for any metric $g$ in $\accentset{n.e.}{\mbf{c}}$. In particular, for every $g$ in $\accentset{n.e.}{\mbf{c}}$, the $1$-form $\kappa = g(k, \cdot)$ satisfies $\mathsterling_k \kappa = 0$, and descends to a contact form $2 \, \ul{\theta}^0$ on $(\ul{\mc{M}}, \ul{H},\ul{J})$ -- the factor of $2$ has been added for later convenience. We can choose a coframe $\{ \ul{\theta}{}^0 , \ul{\theta}{}^{\alpha}, \overline{\ul{\theta}}{}^{\bar{\alpha}} \}$ adapted to $(\ul{H}, \ul{J})$, which we then pull back to $\mc{M}$. Similarly, the Levi form $\ul{h}_{\alpha \bar{\beta}}$ can be pulled back to $\mc{M}$, and can be identified with the screen bundle metric induced from $g$. Denoting by $\varpi$ the local surjective submersion from $\mc{M}$ to $\ul{\mc{M}}$, we can then express the metric $g$ as
\begin{align}\label{eq:g_can}
g & = 2 \, \theta^0 \odot \theta_0 + 2 \, h_{\alpha \bar{\beta}} \theta{}^\alpha \odot \overline{\theta}{}^{\bar{\beta}} \, ,
\end{align}
where 
\begin{align*}
\kappa = \theta^0 & = 2 \, \varpi^* \ul{\theta}^0 \, , & \theta^{\alpha} & = \varpi^* \ul{\theta}^{\alpha} \, , & \overline{\theta}{}^{\bar{\alpha}} & = \varpi^* \overline{\ul{\theta}}{}^{\bar{\alpha}} \, , & h_{\alpha \bar{\beta}} & =  \varpi^*  \ul{h}_{\alpha \bar{\beta}} \, .
\end{align*}
and the $1$-form $\lambda = \theta_{0}$ on $\mc{M}$ is uniquely determined by $g$ and any adapted frame for a given contact form $\ul{\theta}^0$. Its exterior derivative
\begin{multline}\label{eq:dtheta_0}
\dd \theta_{0} =  - B_{\alpha \beta} \theta^{\alpha} \wedge \theta^{\beta} - 2 \,  B_{\alpha \bar{\beta}} \theta^{\alpha} \wedge \overline{\theta}{}^{\bar{\beta}} - B_{\bar{\alpha} \bar{\beta}} \overline{\theta}{}^{\bar{\alpha}} \wedge \overline{\theta}{}^{\bar{\beta}} \\
- C_{\alpha}  \theta^{\alpha} \wedge \theta^{0} - C_{\bar{\alpha}} \overline{\theta}{}^{\bar{\alpha}} \wedge  \theta^{0}    - 2 \, E_{\alpha}  \theta^{\alpha} \wedge \theta_{0} - 2 \, E_{\bar{\alpha}}  \bar{\theta}{}^{\bar{\alpha}} \wedge \theta_{0} - E_0 \theta^0 \wedge \theta_0 \, ,
\end{multline}
defines smooth functions $B_{\alpha \beta}$, $B_{\alpha \bar{\beta}}$, $C_{\alpha}$, $E_{\alpha}$ and $E_0$, and their complex conjugates when relevant. Note that $E_0$ is real and $\overline{B_{\alpha \bar{\beta}}} = - B_{\alpha \bar{\beta}}$. Taken together, equations \eqref{eq:dtheta_0} and \eqref{eq:structure_CR} form the Cartan structure equations for $g$, which in turn determine the connection $1$-form of the Levi-Civita connection $\nabla$ uniquely. Explicitly, and dropping $\varpi^*$ for clarity, we have
\begin{subequations}\label{eq:cov_der_frame}
	\begin{align}
	\nabla \theta^0 & = 2  \ii \,   \ul{h}_{\alpha \bar{\beta}} \theta^{\alpha} \wedge \overline{\theta}{}^{\bar{\beta}} + 2 \, E_{\alpha} \theta^{\alpha}  \odot \theta^0 + 2 \,  E_{\bar{\alpha}} \overline{\theta}{}^{\bar{\alpha}} \odot \theta^0 + E_0 \theta^{0} \otimes \theta^0  \, , \\
	\nabla \theta^\alpha & = \ul{\nabla} \theta^\alpha -   \frac{1}{2} \ul{A}^{\alpha}{}_{\bar{\beta}}  \overline{\theta}{}^{\bar{\beta}}  \otimes \theta^0 - \ul{\Nh}^{\alpha}{}_{\bar{\beta} \bar{\gamma}} \overline{\theta}{}^{\bar{\gamma}}  \otimes \overline{\theta}{}^{\bar{\beta}} - 2 \ii \,  \theta_{0}  \odot \theta^{\alpha} \nonumber  \\
	& \qquad + 2 \, B_{\beta}{}^{\alpha} \theta^{0} \odot \theta^{\beta}   +  2 \, B_{\bar{\beta}}{}^{\alpha } \theta^{0}  \odot \overline{\theta}{}^{\bar{\beta}}   - 2 \,  E^{\alpha} \theta_{0}  \odot \theta^0  -  C^{\alpha} \theta^{0}  \otimes \theta^0 \, , \\
	\nabla \theta_0 & =   \frac{1}{2}  \ul{A}_{\alpha \beta}  \theta^{\alpha} \odot  \theta^{\beta}  + \frac{1}{2}  \ul{A}_{\bar{\alpha} \bar{\beta}}  \overline{\theta}{}^{\bar{\alpha}} \odot \overline{\theta}{}^{\bar{\beta}} 
	  - 2 \, B_{\alpha \bar{\beta}}  \theta^{\alpha} \wedge \overline{\theta}{}^{\bar{\beta}} - B_{\alpha \beta}  \theta^{\alpha} \wedge  \theta^{\beta} - B_{\bar{\alpha} \bar{\beta}} \overline{\theta}{}^{\bar{\alpha}}\wedge \overline{\theta}{}^{\bar{\beta}} \nonumber \\
	& \qquad  - 2 \, E_{\alpha} \theta^{\alpha} \wedge \theta_0 - 2 \, E_{\bar{\alpha}} \overline{\theta}{}^{\bar{\alpha}} \wedge \theta_0 - E_0 \theta^{0} \otimes \theta_0 +  C_{\alpha} \theta^{0} \otimes \theta^{\alpha} +  C_{\bar{\alpha}} \theta^{0} \otimes \overline{\theta}{}^{\bar{\alpha}} \, .
	\end{align}
\end{subequations}
The curvature tensors of $\nabla$ are computed in Appendix \ref{app:computations} as will be needed in the next sections.

Now let $\{ e_0 , e_{\alpha}, \overline{e}_{\bar{\alpha}} , e^0 \}$ be the frame on $\mc{M}$ dual to $\{ \theta^0, \theta^{\alpha} , \overline{\theta}{}^{\bar{\alpha}}, \theta_0 \}$. In relation to our previous notation, $e^0 = k$ and $e_0 = \ell$. Choose a local affine parameter $\phi$ for the geodesics of $\mc{K}$ so that
\begin{align*}
k & = \parderv{}{\phi} \, , & \lambda = \theta_{0} & = \dd \phi + \lambda_{\alpha} \ul{\theta}^{\alpha} + \lambda_{\bar{\alpha}} \overline{\ul{\theta}}{}^{\bar{\alpha}} + \lambda_0 \ul{\theta}^{0} \, ,
\end{align*}
for some smooth functions $\lambda_0$, $\lambda_{\alpha}$ and $\lambda_{\bar{\alpha}}$ on $\mc{M}$. The second of these equations follows from the fact that $\lambda(k) = 1$. Throughout this article, we shall write $\dot{f} := \mathsterling_k f$, $\ddot{f}  := \mathsterling_k \mathsterling_k f$ and so on, for any smooth function $f$. This notation will be extended to tensor fields annihilated by $k$. Now, taking the exterior derivative of $\lambda$, we deduce that
\begin{align*}
B_{\alpha \beta} & = - \ul{\nabla}_{[\alpha} \lambda_{\beta]} + \lambda_{[\alpha} \dot{\lambda}_{\beta]} + \frac{1}{2} \ul{\Nh}_{\alpha \beta \gamma} \lambda^{\gamma} \, , \\
B_{\alpha \bar{\beta}} & = - \frac{1}{2} \left( \ul{\nabla}_{\alpha} \lambda_{\bar{\beta}} - \ul{\nabla}_{\bar{\beta}} \lambda_{\alpha} -\lambda_{\alpha} \dot{\lambda}_{\bar{\beta}} + \lambda_{\bar{\beta}} \dot{\lambda}_{\alpha} + \ii \lambda_0 \ul{h}_{\alpha \bar{\beta}} \right)  \, , \\
C_{\alpha} & = -\frac{1}{2} \left( \ul{\nabla}_{\alpha} \lambda_0 - \ul{\nabla}_{0} \lambda_{\alpha} - \lambda_{\alpha} \dot{\lambda}_0  + \lambda_{0} \dot{\lambda}_{\alpha} -  \ul{A}_{\alpha \beta} \lambda^{\beta} \right) \, , \\
E_{\alpha} & = \frac{1}{2} \dot{\lambda}_{\alpha} \, , \\
E_0 & = \frac{1}{2} \dot{\lambda}_{0} \, .
\end{align*}
We also note that locally, with this choice of frame, sections of the eigenbundles $\ul{H}^{(1,0)}$ and $\ul{H}^{(0,1)}$ of $\ul{J}$ can be pulled back to sections of the bundles $\overline{N}/{}^\C K$ and $N/{}^\C K$ respectively. In particular,
\begin{align*}
e_0 & = \frac{1}{2}\left( \ul{e}_0 - \lambda_0 e^0 \right) \, , & e_{\alpha} & = \ul{e}_{\alpha} - \lambda_{\alpha} e^0 \, , & e_{\bar{\alpha}} & =\ul{e}_{\bar{\alpha}} - \lambda_{\bar{\alpha}} e^0 \, .
\end{align*}
With our conventions, sections of $\overline{N}/{}^\C K$ and $N/{}^\C K$ will be adorned with minuscule Greek indices, plain and barred, i.e.\ $v^\alpha$ and $w^{\bar{\alpha}}$ respectively, and similarly for their duals.

\section{Optical geometries, almost Robinson structures and almost CR structures}\label{sec-curv-cond}
We are now in a position to return to our study of non-shearing congruences of null geodesics. We assume $n=2m$ with $m>1$. For clarity, we project the integrability condition \eqref{eq-int-cond-non-shearing} onto the screen bundle:
\begin{align}\label{eq-int-cond-non-shearing-pr}
W^{0}{}_{i j}{}^{0} = \bm{\tau}_{i k} \bm{\tau}^{k}{}_{j} + \frac{1}{2m} \bm{\tau}_{k l} \bm{\tau}^{kl} \bm{h}_{i j} \, .
\end{align}
With reference to the algebraic constraint \eqref{eq-F^2} on the twist,  we conclude that the LHS of \eqref{eq-int-cond-non-shearing-pr} vanishes, i.e.\ $W^{0}{}_{i}{}^{0}{}_{j} =0$, if and only if the twist of $\mc{K}$ induces an almost Robinson structure. Therefore, following the discussion of Section \ref{sec-Robinson} on the relation between almost CR structures and almost Robinson structures, we immediately arrive at the first main result, Theorem \ref{thm-int-cond2m}, where we note that the curvature condition \eqref{eq-weakestC} is equivalent to
\begin{align*}
\bm{\kappa}_{[a} W_{b]ef[c}  \bm{\kappa}_{d]} k^e k^f & = 0 \, , & \mbox{for any optical vector field $k$, with $\bm{\kappa} = \bm{g}(k, \cdot)$.}
\end{align*}
If any of the equivalent conditions of Theorem \ref{thm-int-cond2m} is satisfied, we can then choose a metric $g$ in $\accentset{n.e.}{\mbf{c}}$, cast it in the form \eqref{eq:g_can}, and use the computation of the curvature tensors for $g$ given in Appendix \ref{app:computations}. This will be assumed throughout the remaining of the paper. Using the index notation introduced earlier accordingly, we shall give further results that relate the degeneracy of the Weyl tensor to the invariant properties of the almost CR structure on the leaf space. Note that none of the statements will depend on the choice of frame for $(N,K)$ as can be checked using the results of \cite{Taghavi-Chabert2014}.

\begin{prop}\label{prop:tw-in-Rob-NS-deg_W}
	Let $(\mc{M},\mbf{c},N,K)$ be a $(2m+2)$-dimensional twist-induced almost Robinson geometry, where $m>1$, with non-shearing congruence of null geodesics $\mc{K}$. Then the Weyl tensor satisfies
	\begin{subequations}\label{eq:NSCNG-AR0}
		\begin{align}
		W^{0}{}_{\alpha}{}^{0}{}_{\beta} & = 0  \, , \label{eq:NSCNG-AR}\\ 
		W^{0}{}_{\alpha}{}^{0}{}_{\bar{\beta}} & =0 \, , \label{eq:NSCNG-AR2} \\
		\left( W^{0}{}_{\alpha \beta \bar{\gamma}} \right)_\circ & = 0 \, . \label{eq:NSCNG-AR-x}
		\end{align}
	\end{subequations}
\end{prop}

\begin{proof}
	Conditions \eqref{eq:NSCNG-AR} and \eqref{eq:NSCNG-AR2} follow from Theorem \ref{thm-int-cond2m} -- see also equations \eqref{eq:R-Nsh} and \eqref{eq:R-Nsh2}. For condition \eqref{eq:NSCNG-AR-x}, we refer to equation \eqref{eq:R-E}.
\end{proof}

In addition, with reference to equations \eqref{eq:R-T}, \eqref{eq:R-DT}, \eqref{eq:R-DbT} and \eqref{eq:R-ST}, and Section \ref{sec:WT_conn}, we find
\begin{subequations}\label{eq:Weyl-CR}
	\begin{align}
	W^0{}_{\alpha \beta \gamma} & = - 2 \ii \ul{\Nh}_{\beta \gamma \alpha} \, , \label{eq:W-N}
	\\
	W_{\alpha \beta \gamma \delta} & = 2 \, \ul{\nabla}_{[\alpha|} \ul{\Nh}_{\gamma \delta |\beta]} \, , \label{eq:W-DN}
	\\
	\left( W_{\bar{\alpha} \beta \gamma \delta} \right)_\circ & =  \left( \ul{\nabla}_{\bar{\alpha}} \ul{\Nh}_{\gamma \delta \beta} \right)_\circ \, , \label{eq:W-DN2}
	\\
	\left( W_{\beta \bar{\delta} \bar{\gamma} \alpha} \right)_\circ & =  \left( \frac{1}{2} \ul{\Nh}{}_{\bar{\gamma} \bar{\delta}}{}^{\epsilon} \ul{\Nh}_{\alpha \beta \epsilon} -  \ul{\Nh}_{\epsilon (\alpha \beta)} \ul{\Nh}{}_{\bar{\gamma} \bar{\delta}}{}^{\epsilon}  - \ul{\Nh}^{\epsilon}{}_{(\bar{\gamma} \bar{\delta})} \ul{\Nh}_{\alpha \beta \epsilon} +  \ul{\Nh}_{\epsilon (\alpha \beta)}  \ul{\Nh}{}^{\epsilon}{}_{(\bar{\gamma} \bar{\delta})} \right)_\circ + \ul{S}_{\alpha \bar{\gamma} \beta \bar{\delta}} \, ,  \label{eq:W-S-N2}
	\end{align}
\end{subequations}
where  we recall $\ul{\nabla}$ is the Webster--Tanaka connection corresponding to the contact form $\ul{\theta}^0$ associated to the metric $g$, $\ul{\Nh}_{\alpha \beta \gamma}$ and $\ul{S}_{\alpha \bar{\beta} \gamma \bar{\delta}}$ are the Nijenhuis tensor and the Chern--Moser tensor of $(\ul{\mc{M}}, \ul{H}, \ul{J})$ respectively. As a direct consequence of equations \eqref{eq:Weyl-CR}, we obtain the following three theorems:
\begin{thm}\label{thm:main_int}
	Let $(\mc{M},\mbf{c},N,K)$ be a $(2m+2)$-dimensional twist-induced almost Robinson geometry, where $m>1$, with non-shearing congruence of null geodesics $\mc{K}$.  Denote by $(\ul{H}, \ul{J})$ the induced partially integrable contact almost CR structure on the (local) leaf space $\ul{\mc{M}}$ of $\mc{K}$. The following statements are equivalent:
	\begin{enumerate}
		\item In addition to conditions \eqref{eq:NSCNG-AR0}, the Weyl tensor satisfies
		\begin{align}\label{eq:W_int1}
		W^{0}{}_{\alpha \beta \gamma} = 0 \, .
		\end{align}
		\item The almost Robinson structure $(N,K)$ is integrable.
		\item The almost CR structure $(\ul{H}, \ul{J})$ is integrable.
	\end{enumerate}
	If any of these conditions holds, the Weyl tensor also satisfies
	\begin{subequations}\label{eq:W_int0} 
		\begin{align}
		W_{\alpha \beta \gamma \delta} & = 0 \, , \label{eq:W_int2} \\
		\left( W_{\bar{\gamma} \delta \alpha \beta} \right)_\circ = \left( W_{\bar{\gamma} \bar{\delta} \alpha \beta} \right)_\circ & = 0 \, . \label{eq:W_int_ex}
		\end{align}
	\end{subequations}
\end{thm}

\begin{rem}
	The combined equations \eqref{eq:NSCNG-AR}, \eqref{eq:W_int1} and \eqref{eq:W_int2} can be rewritten as
	\begin{align*}
	W(u,v,w,z) & = 0 \, , & \mbox{for any sections $u$, $v$, $w$ and $z$ of $N$,}
	\end{align*}
	and constitute the integrability condition of an almost Robinson structure \emph{regardless of whether the associated congruence of null curves is geodesic and non-shearing, or not} -- see \cite{Hughston1988,Taghavi-Chabert2016,Taghavi-Chabert2014}. Thus, from the remaining conditions in \eqref{eq:W_int_ex}, we see that the  degeneracy of the Weyl curvature is stronger when the almost  Robinson structure is induced from a twisting non-shearing geodetic congruence. On the other hand, condition \eqref{eq:W_int1} is sufficient to establish the involutivity of the almost Robinson structure.
\end{rem}

By inspection of \eqref{eq:W-S-N2}, we prove:
\begin{thm}\label{thm:main-pre-flatness}
	Let $(\mc{M},\mbf{c},N,K)$ be a $(2m+2)$-dimensional twist-induced almost Robinson geometry, where $m>1$, with non-shearing congruence of null geodesics $\mc{K}$. Denote by $(\ul{H}, \ul{J})$ the induced partially integrable contact almost CR structure on the (local) leaf space $\ul{\mc{M}}$ of $\mc{K}$. The following statements are equivalent:
	\begin{enumerate}
		\item In addition to conditions \eqref{eq:NSCNG-AR0}, the Weyl tensor satisfies
		\begin{align*}
		\left( W_{(\alpha}{}^{(\beta}{}_{\gamma)}{}^{\delta)} \right)_\circ = 0 \, .
		\end{align*}
		\item The Chern--Moser tensor and the Nijenhuis tensor of $(\ul{H}, \ul{J})$ satisfy
		\begin{align}\label{eq:S+NN}
		\ul{S}_{\alpha}{}^{\gamma}{}_{\beta}{}^{\delta} + \left( \ul{\Nh}_{ \epsilon(\alpha  \beta)} \ul{\Nh}^{ \epsilon(\gamma \delta)} \right)_\circ & = 0 \, .
		\end{align}
	\end{enumerate}
\end{thm}
\begin{rem}
	In dimension five, i.e.\ $m=2$, condition \eqref{eq:S+NN} reduces to $\ul{S}_{\alpha}{}^{\gamma}{}_{\beta}{}^{\delta} = 0$.
\end{rem}
Combining Theorem \ref{thm:main_int} and Theorem \ref{thm:main-pre-flatness} proves:
\begin{thm}\label{thm:main-flatness}
	Let $(\mc{M},\mbf{c},N,K)$ be a $(2m+2)$-dimensional twist-induced almost Robinson geometry, where $m>1$, with non-shearing congruence of null geodesics $\mc{K}$. Denote by $(\ul{H}, \ul{J})$ the induced partially integrable contact almost CR structure on the (local) leaf space $\ul{\mc{M}}$ of $\mc{K}$. The following statements are equivalent:
	\begin{enumerate}
		\item In addition to conditions \eqref{eq:NSCNG-AR0}, the Weyl tensor satisfies
		\begin{align*}
		W^0{}_{\alpha \beta \gamma} = \left( W_{(\alpha}{}^{(\beta}{}_{\gamma)}{}^{\delta)} \right)_\circ = 0 \, . 
		\end{align*}
		\item $(\ul{H}, \ul{J})$ is locally flat, i.e.\ the Chern--Moser tensor and  the Nijenhuis tensor of $(\ul{H}, \ul{J})$ vanish, i.e.\
		\begin{align*}
		\ul{S}_{\alpha \bar{\gamma} \beta \bar{\delta}} = \ul{\Nh}_{\alpha  \beta \gamma}  & = 0 \, .
		\end{align*}
	\end{enumerate}
	If any of these conditions holds, the Weyl tensor also satisfies conditions \eqref{eq:W_int0}.
\end{thm}

The previous theorem immediately yields:
\begin{cor}\label{cor:Conf-CR flat}
	Let $(\mc{M}, \mbf{c})$ be an oriented locally conformally flat manifold of even dimension greater than four. Suppose that it admits an optical structure $K$ with twisting non-shearing congruence of geodesics $\mc{K}$. Then the twist of $\mc{K}$ induces a flat contact CR structure on the (local) leaf space of $\mc{K}$.
\end{cor}

\begin{rem}
	Theorem \ref{thm:main-flatness} and Corollary \ref{cor:Conf-CR flat} should be contrasted with the situation in dimension four where:
	\begin{enumerate}
		\item CR flatness cannot be inferred from the Petrov types alone. For instance, the so-called \emph{Robinson congruence} whose underlying CR structure is flat, occurs in Minkowski space (Petrov type O), the Taub-NUT metric (Petrov type D) and Hauser's waves of Petrov type N -- see \cite{Nurowski2002,Trautman2002}.
		\item Conformal flatness does not imply CR flatness. Indeed, the Kerr theorem asserts that any analytic non-shearing congruence of null geodesics in conformally flat spacetime arises as the intersection of a complex submanifold of complex projective $3$-space $\mathbf{CP}^3$ and the real five-dimensional CR hyperquadric embedded therein. There are many examples of such complex submanifolds that do not give rise to a flat CR structure: for instance, the locus of a homogeneous polynomial of degree two in $\mathbf{CP}^3$ -- see \cite{Penrose1967,Penrose1986,Nurowski2002,Trautman2002}.
	\end{enumerate}
\end{rem}

\section{Algebraic conditions on the Weyl tensor}\label{sec:Weyl_tensor}
The purpose of this section is to justify an additional assumption on the Weyl tensor that we will be using when solving the Einstein field equations. We start by recalling the well-known fact \cite{Cartan1922,Petrov1954} that at any point $p$ of a four-dimensional Lorentzian conformal manifold  $(\mc{M},\mbf{c})$, the Weyl tensor determines four null directions, and thus four optical structures in a neighbourhood of $p$: any null vector $k$ that defines such a null direction is a solution of
\begin{align*}
\bm{\kappa}_{[a} W_{b] e f [c} \bm{\kappa}_{d]} k^e k^f & = 0 \, , & \mbox{where $\bm{\kappa} = \bm{g}(k,\cdot)$,}
\end{align*}
and (the span of) $k$ is said to be a \emph{principal null direction} of the Weyl tensor at $p$. If $k$ satisfies the stronger condition
\begin{align*}
k^d W_{d a e [b} \bm{\kappa}_{c]} k^e & = 0 \, , & \mbox{where $\bm{\kappa} = \bm{g}(k,\cdot)$,}
\end{align*}
the corresponding null direction is said to be \emph{repeated}. In this case, we say that the Weyl tensor is \emph{algebraically special} at $p$, or of Petrov type II or more degenerate at $p$. This condition is particularly important in general relativity especially in relation to the Goldberg--Sachs theorem \cite{Goldberg1962,Goldberg2009} as explained in the introduction.

In the present context, our starting point is an optical geometry $(\mc{M},\mbf{c},K)$ of dimension $2m+2$ greater than four, with twisting non-shearing congruence of null geodesics $\mc{K}$. As in dimension four, the line distribution $K$ provides a number of criteria on the basis of which we may describe the algebraic degeneracy of the Weyl tensor as is discussed in e.g.\ \cite{Milson2005a,Ortaggio2009b}. One such degeneracy criterion is that for any optical vector field $k$, the Weyl tensor satisfies
\begin{align}\label{eq:Weyl_deg_II}
\bm{\kappa}_{[a} W_{b c] f [d} \bm{\kappa}_{e]} k^f & = 0 \, ,  & \mbox{where $\bm{\kappa} = \bm{g}(k,\cdot)$.}
\end{align}
Taking the trace and the tracefree part of condition \eqref{eq:Weyl_deg_II} yields the two respective weaker conditions
\begin{align}
k^d W_{d a e [b} \bm{\kappa}_{c]} k^e & = 0 \, , \label{eq:W_I_inv} \\
\left( \bm{\kappa}_{[a} W_{b c] f [d} \bm{\kappa}_{e]} k^f \right)_\circ & = 0 \, . \label{eq:W_II_inv}
\end{align}
Condition \eqref{eq:W_II_inv} has no analogue in dimension four, and turns out to be too strong under general assumptions. Instead, we shall show that condition \eqref{eq:W_I_inv} will be sufficient to facilitate our computations. 

Let us examine the consequence of each of these conditions in the present context. As before, we  assume the setting of Theorem \ref{thm-int-cond2m} and work with a metric $g$ given by \eqref{eq:g_can}  in an adapted frame.

\subsection{Degeneracy condition \eqref{eq:W_I_inv}}
Let us rewrite condition \eqref{eq:W_I_inv} as
\begin{align}
W_{\alpha}{}^{0}{}_{\beta}{}^{0} = W_{\alpha}{}^{0}{}_{\bar{\beta}}{}^{0} & = 0 \, , \label{eq:Weyl_degI_triv} \\
W_{\alpha}{}^{0}{}_{0}{}^{0} & = 0 \, . \label{eq:Weyl_degI}
\end{align}
Conditions \eqref{eq:Weyl_degI_triv} is trivially satisfied by virtue of Theorem \ref{thm-int-cond2m}. On the other hand, from equations \eqref{eq:R-diffE} and \eqref{eq:Ric-diffE}, the LHS of \eqref{eq:Weyl_degI} yields
\begin{align*}
W_{\alpha}{}^{0 0}{}_{0}  & =  \frac{1}{2m} \left( (2m-1) \dot{E}_{\alpha}   - ( 2m-4 ) \ii  E_{\alpha}  \right) \, .
\end{align*}
We thus see that condition \eqref{eq:W_I_inv} is equivalent to
\begin{align}\label{eq:ODE_Weyl}
\dot{E}_{\alpha} & = \frac{2m-4}{2m-1}  \ii E_{\alpha}  & & \mbox{i.e.} & &\ddot{\lambda}_{\alpha}  = \frac{2m-4}{2m-1}  \ii \dot{\lambda}_{\alpha}   \, .
\end{align}
This ODE has general solution
\begin{subequations}\label{eq:lambda_special}
	\begin{align}
	\lambda_{\alpha} & = - 2 \ii \frac{2m-1}{2m-4} \ul{E}_{\alpha} \ee^{\frac{2m-4 }{2m-1} \ii  \phi}  + \ul{\lambda}_{\alpha} \, , & m \neq 2 \, ,\\
	\lambda_{\alpha} & = 2 \ul{E}_{\alpha} \phi + \ul{\lambda}_{\alpha} \, , & m = 2 \, ,
	\end{align}
\end{subequations}
for some functions $\ul{E}_{\alpha}$ and $\ul{\lambda}_{\alpha}$ on $(\ul{\mc{M}},\ul{H},\ul{J})$.

\subsection{Degeneracy condition \eqref{eq:W_II_inv}}
Let us now rewrite the LHS of \eqref{eq:W_II_inv} as $\left( W^{0}{}_{i j k} \right)_\circ$. This tensor splits into three components under the structure group of the almost Robinson manifold  \cite{Taghavi-Chabert2014}:
\begin{itemize}
	\item $\left( W^{0}{}_{\alpha \beta \bar{\gamma}} \right)_\circ$, which vanishes -- see \eqref{eq:NSCNG-AR-x} of Proposition \ref{prop:tw-in-Rob-NS-deg_W}, or \eqref{eq:R-E},
	\item $W^{0}{}_{\alpha \beta \gamma}$, which is proportional to the Nijenhuis tensor -- see \eqref{eq:W-N} or \eqref{eq:R-T},
	\item and the trace of $\left( W^{0}{}_{i j k} \right)_\circ$ with respect to the bundle Hermitian form $\omega_{i j}$ induced from the twist of $k$, and which is proportional to $E_i$ by a constant factor, as follows from \eqref{eq:R-diffE}, \eqref{eq:R-E} and \eqref{eq:Ric-diffE}.
\end{itemize}
Hence, given our assumptions, condition \eqref{eq:W_II_inv} holds if and only if  $\ul{\Nh}_{\alpha \beta \gamma} =0$ and $E_\alpha = 0$. In the context of finding solutions to Einstein field equations, such a condition is too strong since we will allow the almost Robinson structure to be non-integrable, i.e.\ $\ul{\Nh}_{\alpha \beta \gamma}$ does not necessarily vanish. We shall show however at Step 2 of Section \ref{sec:Computations} that $\lambda_{\alpha}$, and thus $E_{\alpha}$, must eventually vanish.

\section{Einstein metrics}\label{sec:Einstein}
Throughout this section, we assume that $(\mc{M},\mbf{c},N, K)$ is a twist-induced almost Robinson geometry of dimension $2m+2$ with non-shearing congruence of null geodesics $\mc{K}$ as in Theorem \ref{thm-int-cond2m}. We assume $m>1$, only pointing out notable differences when $m=1$. We now seek a metric in the conformal class $\mbf{c}$, which satisfies the Einstein field equations \eqref{eq:EFE} below.

We proceed as follows. Let $g$ and $\wh{g}$ be two metrics in $\mbf{c}$ related by \eqref{eq-conf-res}. For convenience, and with no loss of generality, we take $g$ to be in the conformal subclass $\accentset{n.e.}{\mbf{c}}$ so that it takes the form \eqref{eq:g_can} and we can use the computation of the curvature given in Appendix \ref{app:computations}. The reader should also refer to Section \ref{sec-Robinson} for the general setup and notation. The unknown metric $\wh{g}$ must satisfy the Einstein field equation with pure radiation
\begin{align}\label{eq:EFE}
\wh{\Ric}_{a b} & = \Lambda \wh{g}_{a b} + \Phi \wh{\kappa}_a \wh{\kappa}_b \, ,
\end{align}
where $\Lambda$ is the cosmological constant and $\Phi$ a smooth function on $\mc{M}$, which \emph{may or may not be identically zero}. In terms of the Schouten tensor, these read as
\begin{align}\label{eq:EFE_Rho}
\wh{\Rho}_{a b} & = \frac{1}{2(2m+1)} \Lambda \wh{g}_{a b} +  \frac{1}{2m} \Phi \wh{\kappa}_a \wh{\kappa}_b \, .
\end{align}
At Step 2 below, we shall also impose condition \eqref{eq:W_I_inv} on the Weyl tensor.

\subsection{Computations}\label{sec:Computations}
Following the strategy set up in \cite{Lewandowski1990}, we shall integrate the components of the Einstein field equations \eqref{eq:EFE_Rho} successively. The computation of the components of $\widehat{\Rho}_{a b}$ will be achieved using the transformation law \eqref{eq:Rho_transf} between the Schouten tensors  of $g$ and $\wh{g}$, and the components of the curvature tensors given in Appendix \ref{app:computations}. The computation is lengthy and we have deliberately left out many details.

\begin{fleqn}
	\begin{align}
	\quad \bullet \; \mbox{\textbf{Step 1:}} & &  \widehat{\Rho}{}^{0 0} = 0  \, . \label{eq:EFE00} 
	\end{align}
\end{fleqn}
We compute
\begin{align*}
\widehat{\Rho}{}^{0 0}  & =  1 - \ddot{\varphi} + \dot{\varphi}^2 \, .
\end{align*}
Hence, equation \eqref{eq:EFE00} holds if and only if
\begin{align*}
\ddot{\varphi} - \dot{\varphi}^2 = 1 \, .
\end{align*}
The general solution to this equation is $\varphi = \frac{1}{2} \ul{\varphi} - \ln|\cos(\phi+\ul{\psi})|$ where $\ul{\psi}$ and $\ul{\varphi}$ are functions on $\ul{\mc{M}}$.
Hence the conformal factor is given by $\ee^{2 \varphi} =\frac{ \ee^{\ul{\varphi}}}{\cos^2(\phi+\ul{\psi})}$.
We can use $\ul{\varphi}$ to change adapted CR frames by absorbing $\ee^{\ul{\varphi}}$ into $\ul{\theta}^0$. Further, a change of affine parameter along the geodesics of $k$ together with some redefinitions of the functions $\lambda_{\alpha}$ and $\lambda_0$ can be used to eliminate the function $\ul{\psi}$. Thus, with no loss of generality, we shall take
\begin{align*}
\varphi & = - \ln|\cos \phi | \, ,
\end{align*}
Hence, the conformal factor can be taken to be
\begin{align*}
\ee^{2 \varphi} & =\frac{1}{\cos^2 \phi} = \sec^2 \phi\, .
\end{align*}
For future use, we record $\dot{\varphi} = \tan \phi$ and $\ddot{\varphi} = \sec^2 \phi$.

\begin{rem}\label{rem:circle}
	The inverse of the conformal factor is a periodic function of period $\pi$, and one can already anticipate that the metric $\wh{g}$ is defined on a circle bundle over $\ul{\mc{M}}$. With no loss of generality, we may choose $\phi$ to lie in the interval $(-\frac{\pi}{2} , \frac{\pi}{2})$. This is particularly relevant to the relation with Fefferman--Einstein metrics given in Section \ref{sec:Feff-Einstein}.
\end{rem}

\begin{fleqn}
	\begin{align}
	\quad \bullet \; \mbox{\textbf{Step 2:}} & & \widehat{\Rho}_{\alpha}{}^{0} = 0 \, . \label{eq:EFEA0} 
	\end{align}
	We find
\end{fleqn}
\begin{align*}
\wh{\Rho}_{\alpha}{}^{0} & = \frac{1}{2m} (\dot{E}_{\alpha} - 4 \ii E_{\alpha} ) + ( 1-  \ii  \dot{\varphi} ) \lambda_\alpha + \dot{\varphi} E_{\alpha}
\end{align*}
Hence, equation \eqref{eq:EFEA0} holds if and only if
\begin{align}\label{eq:ODE_step2}
\ddot{\lambda}_{\alpha} - \left( 4 \ii  - 2m \tan \phi \right) \dot{\lambda}_{\alpha} + 4m ( 1-  \ii  \tan \phi ) \lambda_\alpha  & = 0 \, .
\end{align}
The general solution of this second-order linear ODE will be treated elsewhere. Instead, at this stage, we shall impose the condition \eqref{eq:W_I_inv} so that $\lambda_{\alpha}$ takes the form \eqref{eq:lambda_special}. Plugging \eqref{eq:lambda_special} into the above equation and assuming $m > 1$ yields
\begin{align*}
\lambda_{\alpha} & = 0 \, .
\end{align*}

\begin{rem}
	In dimension four, that is, when $m=1$, the situation is remarkably different. For one, if one assumes the Einstein condition, the degeneracy condition \eqref{eq:W_I_inv} on the Weyl tensor \emph{follows} from the Goldberg--Sachs theorem. In addition, the general solution of the system of ODEs \eqref{eq:ODE_Weyl} and \eqref{eq:ODE_step2} is given by
	\begin{align*}
	\lambda_{\alpha} & = \left( 1 + \frac{1}{2} \ee^{-2\ii \phi} \right) \ul{\lambda}_{\alpha}  \, ,
	\end{align*}
	for some smooth functions $\ul{\lambda}_{\alpha}$ on $\ul{\mc{M}}$. We thus recover the result already obtained in \cite{Lewandowski1990}. The fact that $\lambda_{\alpha}$ may be non-zero here allows many more possible solutions of the Einstein field equations.
\end{rem}

\begin{fleqn}
	\begin{align}
	\quad \bullet \; \mbox{\textbf{Step 3:}} & &  \widehat{\Rho}_{\alpha \beta} = 0 \, . \label{eq:EFEAB} 
	\end{align}
\end{fleqn}
We compute
\begin{align*}
\widehat{\Rho}_{\alpha \beta}  & = \frac{1}{m}  \ul{\nabla}^{\gamma}{\ul{\Nh}_{\gamma (\alpha \beta)}} + \frac{\ii}{2} (1 + \ii  \dot{\varphi}  ) \ul{A}_{\alpha \beta} \, .
\end{align*}
Hence, equation \eqref{eq:EFEAB} holds if and only if
\begin{subequations}\label{eq:EFE->A_DN=0}
	\begin{align}
	\ul{A}_{\alpha \beta} & =  0 \, , \label{eq:EFE->A=0} \\
	\ul{\nabla}^{\gamma}{\ul{\Nh}_{\gamma (\alpha \beta)}} & = 0 \, . \label{eq:EFE->DN=0} 
	\end{align}
\end{subequations}
\begin{fleqn}
	\begin{align}
	\quad \bullet \; \mbox{\textbf{Step 4:}} & &  \left( \wh{\Rho}_{\alpha \bar{\beta}} \right)_\circ = 0 \, . \label{eq:EFEABMo} 
	\end{align}
\end{fleqn}
We compute
\begin{multline}\label{eq:whRhoABM}
\widehat{\Rho}_{\alpha}{}^{\beta} = \frac{1}{2m}  \left(  \ul{\Ric}_{\alpha}\,^{\beta} - \ul{\Nh}_{\alpha \delta \gamma} \ul{\Nh}^{\beta \delta \gamma} \right)   \\
- \frac{1}{2m(2m+1)}  \left(\frac{1}{2} \ddot{\lambda}\,_{0} + \left( 2 (m+1)^2 - m(2m+1) \ddot{\varphi}  \right)\lambda_0  + m \ul{\Lambda} \right) \delta_{\alpha}^{\beta} \, ,
\end{multline}
where we have defined
\begin{align*}
\ul{\Lambda}  & := \frac{1}{m} \left( \ul{\Sc} -  \ul{\Nh}_{\alpha \delta \gamma} \ul{\Nh}^{\alpha \delta \gamma} \right)   \, .
\end{align*}
Taking the tracefree part of \eqref{eq:whRhoABM}, we immediately conclude that \eqref{eq:EFEABMo} holds if and only if
\begin{align}\label{eq:EFE->Ric-N=const}
\ul{\Ric}_{\alpha}\,^{\beta} -\ul{\Nh}_{\alpha \delta \gamma} \ul{\Nh}^{\beta \delta \gamma}   & = \ul{\Lambda} \delta_{\alpha}^{\beta} \, .
\end{align}
Condition \eqref{eq:EFE->Ric-N=const} together with conditions \eqref{eq:EFE->A_DN=0} allows us to conclude immediately that the almost CR manifold is CR--Einstein, and in particular, $\ul{\Lambda}$ is  constant.

\begin{fleqn}
	\begin{align}
	\quad \bullet \; \mbox{\textbf{Step 5:}} & &   \wh{\Rho} = \frac{m+1}{2m+1}  \Lambda \, . \label{eq:EFESc} 
	\end{align}
\end{fleqn}
We take the trace of equation \eqref{eq:whRhoABM} to find
\begin{align}\label{eq:whRhoABM-tr}
\widehat{\Rho}_{\alpha}{}^{\alpha} 
& =  - \frac{1}{2(2m+1)}  \left(\frac{1}{2} \ddot{\lambda}\,_{0} + \left( 2 (m+1)^2 - m(2m+1) \ddot{\varphi}  \right)\lambda_0   \right)  + \frac{m+1}{2(2m+1)} \ul{\Lambda}  \, ,
\end{align}
and we  compute
\begin{align}
\widehat{\Rho}_{0}{}^{0} 
& = -  \frac{1}{2(2m+1)}  \ul{\Lambda}  
+  \frac{1}{2(2m+1)}  \ddot{\lambda}\,_{0}  + \frac{1}{2} \dot{\varphi}  \dot{\lambda}_0  + \left( \frac{m+1}{2m+1}
+ \frac{1}{2} \ddot{\varphi} \right) \lambda_0   \, . \label{eq:whRho00M}
\end{align}
Hence, using the fact that
$\wh{\Rho} = 2 \ee^{-2 \varphi} \left( \wh{\Rho}_{0}{}^0 + \wh{\Rho}_{\alpha}{}^{\alpha} \right)$,
we find
\begin{align*}
\ee^{2 \varphi} \widehat{\Rho} & = \frac{1}{2(2m+1)} \ddot{\lambda}\,_{0} +  \dot{\varphi}  \dot{\lambda}_0  + \left( -\frac{2m (m+1)}{2m+1}  + (m+1) \ddot{\varphi}  \right)\lambda_0     + \frac{m}{2m+1} \ul{\Lambda}  \, .
\end{align*}
Therefore, equation \eqref{eq:EFESc} holds if and only if
\begin{multline}\label{eq1}
\ddot{\lambda}\,_{0}  + 2(2m+1)  \tan \phi \dot{\lambda}_0     +   \left( - 4m(m+1) +  2 (m +1 ) (2m+1) \sec^2 \phi  \right)\lambda_0   = \\
2 (m+1)\Lambda \sec^2 \phi -  2 m \ul{\Lambda} \, .
\end{multline}
Equation \eqref{eq1} is a second-order linear ODE with general solution
\begin{multline}\label{eq-soln}
\lambda_0 = \frac{1}{2m+2} \ul{\Lambda}  +   \left( \frac{1}{2m+1} \Lambda - \frac{1}{2m+2} \ul{\Lambda} \right)  \sum_{j=0}^{m} a_{j} \cos^{2 j} \phi  \\ + \ul{b}  \cos^{2m+2} \phi  
+ \ul{c} \cos^{2m+1} \phi \sin \phi \, ,
\end{multline}
where the  $a_i$ are given by
\begin{align*}
a_{0} = 1  \, , & & a_{j} =  \frac{2m-2j+4}{2m-2j+1 }  a_{j-1} \, , & & j = 1, \ldots, m \, ,
\end{align*}
and $\ul{b}$ and $\ul{c}$ are arbitrary functions on $\ul{\mc{M}}$.

\begin{fleqn}
	\begin{align}
	\quad \bullet \; \mbox{\textbf{Step 6:}} & &   \widehat{\Rho}_{0}{}^{0} =\frac{1}{2(2m+1)} \Lambda \ee^{2 \varphi}  \, . \label{eq:EFE00M} 
	\end{align}
\end{fleqn}
From equation \eqref{eq:whRho00M}, we conclude that equation \eqref{eq:EFE00M} holds if and only if
\begin{align}
\ddot{\lambda}\,_{0}  +  (2m+1) \tan \phi \dot{\lambda}_0  + \left( 2(m+1)
+ (2m+1) \sec^2 \phi \right) \lambda_0  
& = \Lambda \sec^2 \phi
+ \ul{\Lambda}  \, . \label{eq2}
\end{align}
Since the function \eqref{eq-soln} is also a solution of \eqref{eq2}, we will be able to reduce the number of arbitrary functions to one. To facilitate the computation, and for future use, we reduce the second-order ODEs to a first-order ODE by plugging equation \eqref{eq1} into \eqref{eq2}. We find
\begin{align}\label{eq4}
\tan \phi \dot{\lambda}_0  - \left( 2m+2 - (2m+1) \sec^2 \phi \right) \lambda_0  
& =  \Lambda \sec^2 \phi - \ul{\Lambda} \, .
\end{align}
It remains to plug the solution \eqref{eq-soln} into \eqref{eq4},  which gives the relation
\begin{align}\label{eq:bam}
\ul{b} & =  - 2 \left( \frac{1}{2m+1} \Lambda - \frac{1}{2m+2} \ul{\Lambda} \right) a_{m} \, .
\end{align}
At this stage, all the coefficients of the function $\lambda_0$ with the exception of $\ul{c}$ have been determined.

\begin{rem}
	One can alternatively use $\wh{\Rho}_{\alpha}{}^{\alpha} = \frac{m}{2(2m+1)} \Lambda \ee^{2 \varphi}$. In this case, equation \eqref{eq:whRhoABM-tr} tells us that $\lambda_0$ must satisfy
	\begin{align}\label{eq3}
	\ddot{\lambda}\,_{0} +\left( 4 (m+1)^2 - 2m(2m+1) \sec^2 \phi  \right)\lambda_0  & = - 2 m \Lambda  \sec^2 \phi
	+  2 (m+1) \ul{\Lambda}  \, .
	\end{align}
	Plugging \eqref{eq3} into \eqref{eq1} also yields \eqref{eq4}, from which the same conclusion \eqref{eq:bam} follows.
\end{rem}

\begin{fleqn}
	\begin{align}
	\quad \bullet \; \mbox{\textbf{Step 7:}} & &  \widehat{\Rho}_{\alpha 0} = 0 \, . \label{eq:EFEA0D} 
	\end{align}
\end{fleqn}
We compute
\begin{align*}
\widehat{\Rho}_{\alpha 0} & =  \left(  - \frac{1}{4m} \cos^{2m} \phi + \frac{\ii}{2m} \cos^{2m+1} \phi \sin \phi + \frac{1}{2m} \cos^{2m+2} \phi \right) \ul{\nabla}_{\alpha} \ul{c} \, .
\end{align*}
Hence, equation \eqref{eq:EFEA0D} holds if and only if $\ul{\nabla}_{\alpha} \ul{c} = 0$. We can take a covariant derivative $\ul{\nabla}_{\bar{\beta}}$ of this expression and commute the derivatives so that the first of the commutation relations \eqref{eq:CR_com} together with the fact that $\ul{c}$ is real allows us to conclude $\ul{\nabla}_{0} \ul{c} = 0$, i.e.\  $\ul{c}$ is constant.  

We have now completely determined the function $\lambda_0$, and thus the $1$-form $\lambda$ and metric $\wh{g}$. There remains to check, in the final step, under which conditions the metric $\wh{g}$ is consistent with the remaining component of \eqref{eq:EFE_Rho}.
\begin{fleqn}
	\begin{align}
	\quad \bullet \; \mbox{\textbf{Step 8:}} & &  \wh{\Rho}_{0 0} = \Phi \ee^{4 \varphi}  \, . \label{eq:EFE00D} 
	\end{align}
\end{fleqn}
Our previous findings lead us to assert that $\widehat{\Rho}_{0 0}  = 0$. Thus, equation \eqref{eq:EFE00D} holds if and only if $\Phi = 0$,
i.e.\ our hypotheses do not allow for the presence of pure radiation on $(\mc{M}, \wh{g})$.

\subsection{Conclusions}
From the computations of Section \ref{sec:Computations}, we extract the following results:
\begin{thm}\label{thm:Ricci-deg}
	Let $(\mc{M},\mbf{c},K)$ be a $(2m+2)$-dimensional conformal optical geometry, where $m>1$, with twisting non-shearing congruence of null geodesics $\mc{K}$. Suppose that the Weyl tensor satisfies
	\begin{align*}
	W (k, v, k, \cdot) & = 0 \, , &  \mbox{for any sections $k$ of $K$, $v$ of $K^\perp$,}
	\end{align*}
	so that  $(\mc{M},\mbf{c},K)$ admits a twist-induced almost Robinson structure $(N,K)$. Denote by $(\ul{H}, \ul{J})$ the induced partially integrable contact almost CR structure of positive definite signature on the (local) leaf space $\ul{\mc{M}}$ of $\mc{K}$, and by $\varpi$ the natural projection from $\mc{M}$ to $\ul{\mc{M}}$.
	\begin{enumerate}
		\item Suppose that $\mbf{c}$ contains a metric $\wh{g}$ whose Ricci tensor satisfies \label{item:Ricci-deg1}
		\begin{align*}
		\wh{\Ric}(w,w) & = 0 \, , &  \mbox{for any sections $w$ of $N$.}
		\end{align*}
		Then $\mc{M}$ is locally diffeomorphic to $\left( -\frac{\pi}{2} , \frac{\pi}{2} \right) \times \ul{\mc{M}}$, and there is a distinguished contact form $\ul{\theta}^0$ for $(\ul{H}, \ul{J})$ with Levi form $\ul{h}$ such that the pseudo-Hermitian torsion tensor $\ul{A}_{\alpha \beta}$ and the Nijenhuis tensor $\ul{\Nh}_{\gamma \alpha \beta}$ satisfy $\ul{A}_{\alpha \beta} = \ul{\nabla}^{\gamma}{\ul{\Nh}_{\gamma (\alpha \beta)}} = 0$, and the metric $\wh{g}$ takes the form
		\begin{align*}
		\wh{g} & = \sec^2 \phi \, g \, , & \mbox{for $- \frac{\pi}{2} < \phi < \frac{\pi}{2}$,} 
		\end{align*}
		where
		\begin{align*}
		g & =2 \, \kappa \odot \lambda + h  \, , \\
		\kappa & = 2 \, \varpi^* \ul{\theta}^0 \, , &
		h & = \varpi^* \ul{h} \, , &
		\lambda & = \dd \phi + \lambda_0 \, \varpi^* \ul{\theta}^0 \, , 
		\end{align*}
		for some smooth function $\lambda_0$ on $\mc{M}$.
		\item Suppose that $\mbf{c}$ contains a metric $\wh{g}$ whose Ricci tensor satisfies \label{item:Ricci-deg2}
		\begin{align*}
		\wh{\Ric}(v,v) & = \Lambda \wh{g}(v,v)  \, , & & \mbox{for any section $v$ of $K^\perp$,} \\
		\wh{\Ric}(k,\cdot) & = \Lambda \wh{g}(k,\cdot) \, , & & \mbox{for any section $k$ of $K$,}
		\end{align*}
		for some constant $\Lambda$, i.e.\ $\kappa_{[a} \wh{\Ric}_{b] [c} \kappa_{d]} = \Lambda \kappa_{[a} \wh{g}_{b] [c} \kappa_{d]}$ and $k^a \wh{\Ric}_{a}{}^{b}  = \Lambda k^b$, for any optical $1$-form $\kappa_a$ and optical vector field $k^a$.
		Then part \eqref{item:Ricci-deg1} holds. In addition, the almost pseudo-Hermitian structure $(\ul{H},\ul{J},\ul{\theta}^0)$ is almost CR--Einstein, and the function $\lambda_0$ is now given by
		\begin{multline*}
		\lambda_0 = \frac{\ul{\Lambda} }{2m+2}  +   \left( \frac{\Lambda}{2m+1}  - \frac{\ul{\Lambda} }{2m+2} \right)\left(  \sum_{j=0}^{m} a_{j} \cos^{2 j} \phi   -2  a_{m}  \cos^{2m+2} \phi \right) \\
		+ \ul{c} \cos^{2m+1} \phi \sin \phi \, ,
		\end{multline*}
		where $ \ul{\Lambda}  := \frac{1}{m} \left( \ul{\Sc} -  \ul{\Nh}_{\alpha \delta \gamma} \ul{\Nh}^{\alpha \delta \gamma} \right)$, $\ul{c}$ is a smooth function on $\ul{\mc{M}}$, and 
		\begin{align*}
		a_{0} = 1  \, , & & a_{j} =  \frac{2m-2j+4}{2m-2j+1 }  a_{j-1} \, , & & j = 1, \ldots, m \, .
		\end{align*}
		\item Suppose that $\mbf{c}$ contains a metric $\wh{g}$ whose Ricci tensor satisfies \label{item:Ricci-deg3}
		\begin{align*}
		\wh{\Ric}(v,\cdot) & = \Lambda \wh{g}(v,\cdot) \, , &  \mbox{for any section $v$ of $K^\perp$,}
		\end{align*}
		i.e.\ $\kappa_{[a} \wh{\Ric}_{b] c} = \Lambda \kappa_{[a} \wh{g}_{b] c}$ for any optical $1$-form $\kappa_a$. Then part \eqref{item:Ricci-deg2} holds. In addition, the function $\ul{c}$ in $\lambda_0$ is now a constant.
		
		Further, $\wh{g}$ must be Einstein. In other words, $\mbf{c}$ cannot contain a metric $\wh{g}$ that satisfies the Einstein field equation with pure radiation \eqref{eq:EFE} with non-zero $\Phi$.
	\end{enumerate}
\end{thm}

\begin{proof}
	Part \ref{item:Ricci-deg1} follow from Steps 1,2 and 3 of Section \ref{sec:Computations}. Part \ref{item:Ricci-deg2} from Steps 1 to 6, while  including Steps 7 and 8 as well proves part \ref{item:Ricci-deg3}.
\end{proof}

Combining part \ref{item:Ricci-deg3} of Theorem \ref{thm:Ricci-deg} with Theorem \ref{thm:main_int} now gives Theorem \ref{thm:main-Einstein} where we have written
$\| \ul{\Nh} \|^2_{\ul{h}} = \ul{\Nh}_{\alpha \beta \gamma} \ul{\Nh}^{\alpha \beta \gamma}$.

\begin{rem}
	To construct examples illustrating  Theorems \ref{thm:main-Einstein} and \ref{thm:Ricci-deg}, it suffices to choose a known almost K\"{a}hler--Einstein manifold $(\ut{\mc{M}},\ut{h},\ut{J})$ as discussed at the end of Section \ref{sec:almost_Kahler}, extend $(\ut{h},\ut{J})$ to an almost CR--Einstein structure $(\ul{H},\ul{J},\ul{\theta}^0)$ on a rank-$1$ associated bundle $\ul{\mc{M}}$ as in  Proposition \ref{prop:aK2aCR}, which we then lift to $\mc{M} = \left(-\frac{\pi}{2} , \frac{\pi}{2}\right) \times \ul{\mc{M}}$ as an almost Robinson structure $(N,K)$ for the Einstein metric $\wh{g}$ as in Theorem \ref{thm:main-Einstein}.
\end{rem}

Under the assumptions of Theorem \ref{thm:main-Einstein}, the covariant derivatives of the coframe $1$-forms \eqref{eq:cov_der_frame} with respect to the Levi-Civita connection $\nabla$ of $g$ reduce to
\begin{subequations}\label{eq:Dform_E}
	\begin{align}
	\nabla \kappa & = 2 \, \ii  \ul{h}_{\alpha \bar{\beta}} \theta{}^{\alpha} \wedge \overline{\theta}{}^{\bar{\beta}} + \frac{1}{2} \dot{\lambda}_{0} \kappa \otimes \kappa  \, ,  \label{eq:kappa} \\
	\nabla \theta^\alpha 
	& = \ul{\nabla} \theta^\alpha  - \ul{\Nh}^{\alpha}{}_{\bar{\beta} \bar{\gamma}} \overline{\theta}{}^{\bar{\gamma}}  \otimes \overline{\theta}{}^{\bar{\beta}} - 2 \, \ii  \lambda  \odot \theta^{\alpha} + \ii  \lambda_0  \kappa \odot \theta^{\alpha}   \, , \label{eq:theta}  \\
	\nabla \lambda 
	& =     \ii  \lambda_0 \ul{h}_{\alpha \bar{\beta}} \theta{}^{\alpha} \wedge \overline{\theta}{}^{\bar{\beta}} -  \frac{1}{2} \dot{\lambda}_{0}  \kappa \otimes \lambda  \, . \label{eq:lambda}
	\end{align}
\end{subequations}
The covariant derivative with respect to the Levi-Civita connection $\wh{\nabla}$ of $\wh{g}$ can easily be obtained using formula \eqref{eq-conf-tr-form} with $\Upsilon = \tan \phi \left( \lambda - \frac{1}{2} \lambda_0 \kappa \right)$. For instance, one can easily verify that the expansion of $k = g^{-1}(\kappa,\cdot)$ with respect to $\wh{g}$ is given by $\wh{\epsilon} =2m \tan \phi$.

The non-vanishing components of the Weyl tensor, with reference to Appendix \ref{app:computations}, are given by
\begin{subequations}\label{eq:Weyl-Ein} 
	\begin{align}
	W_{\gamma}{}^{0}{}_{\alpha \beta} & =  2 \, \ii \ul{\Nh}_{\alpha \beta \gamma} \, , \label{eq:Wk1} \\
	W_{\gamma \delta \alpha \beta} & = 2 \,  \ul{\nabla}_{[\gamma|} \ul{\Nh}_{\alpha \beta |\delta]} \, , \\
	W_{\gamma \delta \bar{\beta} \alpha} & = \ul{\nabla}_{\bar{\beta}} \ul{\Nh}_{\gamma \delta \alpha}  \, , \\
	W_{\gamma}{}^{\delta \beta}{}_{\alpha}   
	& =  \ul{R}_{\gamma}{}^{\delta}{}_{\alpha}{}^{\beta}   + \ul{\Nh}^{\epsilon \beta \delta}  \ul{\Nh}_{\epsilon \alpha \gamma} -   \lambda_0 \delta_{\alpha}^{\beta}  \delta^{\delta}_{\gamma}  
	-  \lambda_0 \delta_{\gamma}^{\delta} \delta_{\alpha}^{\beta}  
	- \frac{1}{2}\lambda_0 \delta_{\gamma}^{\beta} \delta_{\alpha}^{\delta}  
	- \frac{1}{2} \lambda_0 \delta_{\alpha}^{\delta} \delta^{\beta}_{\gamma} \nonumber \\
	& \qquad 
	+ \frac{1}{2m(2m+1)}\left( \ddot{\lambda}\,_{0} + 2 (3m+2) \lambda_0 - 2 (m + 1)  \ul{\Lambda}  \right)  \delta_{\alpha}^{\delta} 
	\delta_{\gamma}^{\beta} 
	\, , \\
	W_{\bar{\beta} 0}{}^{0}\,_{\alpha} 
	& = \left( \frac{1}{2}  \ii \dot{\lambda}_{0} 
	+  \frac{1}{2m(2m+1)} \left( \frac{2m-1}{2}
	\ddot{\lambda}\,_{0}   
	-  (2m+2) \lambda_0 
	+   \ul{\Lambda} \right)  \right) \ul{h}_{\alpha \bar{\beta}} \, , \label{eq:Wk2}  \\
	W_{0}{}^{0 0}\,_{0} 
	& = \frac{1}{2m+1} \left( \frac{2m-1}{2}
	\ddot{\lambda}\,_{0}   
	-  (2m+2) \lambda_0 
	+   \ul{\Lambda} \right) \, , \label{eq:Wk3}  \\
	W_{\gamma 0 \alpha \beta}  & =   \ii  \lambda_0  \ul{\Nh}_{\alpha \beta \gamma}  \, , \label{eq:W-b1}
	\end{align}
\end{subequations}
from which we can obtain non-zero expressions for $W^{0}{}_{0}{}^{\alpha}{}_{\beta}$ and  $W_{\alpha \gamma}{}^{\delta \beta}$ using the Bianchi identities. When the almost Robinson structure is integrable, i.e.\ $\ul{\Nh}_{\alpha \beta \gamma} = 0$, the Weyl tensor satisfies the `Goldberg--Sachs'-type curvature condition put forward in \cite{Taghavi-Chabert2012}:
\begin{align*}
W ( v, w, u, \cdot) & = 0 \, , & \mbox{for all $v, w, u \in \Gamma(N)$.}
\end{align*}
A more detailed study of the curvature is given in \cite{Taghavi-Chabert2014}. For the time being, we draw the following conclusion:

\begin{prop}\label{prop:conf_flat_Rob}
	Let $(\mc{M}, \mbf{c}, K)$ be the optical geometry of Theorem \ref{thm:main-Einstein}. The following statements are equivalent:
	\begin{enumerate}
		\item $(\mc{M}, \mbf{c}, K)$ is (locally) conformally flat;
		\item the almost Robinson structure $(N,K)$ is integrable, the almost pseudo-Hermi-tian structure $(\ul{H},\ul{J},\ul{\theta}^0)$ flat, and the metric $\wh{g}$ flat (and in particular Ricci-flat).
	\end{enumerate}
\end{prop}

\begin{proof}
	The equivalence follows essentially from equations \eqref{eq:Weyl-Ein}. See Theorem \ref{thm:main_int} for the integrability of $(N,K)$, and Theorem \ref{thm:main-flatness} for CR flatness. Now, equation \eqref{eq:W-b1} tells us that $\lambda_0=0$, which means that $\Lambda = \ul{\Lambda} = \ul{c} = 0$. But this is equivalent to $\wh{g}$ being Ricci-flat. In addition, all the pseudo-Hermitian torsion and curvature invariants of $\ul{\theta}^0$ vanish, i.e.\ $(\ul{H},\ul{J},\ul{\theta}^0)$ is flat.
\end{proof}

\section{Special cases}\label{sec:spec}
\subsection{Fefferman--Einstein metrics: $(2m+2) \Lambda = (2m+1) \ul{\Lambda}$ and $\ul{c}=0$}\label{sec:Feff-Einstein}
To any contact almost CR manifold $(\ul{\mc{M}}, \ul{H}, \ul{J})$ of positive definite signature, one can associate a canonical Lorentzian conformal manifold $(\mc{M},\mbf{c}_F)$ on the total space of a circle bundle associated to the canonical bundle of $(\ul{\mc{M}}, \ul{H}, \ul{J})$. The construction was originally due to Fefferman \cite{Fefferman1976,Fefferman1976a}, and  $(\mc{M},\mbf{c}_F)$ has since been known as the \emph{Fefferman space} of $(\ul{\mc{M}}, \ul{H}, \ul{J})$ -- see also \cite{Lee1986}. Leitner subsequently generalised Fefferman's construction to partially integrable contact almost CR structures and allowed the inclusion of an additional gauge field  \cite{Leitner2010}. Regardless of the integrability and the presence of a gauge field, the resulting conformal structure admits a null conformal Killing field, which generates a maximally twisting non-shearing congruence of null geodesics $\mc{K}$ having the almost CR manifold $(\ul{\mc{M}}, \ul{H}, \ul{J})$ as its (local) leaf space, and thus a twist-induced almost Robinson structure. Fefferman spaces thus fall into the class of geometries considered in the present article. In the integrable case, a characterisation of conformal structures as Fefferman spaces was presented in \cite{Graham1987,Cap2010}. The non-integrable case will be given in \cite{Taghavi-Chabert}.

A \emph{Fefferman--Einstein metric} is a metric $g_{FE}$ in $\mbf{c}_F$ that is Einstein, at least on some open subset of $\mc{M}$ \cite{Leitner2007} -- in the terminology of \cite{Gover2005a,Cap2008}, such a metric arises from an \emph{almost Einstein structure}. We shall refer to $(\mc{M}, g_{FE})$ as a \emph{Fefferman--Einstein space}. In this definition, we shall also include Einstein metrics in gauged Fefferman conformal structures over partially integrable contact almost CR manifolds.

We now show that for some values of the parameters of the family of metrics found in Theorem \ref{thm:main-Einstein}, the resulting metric is a Fefferman--Einstein metric.  The proof of the implication  \eqref{item:1} $\Rightarrow$ \eqref{item:5} is only given in the integrable case, but can also be obtained in the non-integrable case by appealing to the results of \cite{Taghavi-Chabert}.

\begin{thm}\label{thm:main-Feff-Einstein}
	Let $(\mc{M}, \mbf{c}, K)$ be the optical geometry of Theorem \ref{thm:main-Einstein}. The following statements are equivalent:
	\begin{enumerate}
		\item $k= g^{-1}(\kappa,\cdot)$ is a conformal Killing field; \label{item:1}
		\item $\ell= g^{-1}(\lambda,\cdot)$ is a conformal Killing field; \label{item:2}
		\item $\lambda_0$ is constant; \label{item:3}
		\item the parameters satisfy $(2m+2) \Lambda = (2m+1) \ul{\Lambda}$ and $\ul{c}=0$; \label{item:4}
		\item $(\mc{M},\wh{g})$ is locally isometric to a Fefferman--Einstein space. \label{item:5}
	\end{enumerate}
\end{thm}

\begin{proof}
	The equivalence of statements \eqref{item:1},  \eqref{item:2}, \eqref{item:3} and \eqref{item:4} is immediate from inspection of equations \eqref{eq:kappa}, \eqref{eq:lambda} and the definition of $\lambda_0$ in Theorem \ref{thm:main-Einstein}. That \eqref{item:5} implies \eqref{item:1} follows from \cite{Graham1987,Cap2010,Leitner2010}. To show that condition \eqref{item:1} implies \eqref{item:5}, we assume that the underlying almost CR structure is integrable, and we can check the criteria for the characterisation of the Fefferman space given in Corollary 3.1 of \cite{Cap2010} -- see also \cite{Graham1987}: under the assumption that $k$ is a null conformal Killing field, we must have, for any metric $g$ in $\mbf{c}$,
	\begin{subequations}\label{eq:Feff_char}
		\begin{align}
		& k^a W_{a b c d} = 0 \, , \label{eq:Feff_char_W} \\
		& k^c Y_{a b c} = 0 \, , \label{eq:Feff_char_Y}  \\
		& \frac{1}{(2m+2)^2} (\nabla_a k^a)^2 - k^a k^b \Rho_{a b} - \frac{1}{2m+2} k^a \nabla_a \nabla_b k^b < 0 \, . \label{eq:Feff_char_Rho} 
		\end{align}
	\end{subequations}
	Here $Y_{a b c} := 2 \nabla_{[b} \Rho_{c] a}$ is the Cotton tensor of $g$. The set of conditions \eqref{eq:Feff_char} are conformally invariant. Condition \eqref{eq:Feff_char_W} is satisfied by \eqref{eq:Wk1}, \eqref{eq:Wk2} and \eqref{eq:Wk3}. Since $\mbf{c}$ contains an Einstein metric $\wh{g}$, we have $\wh{Y}_{a b c}=0$, so condition \eqref{eq:Feff_char_Y} holds. Finally, since for any metric $g$ in $\accentset{n.e.}{\mbf{c}}$, $\nabla_a k^a =0$ and $k^a k^b \Rho_{a b} =1$ by equation \eqref{eq:Ric00}, we have that the LHS of equation \eqref{eq:Feff_char_Rho} can be computed to be $-1$, which is indeed less than zero. This establishes \eqref{item:5}.
\end{proof}

Note that for a Fefferman--Einstein metric, $\Lambda$ is determined by $\ul{\Lambda}$. As a direct consequence of Theorem \ref{thm:main-Feff-Einstein}, we obtain the following two corollaries.

\begin{cor}\label{cor:Ricci-flat-Feff-Einstein}
	Let $(\mc{M}, \mbf{c}, K)$ be the optical geometry of Theorem \ref{thm:main-Einstein}. Then $\wh{g}$ is a Ricci-flat Fefferman--Einstein metric if and only if $\lambda_0 = 0$, i.e.\ $\Lambda = \ul{\Lambda} = \ul{c} =0$.
\end{cor}

\begin{cor}
	Let $(\mc{M}, \mbf{c}, K)$ be the optical geometry given in Theorem \ref{thm:main-Einstein}. Then, the Einstein metric $\wh{g}$ can be cast into the Kerr-Schild form
	\begin{align*}
	\wh{g} & = g_{FE} +  \sec^2 \phi \left( \lambda_0 - \frac{1}{2m+2} \ul{\Lambda} \right) \kappa \otimes \kappa \, , & \mbox{for $- \frac{\pi}{2} < \phi < \frac{\pi}{2}$,} 
	\end{align*}
	where $g_{FE}$ is the Fefferman--Einstein metric associated to the underlying almost CR--Einstein structure $(\ul{H},\ul{J},\ul{\theta}^0)$.
\end{cor}
The previous corollary comes as no surprise considering that the conformal factor involved in $\wh{g}$ (see Remark \ref{rem:circle}) and the function $\lambda_0$ depend on periodic functions of period $\pi$, which indeed suggests these Einstein metrics live on a circle bundle.

\subsection{Taub--NUT-type spacetimes: $\ul{\Lambda}\neq0$}\label{sec:Taub-NUT}
Originally discovered in dimension four by Taub \cite{Taub1951}, and independently, by Newman, Unti, Tamburino \cite{Newman1963}, the \emph{Taub--NUT spacetime} is an Einstein Lorentzian manifold of dimension $2m+2$ associated to any $2m$-dimensional K\"{a}hler--Einstein metric \cite{Bais1985,Awad2002,Alekseevsky2021}. We shall presently show that the construction can be generalised to the non-integrable case and is locally isometric to the class of metrics of Theorem \ref{thm:main-Einstein} for which $\ul{\Lambda} \neq 0$.

Let $(\ut{\mc{M}}, \ut{h}, \ut{J})$ be a $2m$-dimensional almost K\"{a}hler--Einstein manifold with \emph{non-zero} Ricci scalar $2m  \ut{\Lambda}$. By Proposition \ref{prop:aK2aCR}, the circle bundle $\ul{\mc{M}} \longrightarrow \ut{\mc{M}}$ defined by \eqref{eq:S1bdl} in Section \ref{sec:almost_Kahler} admits an almost CR--Einstein structure $(\ul{H},\ul{J},\ul{\theta}^0)$. Following our previous conventions, we write $\ul{\Lambda}$ for $\ut{\Lambda}$, and we omit pullback maps for clarity. We can then equip  the radial extension $\mc{M}_{TN} = \R \times \ul{\mc{M}}$ of $\ul{\mc{M}}$, with coordinate $r$ on $\R$, with the Lorentzian metric
\begin{align}\label{eq:T-NUT_metric}
g_{TN} & = \frac{1}{\ul{\Lambda}^2 F(r)} \dd r \otimes \dd r - 4 \, F(r) \ul{\theta}^0 \otimes \ul{\theta}^0 + \frac{r^2 + \ul{\Lambda}^2}{\ul{\Lambda}^2} \ul{h} \, ,
\end{align}
where $F(r)$ is a smooth function which satisfies
\begin{align}\label{eq:Einstein_F}
\frac{\dd}{\dd r}\left( \frac{(r^2 +  \ul{\Lambda}^2)^{m}}{r} F(r)\right)
& =  \frac{(r^2 + \ul{\Lambda}^2)^{m}}{r^2}  \ul{\Lambda} -  \frac{(r^2 +  \ul{\Lambda}^2)^{m+1}}{r^2}  \frac{\Lambda}{\ul{\Lambda}^2} \, .
\end{align}
for some constant $\Lambda$. We shall denote the constant of integration by $M$.

Up to rescaling of the constants by factors involving $\ul{\Lambda}$, the metric \eqref{eq:T-NUT_metric} corresponds to the one given in e.g.\ \cite{Awad2002}  when $(\ut{\mc{M}}, \ut{h}, \ut{J})$ is merely K\"{a}hler--Einstein, and the choice of function $F(r)$ satisfying \eqref{eq:Einstein_F} is to ensure that $g_{TN}$ is Einstein. We shall show that this also holds true when $(\ut{\mc{M}}, \ut{h}, \ut{J})$ is \emph{strictly almost K\"{a}hler}.

\begin{defn}\label{def:Taub-NUT}
	We call the Lorentzian manifold $(\mc{M}_{TN}, g_{TN})$ constructed above a \emph{Taub--NUT-type spacetime}.
\end{defn}
A Taub--NUT-type spacetime is sometimes referred to more specifically as \emph{Taub--NUT--AdS} when $\Lambda<0$, as \emph{Taub--NUT} when $\Lambda=0$, and as \emph{Taub--NUT--dS} when $\Lambda>0$. Here, AdS stands for anti-de Sitter, and dS for de Sitter. The respective physical interpretation of the constants $\Lambda$ and $M$ are that of a cosmological constant and a mass. The parameter $\ul{\Lambda}$ is referred to as the NUT (Newman--Unti--Tamburino) parameter.

Let us recast the metric \eqref{eq:T-NUT_metric} as
\begin{align*}
g_{TN} & = \frac{r^2 + \ul{\Lambda}^2}{\ul{\Lambda}^2} \left( 2 \kappa \odot \lambda + \ul{h} \right) \, ,
\end{align*}
where
\begin{align*}
\kappa & = 2 \, \ul{\theta}^0 + \frac{1}{\ul{\Lambda}F(r)} \dd r  \, , &
\lambda & = \frac{\ul{\Lambda}}{r^2+\ul{\Lambda}^2} \dd r - \frac{1}{2} \frac{\ul{\Lambda}^2}{r^2+\ul{\Lambda}^2}  F(r)\kappa \, .
\end{align*}
We note also $\lambda = \frac{1}{2} \frac{\ul{\Lambda}^2}{r^2+\ul{\Lambda}^2}  \left( - 2 F(r) \ul{\theta}^0 + \frac{1}{\ul{\Lambda}} \dd r \right)$. We now perform the change of variables
\begin{align*}
r = \ul{\Lambda} \tan \phi \, , & & \mbox{i.e.} & & \phi  = \arctan  \frac{r}{\ul{\Lambda}} \, .
\end{align*}
We note that $r \rightarrow \pm \infty$ as $\phi \rightarrow \pm \frac{\pi}{2}$. This defines a diffeomorphism for $\phi \in \left( - \frac{\pi}{2}, \frac{\pi}{2} \right)$. Now
\begin{align*}
r^2 + \ul{\Lambda}^2 & = \ul{\Lambda}^2 \sec^2 \phi \, , &
\frac{\dd}{\dd \phi} & = \frac{r^2 + \ul{\Lambda}^2}{\ul{\Lambda}} \frac{\dd}{\dd r} \, , &
\dd r & = \ul{\Lambda} \sec^2 \phi \, \dd \phi \, .
\end{align*}
Hence, defining
\begin{align*}
F(r) = - \frac{r^2 + \ul{\Lambda}^2}{\ul{\Lambda}^2} \lambda_0 \, , & & \mbox{i.e.} & &
\lambda_0 = - \frac{\ul{\Lambda}^2}{r^2 + \ul{\Lambda}^2} F(r) \, , 
\end{align*}
we have that $\kappa = 2 \, \ul{\theta}^0 - \frac{1}{\lambda_0} \dd \phi$. The last term is clearly closed, and thus locally exact, i.e.\ $\kappa$ is gauge equivalent to $2 \, \ul{\theta}^0$. We can thus assume that $\kappa = 2 \, \ul{\theta}^0$ by means of a local diffeomorphism. It can then be checked that $\lambda = \dd \phi + \lambda_0 \ul{\theta}^0$ as in Theorem \ref{thm:main-Einstein}. In addition, we check that $F(r)$ satisfies \eqref{eq:Einstein_F} if and only if $\lambda_0$ satisfies
\begin{align*}
\frac{\dd}{\dd \phi} \left( \frac{\sec^{2m+2} \phi}{\tan \phi} \lambda_0 \right)
& =  \frac{\sec^{2m+4} \phi }{\tan^2 \phi}  \Lambda - \frac{\sec^{2m+2} \phi}{\tan^2 \phi} \ul{\Lambda}  \, ,
\end{align*}
which is none other than equation \eqref{eq4}. To ensure consistency of these equations, we must have $M = - \ul{\Lambda}^{2m-1} \ul{c}$. This proves the following result.

\begin{thm}\label{thm:main-Taub_NUT-(A)dS}
	Let $(\mc{M}, \mbf{c}, K)$ be the optical geometry of Theorem \ref{thm:main-Einstein}. Assume $\ul{\Lambda}$ to be non-zero. Then the Einstein manifold $(\mc{M},\wh{g})$ is locally isometric to the Taub--NUT-type spacetime $(\mc{M}_{TN}, g_{TN})$. In particular, $g_{TN}$ is Einstein with Ricci scalar $2(m+1) \Lambda$.
\end{thm}

\section{Further geometric properties}\label{sec:further_prop}
Viewed as a $G$-structure, the almost Robinson structure $(N,K)$ of Theorem \ref{thm:main-Einstein} can easily be described in terms of its intrinsic torsion as can be gleaned from equations \eqref{eq:Dform_E}. We shall not pursue the matter here, which is dealt with in \cite{Fino2021}. Instead, we focus on the existence of additional geometric structures arising from the Einstein condition.

\subsection{A distinguished conformal Killing field}
\begin{prop}\label{prop:Killing}
	Let $(\mc{M}, \mbf{c}, K)$ be the optical geometry of Theorem \ref{thm:main-Einstein}. Then the vector field $v = g(\alpha, \cdot)$, where $\alpha := \lambda + \frac{1}{2} \lambda_0 \kappa$, is a Killling field for the metric $g$. Further, $v$ descends to the infinitesimal transverse symmetry $\ul{e}_0$ of the almost CR--Einstein structure $(\ul{H},\ul{J},\ul{\theta}^0)$ on the (local) leaf space $\ul{\mc{M}}$ of $\mc{K}$.
	
	This conformal Killing field is null if and only if $\wh{g}$ is a Ricci-flat Fefferman--Einstein metric.
\end{prop}

\begin{proof}
	Using equations \eqref{eq:kappa} and \eqref{eq:lambda} together with the Leibniz rule, we find
	\begin{align*}
	\nabla \alpha & = 2 \ii \lambda_0 \ul{h}_{\alpha \bar{\beta}} \theta^{\alpha} \wedge \overline{\theta}{}^{\bar{\beta}} -  \dot{\lambda}_0 \kappa \wedge \lambda  \, ,
	\end{align*}
	i.e.\ $\nabla \alpha = \dd \alpha$, which shows that $v$ is a Killing field for $g$ as claimed. Further, since $\mathsterling_k v = 0$, the vector field $v$ projects down to a vector field $\ul{v}$ on $(\ul{\mc{M}}, \ul{H}, \ul{J})$. It is then clear that $\ul{v}$ is the Reeb vector field $\ul{e}_0$ of the contact form $\ul{\theta}^0$. Thus, by definition of almost CR--Einstein, it must be an infinitesimal symmetry of $(\ul{H}, \ul{J})$. Finally, we note that $g (v, v) = \lambda_0$, and the last claim follows from Corollary \ref{cor:Ricci-flat-Feff-Einstein}.
\end{proof}

\subsection{A second almost Robinson structure}
Under the assumption of Theorem \ref{thm:main-Einstein}, the existence of the Einstein metric $\wh{g}$ is essentially equivalent to the existence of the distinguished $1$-form $\lambda$. This not only determines a second optical structure $L$, where $L^\perp = \Ann(\lambda)$, but also a second almost Robinson structure which we now describe, carefully distinguishing between the cases $\lambda_0 \neq 0$ and $\lambda_0 =0$. When $\lambda_0 \neq 0$, it will be convenient to apply a boost transformation to $\kappa$ and $\lambda$ to obtain
\begin{align}\label{eq:new_lam_kap}
\lambda' & := 2 \frac{1}{\lambda_0} \lambda \, , & \kappa' & := \frac{1}{2} \lambda_0 \kappa \, \, .
\end{align}
These $1$-forms satisfy
\begin{subequations}\label{eq:new_Dlam_Dkap}
	\begin{align}
	\nabla \lambda' & =  2   \ii   \ul{h}_{\alpha \bar{\beta}} \theta{}^{\alpha} \wedge \overline{\theta}{}^{\bar{\beta}}  - \frac{1}{2}  \dot{\lambda}_0   \lambda' \otimes \lambda' \, , \label{eq:new_Dlam} \\
	\nabla \kappa' & = \ii \lambda_0 \ul{h}_{\alpha \bar{\beta}} \theta^{\alpha} \wedge \overline{\theta}{}^{\bar{\beta}} + \frac{1}{2} \dot{\lambda}_0 \lambda' \otimes \kappa' \, ,\label{eq:new_Dkap}
	\end{align}
\end{subequations}
respectively.

\begin{prop}\label{prop:aRob2}
	Let $(\mc{M}, \mbf{c}, N, K)$ be the twist-induced almost Robinson manifold of Theorem \ref{thm:main-Einstein}. Then the null $1$-form $\lambda$ defines an optical structure $L$ with non-shearing congruence of null geodesics $\mc{L}$, non-expanding with respect to $g$, and an almost Robinson structure $(N^*,L)$ dual to $(N,K)$, where
	\begin{align*}
	N^* & \cong  \left(\overline{N} \cap {}^\C H_{K,L} \right) \oplus {}^\C L \, , & H_{K,L}  & = K^\perp \cap L^\perp \, .
	\end{align*}
	These enjoy the following properties. 
	\begin{enumerate}
		\item If $\wh{g}$ is not a Ricci-flat Fefferman--Einstein metric, i.e.\ $\lambda_0 \neq 0$, then $\mc{L}$ is twisting and $(N^*, L)$ is twist-induced. Further, $(N^*,L)$, with $\lambda'$ defined in \eqref{eq:new_lam_kap}, induces on the (local) leaf space of $\mc{L}$ an almost CR--Einstein structure equivalent to $(\ul{H},\ul{J},\ul{\theta}^0)$.
		\item If $\wh{g}$ is a Ricci-flat Fefferman--Einstein metric, i.e.\ $\lambda_0 = 0$, then $\lambda$ is a null $1$-form parallel with respect to the Levi-Civita connection $\nabla$ of $g$, and in particular, $\mc{L}$ is non-twisting. Further, $(N^*,L)$ induces a local foliation by $2m$-dimensional Ricci-flat almost K\"{a}hler--Einstein manifolds on the (local) leaf space of $\mc{L}$.
	\end{enumerate}
\end{prop}

\begin{proof}
	The existence of the dual almost Robinson structure is self-explanatory, and that $\mc{L}$ is a non-shearing congruence of null geodesics, non-expanding with respect to $g$, non-twisting when  $\lambda_0=0$, twisting otherwise, follows directly from equation \eqref{eq:lambda}.
	\begin{enumerate}
		\item When $\lambda_0\neq 0$, we can work with an adapted coframe $\{ \kappa' , \theta^\alpha , \overline{\theta}{}^{\bar{\alpha}} , \lambda' \}$ where $\kappa'$ and $\lambda'$ are defined by \eqref{eq:new_lam_kap}. By inspection of \eqref{eq:new_Dlam_Dkap}, we note that the sets of $1$-forms $\{ \lambda', \theta^{\alpha} \}$ and $\{ \kappa, \theta^{\alpha} \}$ are on the same geometric footing. Both sets can be viewed as the pullbacks of adapted CR coframes on the respective (local) leaf spaces of $\mc{L}$ and $\mc{K}$. There, their respective structure equations are clearly determined by the same almost pseudo-Hermitian invariants. In particular, their respective pseudo-Hermitian structures must be equivalent.
		\item When $\lambda_0=0$, the existence of a foliation $\ul{\mc{H}}_{\mc{L}}$, say, by almost K\"{a}hler manifolds on the leaf space $\ul{\mc{M}}_{\mc{L}}$ of $\mc{L}$ follows from the analysis given in \cite{Fino2020,Fino2021}. Note that the tensor $h$ in Theorem \ref{thm:main-Einstein} can also be viewed as the pullback of a tensor field on $\ul{\mc{M}}_{\mc{L}}$, which restricts to an almost K\"{a}hler metric on each leaf of $\ul{\mc{H}}_{\mc{L}}$. That these metrics are Ricci-flat can be checked by restricting the Ricci tensor of $g$ to projectable vector fields tangent to $H_{K,L}$ using the computation of Appendix \ref{app:computations}.
	\end{enumerate}
	
\end{proof}

\begin{rem}
	By Proposition \ref{prop:conf_flat_Rob}, if $(\mc{M}, \mbf{c}, K)$ of Theorem \ref{thm:main-Einstein} is conformally flat, then $\wh{g}$ is a Ricci-flat Fefferman--Einstein metric so that the null $1$-form $\lambda$ defines a non-twisting non-shearing congruence of null geodesics.
\end{rem}

\begin{rem}
	When $\lambda_0 \neq 0$, the situation of Proposition \ref{prop:Killing} with respect to the almost Robinson structure $(N^*,L)$ is completely analogous, and one can check that $\mathsterling_{\ell'} v = 0$ where $\ell'=g^{-1}(\lambda',\cdot)$ with $\lambda'$ given by \eqref{eq:new_lam_kap}, i.e.\ $v$ descends to the transverse infinitesimal symmetry  of the almost CR--Einstein induced by $(N^*,L)$ on the leaf space $\ul{\mc{M}}_{\mc{L}}$ of $\mc{L}$. In fact, we can rewrite  $\alpha= \kappa' + \frac{1}{2} \lambda_0 \lambda'$.
	
	On the other hand, when $\lambda_0=0$, $v = \ell = g^{-1}(\lambda,\cdot)$ is tangent to $\mc{L}$ and thus projects down to the zero vector field on $\ul{\mc{M}}_{\mc{L}}$. However, $k$ now commutes with $\ell$ and descends to a transverse infinitesimal symmetry of the almost K\"{a}hler--Einstein foliation described in Proposition \ref{prop:aRob2}. One can check that the quotient by this symmetry yields an almost K\"{a}hler--Einstein manifold.
\end{rem}

\begin{rem}
	Note that the distribution $K + L$ is involutive, and thus $(\mc{M},\mbf{c},K)$ is foliated by two-dimensional leaves. It is straightforward to check that one recovers an almost K\"{a}hler--Einstein manifold on the leaf space of this foliation. This was already noted in \cite{Leitner2007} in the integrable case, where the Fefferman--Einstein space is described as a $2$-torus fibration over a K\"{a}hler--Einstein manifold.
\end{rem}

\section{Other signatures}\label{sec:diff_sign}
The results presented in this article can easily be adapted to conformal structures of signature $(p+1, q+1)$ with both $p$ and $q$ even: here, the notion of almost Robinson structure is well-defined since for these signatures, the existence of an almost null structure of real index one is possible -- see \cite{Kopczy'nski1992}. The only point of caution is that since the screen bundle conformal structure is no longer positive definite, the norm of the twist of a congruence of null geodesics may be non-positive, in which case it cannot define an almost Robinson structure. In split signature $(m+1,m+1)$, one can define a totally real analogue of a twist-induced almost Robinson structure if the norm of the twist is negative. When both $p$ and $q$ are both greater than one, the twist endomorphism can be nilpotent, a property that allows yet another type of geometric structure. All these cases will be treated elsewhere.
If the twist of a non-shearing congruence of null geodesics does induce an almost Robinson structure, the leaf space of the congruence inherits a partially integrable contact almost CR structure, this time, of signature $(p,q)$.

In odd dimensions, the geometry of non-shearing congruences of null geodesics is much more restrictive. The first point to note is that the leaf space is even-dimensional, which prohibits the existence of a contact structure there. Second, the algebraic constraint \eqref{eq-F^2} forces the twist to be either identically zero or null, which restricts the possible metric signatures. Hence, in Lorentzian signature, odd dimensions, a non-shearing geodetic congruence with Weyl curvature prescription \eqref{eq:aligned}
must be non-twisting \cite{Ortaggio2007}. This is also true when $p<2$ or $q<2$. For other signatures, the situation is a little bit more delicate -- again, this will be treated elsewhere.

\appendix

\section{Computation of the curvature}\label{app:computations}
In this section, we compute the curvature tensors of the metric $g$ given in Section \ref{sec-Robinson}. The use of the open-source software \texttt{cadabra} \cite{Peeters2007a,Peeters2018} was particularly helpful for that purpose. Not all components of the curvature tensors are given, but all may be obtained by means of complex conjugation and index manipulation. Parts of the first Bianchi identities are also given to bring out other forms of the components -- the remaining, `purely CR', parts being given by equations \eqref{eq:Bianchi1_CR}. We omit pullback maps for clarity.

\subsection{Riemann tensor}
\begin{fleqn}
	\begin{align}
	R_{\gamma}{}^{0 0}\,_{\alpha} & = 0 \, , \label{eq:R-Nsh}\\
	R_{\bar{\gamma}}{}^{0 0}\,_{\alpha} & = - \ul{h}_{\alpha \bar{\gamma}} \, , \label{eq:R-Nsh2}
	\end{align}
	\begin{align}
	R_{0}{}^{0 0}\,_{\alpha} & =     \dot{E}_{\alpha}    - \ii E_{\alpha}  \, , \label{eq:R-diffE}\\ 
	R_{\beta \delta}{}^{0}\,_{\alpha} & =     -  2 \ii \ul{\Nh}_{\beta \delta \alpha}  \, ,  \label{eq:R-T} \\
	R_{\beta \bar{\gamma}}{}^{0}\,_{\alpha} & = -  2 \ii E_{\alpha} \ul{h}_{\beta \bar{\gamma}} - \ii E_{\beta} \ul{h}_{\alpha \bar{\gamma}}  \label{eq:R-E} \, ,
	\end{align}
	\begin{align}
	R_{\gamma 0}{}^{0}\,_{\alpha} & =  - \ul{\nabla}_{\gamma} E_{\alpha} + \lambda_{\gamma} \dot{E}_{\alpha} +  E_{\alpha} E_{\gamma}    +E^{\beta} \ul{\Nh}_{\beta \alpha \gamma}      - \frac{1}{2} \ii \ul{A}_{\alpha \gamma}       - \ii B_{\gamma \alpha}  \, , \\
	R{}_{\bar{\beta} 0}{}^{0}{}_{\alpha} & = - \ul{\nabla}_{\bar{\beta}} E_{\alpha} + \lambda_{\bar{\beta}} \dot{E}_{\alpha} +  E_{\alpha} E_{\bar{\beta}}  + \ii E_{0} \ul{h}_{\alpha \bar{\beta}}   - \ii B_{\alpha \bar{\beta}} \, , \\
	R_{0}{}^{0 0}\,_{0} & = \dot{E}_{0}  -2  E_{\alpha} E^{\alpha}     \, , \\ 
	R_{\gamma \delta \alpha \beta} & = 2   \ul{\nabla}_{[\gamma|} \ul{\Nh}_{\alpha \beta |\delta]} \, ,  \label{eq:R-DT} \\
	R_{\gamma \bar{\delta} \alpha \beta} & =   2  \ii B_{\alpha \beta} \ul{h}_{\gamma \bar{\delta}} - 2 \ii B_{\gamma [\alpha} \ul{h}_{\beta] \bar{\delta}}  -  \ul{\nabla}_{\bar{\delta}} \ul{\Nh}_{\alpha \beta \gamma} + \ii \ul{A}_{\gamma [\alpha} \ul{h}_{\beta] \bar{\delta}}  \, ,  \label{eq:R-DbT} \\
	R_{\gamma \bar{\delta} \bar{\beta} \alpha} & =  \ul{R}_{\gamma \bar{\delta} \alpha \bar{\beta}}  -  2  \ii B_{\alpha \bar{\beta}} \ul{h}_{\gamma \bar{\delta}}  
	- 2 \ii B_{\gamma \bar{\delta}} \ul{h}_{\alpha \bar{\beta}}  
	- \ii B_{\gamma \bar{\beta} } \ul{h}_{\alpha \bar{\delta}}  
	- \ii B_{\alpha \bar{\delta}} \ul{h}_{\gamma \bar{\beta}} 
	+  \ul{\Nh}_{\epsilon \alpha \gamma}  \ul{\Nh}{}^{\epsilon}{}_{\bar{\beta} \bar{\delta}} \, ,  \label{eq:R-ST}
	\end{align}
	\begin{align}
	R_{0}{}^{0}{}_{\alpha 0} & = \frac{1}{2} \ul{\nabla}_{0}  E_{\alpha}  - \frac{1}{2}\lambda_{0} \dot{E}_{\alpha}+ E^{\beta} B_{\alpha \beta}  +E_{\beta} B_{\alpha}\,^{\beta} + \ii C_{\alpha} - \dot{C}_{\alpha} \, , \\
	R_{\gamma 0 \alpha \beta} & =  \ul{\nabla}_{\gamma} B_{\alpha \beta} - \lambda_{\gamma} \dot{B}_{\alpha \beta}  + 2  B_{[\alpha}\,^{\delta} \ul{\Nh}_{\beta] \delta \gamma}  +E_{[\alpha} \ul{A}_{\beta] \gamma} + 2  E_{[\alpha} B_{\beta] \gamma}   - \frac{1}{2}  \ul{\nabla}_{0} \ul{\Nh}_{\alpha \beta \gamma}    \, , \\
	R_{\beta \bar{\gamma} \alpha 0} & =  - \ul{\nabla}_{\bar{\gamma}} B_{\alpha \beta} + \lambda_{\bar{\gamma}} \dot{B}_{\alpha \beta} + \ul{\nabla}_{\beta} B_{\alpha \bar{\gamma}} - \lambda_{\beta} \dot{B}_{\alpha \bar{\gamma}}   -2 E_{\alpha} B_{\beta \bar{\gamma}} +E_{\beta} B_{\alpha \bar{\gamma}}   -E_{\bar{\gamma}}B_{\alpha \beta}  \nonumber  \\
	& \qquad  +  2  \ii C_{\alpha} \ul{h}_{\beta \bar{\gamma}}  - \frac{1}{2}E_{\bar{\gamma}} \ul{A}_{\alpha \beta} + B_{\bar{\gamma}}{}^{\delta} \ul{\Nh}_{\delta \alpha \beta}    - \frac{1}{2} \ul{\nabla}_{\bar{\gamma}} \ul{A}_{\alpha \beta} - \frac{1}{2}\ul{A}_{\bar{\gamma}}{}^{\delta} \ul{\Nh}_{\delta \alpha \beta}  \, ,  
	\end{align}
	\begin{align}
	R_{\beta 0 \alpha 0} & =  - \frac{1}{2} \ul{\nabla}_{0} B_{\alpha \beta} + \frac{1}{2} \lambda_{0} \dot{B}_{\alpha \beta} + \ul{\nabla}_{\beta}  C_{\alpha}  - \lambda_{\beta} \dot{C}_{\alpha}  -B_{\alpha}\,^{\gamma} \ul{A}_{\beta \gamma}    -E_{\alpha} C_{\beta}  +B_{\alpha \gamma} B_{\beta}\,^{\gamma}  \nonumber\\
	& \qquad - C^{\gamma} \ul{\Nh}_{\gamma \alpha \beta}  + B_{\beta \gamma} B_{\alpha}\,^{\gamma} +E_{\beta} C_{\alpha}  - \frac{1}{2}E_{0} \ul{A}_{\alpha \beta}   -E_{0} B_{\alpha \beta}  - \frac{1}{4} \ul{\nabla}_{0} \ul{A}_{\alpha \beta}  \, , \\
	R_{\bar{\beta} 0 \alpha 0}& = - \frac{1}{2} \ul{\nabla}_{0} B_{\alpha \bar{\beta}} + \frac{1}{2} \lambda_{0} \dot{B}_{\alpha \bar{\beta}} + \ul{\nabla}_{\bar{\beta}}  C_{\alpha}  - \lambda_{\bar{\beta}} \dot{C}_{\alpha}   -B_{\alpha \gamma} \ul{A}_{\bar{\beta}}{}^{\gamma} \nonumber \\
	& \qquad -E_{\alpha} C^{\beta} + B_{\alpha \gamma} B_{\bar{\beta} }{}^{\gamma}   -B_{\alpha}\,^{\gamma} B_{\gamma \bar{\beta}} +E_{\bar{\beta}} C_{\alpha} -E_{0} B_{\alpha \bar{\beta}}    - \frac{1}{4}\ul{A}_{\alpha \gamma} \ul{A}_{\bar{\beta} }{}^{\gamma}   \, .
	\end{align}
\end{fleqn}

\subsection{The Ricci tensor}
\begin{fleqn}
	\begin{align}
	\Ric^{0 0} & = 2m \, ,  \label{eq:Ric00}  \\
	\Ric_{\alpha}{}^{0} & = \dot{E}\,_{\alpha}  -4 \ii E_{\alpha}\, ,  \label{eq:Ric-diffE}
	\end{align}
	\begin{align}
	\Ric_{\alpha \beta} & =2 \ul{\nabla}_{(\alpha}{E_{\beta)}} - 2 \lambda_{(\alpha} \dot{E}\,_{\beta)} - 2  E_{\alpha} E_{\beta}
	+ \ii m \ul{A}_{\alpha \beta} + 2 \ul{\nabla}^{\gamma}{\ul{\Nh}_{\gamma (\alpha \beta)}}   -  2 \ul{\Nh}_{\gamma (\alpha \beta)}E^{\gamma}   \, , \\
	\Ric_{\alpha \bar{\beta}} &  =   \ul{\nabla}_{\alpha}{E_{\bar{\beta}}} + \ul{\nabla}_{\bar{\beta}}{E_{\alpha}}   -  \lambda_{\alpha} \dot{E}_{\bar{\beta}} - \lambda_{\bar{\beta}} \dot{E}_{\alpha}
	-2E_{\alpha} E_{\bar{\beta}} - 4 \ii B_{\alpha \bar{\beta}} + \ul{\Ric}_{\alpha \bar{\beta}} - \ul{\Nh}_{\alpha \delta \gamma} \ul{\Nh}_{\bar{\beta}}{}^{\delta \gamma} 
	\, , \\
	\Ric_{0}{}^{0} & =  \ul{\nabla}_{\alpha}{E^{\alpha}} + \ul{\nabla}^{\alpha}{E_{\alpha}}   -  \lambda_{\alpha} \dot{E}^{\alpha}  - \lambda^{\alpha} \dot{E}_{\alpha} - 4E_{\alpha} E^{\alpha}+2 \ii B_{\alpha}\,^{\alpha}+\dot{E}\,_{0} \, ,
	\end{align}
	\begin{align}
	\Ric_{0 \beta} & =   \ul{\nabla}_{\beta}{E_{0}}  
	- \frac{1}{2}\ul{\nabla}_{0}{E_{\beta}} - \lambda_{\beta} \dot{E}_{0}  + \frac{1}{2}\lambda_{0} \dot{E}\,_{\beta}   -\frac{1}{2}E^{\alpha} \ul{A}_{\beta \alpha}  - 2 E^{\alpha} B_{\alpha \beta}  + 2 E_{\alpha} B_{\beta}\,^{\alpha} \nonumber\\
	& \qquad \qquad -\ii  C_{\beta} +  \ul{\nabla}^{\alpha}{B_{\alpha \beta }} - \lambda^{\alpha} \dot{B}\,_{\alpha \beta}  -  \ul{\nabla}_{\alpha}{B_{\beta}\,^{\alpha}} + \lambda_{\alpha}\dot{B}_{\beta}\,^{\alpha}   \nonumber \\
	& \qquad \qquad \qquad \qquad \qquad 
	+\frac{1}{2} B^{\alpha \gamma} \ul{\Nh}_{\alpha \gamma \beta}  + \frac{1}{2} \ul{\nabla}^{\alpha}{\ul{A}_{\beta \alpha}} +\frac{1}{2}\ul{A}^{\alpha \gamma} \ul{\Nh}_{\beta \alpha \gamma} \, , \\
	\Ric_{0 0} & = 
	\ul{\nabla}_{\alpha} {C^{\alpha}} 
	+ \ul{\nabla}^{\alpha}{C_{\alpha}} - \lambda_{\alpha} \dot{C}^{\alpha} -  \lambda^{\alpha}\dot{C}_{\alpha} 
	- \frac{1}{2} \ul{A}_{\alpha \beta} \ul{A}^{\alpha \beta}+2B_{\alpha \beta} B^{\alpha \beta}-2B_{\alpha}\,^{\beta} B_{\beta}\,^{\alpha}\, .
	\end{align}
\end{fleqn}

\subsection{The Ricci scalar}
\begin{fleqn}
	\begin{multline}
	\Sc =  4 \ul{\nabla}_{\alpha}{E^{\alpha}} +4 \ul{\nabla}^{\alpha}{E_{\alpha}}  - 4 \lambda_{\alpha} \dot{E}^{\alpha} -4\lambda^{\alpha} \dot{E}\,_{\alpha}-12E_{\alpha} E^{\alpha} \\
	+2\dot{E}\,_{0}  -4 \ii B_{\alpha}\,^{\alpha} 
	+2 \ul{\Sc}  -2 \ul{\Nh}_{\alpha \beta \gamma} \ul{\Nh}^{\alpha \beta \gamma} \, .
	\end{multline}
\end{fleqn}

\subsection{The first Bianchi identities}
\begin{fleqn}
	\begin{align}
	0 & =  \dot{B}_{\alpha \beta} +  2 \ul{\nabla}_{[\alpha} {E_{\beta]}} - 2 \lambda_{[\alpha} \dot{E}_{\beta]}
	- \ul{\Nh}_{\alpha \beta \gamma} E^{\gamma} \, , \\
	0 & =   \dot{B}_{\alpha \bar{\beta}} +  \ul{\nabla}_{\alpha}{E_{\bar{\beta}}} - \ul{\nabla}_{\bar{\beta}}{E_{\alpha}}   -  \lambda_{\alpha} \dot{E}_{\bar{\beta}} + \lambda_{\bar{\beta}} \dot{E}_{\alpha}
	+  \ii  E_{0} \ul{h}_{\alpha \bar{\beta}}  \, , \\
	0 & =  \dot{C}_{\alpha}  + \ul{\nabla}_{\alpha}{E_{0}} - \ul{\nabla}_{0}{E_{\alpha}} - \lambda_{\alpha} \dot{E}_{0}  + \lambda_{0}\dot{E}_{\alpha}
	- \ul{A}_{\beta \alpha} E^{\beta} \, , \\
	0 & =   \ul{\nabla}_{[\alpha}{B_{\beta \gamma]}} - \lambda_{[\alpha} \dot{B}_{\beta \gamma]} +  B_{[\alpha}\,^{\delta} \ul{\Nh}_{\beta \gamma] \delta}
	+2E_{[\alpha} B_{\beta \gamma]}  \, , \\
	0 & =    2 \ul{\nabla}_{[\alpha}{B_{\beta] \bar{\gamma}}} +\ul{\nabla}_{\bar{\gamma}}{B_{\alpha \beta}} - 2 \lambda_{[\alpha}\dot{B}_{\beta] \bar{\gamma}} 
	- \lambda_{\bar{\gamma}} \dot{B}_{\alpha \beta} \nonumber \\
	& \qquad \qquad +  4E_{[\alpha} B_{\beta] \bar{\gamma}} + 2E_{\bar{\gamma}} B_{\alpha \beta}  -2 \ii C_{[\alpha} \ul{h}_{\beta] \bar{\gamma}}  
	+ \ul{\Nh}_{\alpha \beta \delta}  B_{\bar{\gamma}}{}^{\delta} \, , \\
	0 & =      \ul{\nabla}_{[\alpha}{C_{\beta]}} - \lambda_{[\alpha} \dot{C}_{\beta]}
	+ \frac{1}{2} \ul{\nabla}_{0}{B_{\alpha \beta}} - \frac{1}{2} \lambda_{0} \dot{B}_{\alpha \beta}  \nonumber \\
	& \qquad \qquad + 2E_{[\alpha} C_{\beta]}   +E_{0} B_{\alpha \beta} + B_{[\alpha}\,^{\gamma} \ul{A}_{\beta] \gamma}  -\frac{1}{2} \ul{\Nh}_{\alpha \beta \gamma} C^{\gamma}  \, , \\
	0 & = 
	\ul{\nabla}_{\alpha} {C_{\bar{\beta}}} 
	- \ul{\nabla}_{\bar{\beta}}{C_{\alpha}} - \lambda_{\alpha} \dot{C}_{\bar{\beta}} +  \lambda_{\bar{\beta}}\dot{C}_{\alpha}
	+ \ul{\nabla}_{0}{B_{\alpha \bar{\beta}}} - \lambda_{0}\dot{B}_{\alpha \bar{\beta}}  \nonumber \\
	& \qquad \qquad     +2E_{\alpha} C_{\bar{\beta}}    -2E_{\bar{\beta}} C_{\alpha}  +2E_{0} B_{\alpha \bar{\beta}} - B_{\bar{\beta}}{}^{\gamma} \ul{A}_{\gamma \alpha}   -  \ul{A}_{\bar{\beta}}{}^{\gamma}  B_{ \gamma \alpha} \, .
	\end{align}
\end{fleqn}

\bibliographystyle{abbrv}
\bibliography{biblio}


\end{document}

%% file: macro.tex
\newtheorem{thm}{Theorem}[section]
\newtheorem{cor}{Corollary}[section]
\newtheorem{lem}{Lemma}[section]
\newtheorem{prop}{Proposition}[section]
\theorembodyfont{\rmfamily}
\newtheorem{defn}{Definition}[section]

\newtheorem{rem}{Remark}[section]
\numberwithin{equation}{section}
\newenvironment{proof}{\noindent \emph{Proof.}}{\hspace{\stretch{1}}$\Box$}

\newcommand{\parderv}[2] {\frac{\partial#1}{\partial#2}}

\newcommand{\mbf} [1]{\mathbf{#1}}

\newcommand{\mc} [1]{\mathcal{#1}}

\newcommand{\wh} [1]{\widehat{#1}}

\newcommand{\ul} [1]{\underline{#1}}

\newcommand{\dd} {\mathrm{d}}

\newcommand{\ii} {\mathrm{i}}
\newcommand{\ee} {\mathrm{e}}

\newcommand{\Rho} {\mbox{\textsf{P}}}
\newcommand{\Nh} {\mbox{\textsf{N}}}

\newcommand{\Ric} {\mathrm{Ric}}
\newcommand{\Sc} {\mathrm{Sc}}

\newcommand{\Ann} {\mathrm{Ann}}

\newcommand{\Id} {\mathrm{Id}}

\newcommand{\R} {\mathbf{R}}
\newcommand{\C} {\mathbf{C}}
\newcommand{\Z} {\mathbf{Z}}

\newcounter{mnotecount}[section]
\renewcommand{\themnotecount}{\thesection.\arabic{mnotecount}}

\newcommand{\mnote}[1]
{\protect{\stepcounter{mnotecount}}$^{\mbox{\footnotesize
$
\bullet$\themnotecount}}$ \marginpar{\color{red}
\raggedright\tiny\em
$\!\!\!\!\!\!\,\bullet$\themnotecount: #1} }